\documentclass[reqno,oneside,10pt]{amsart}

\usepackage[utf8]{inputenc}
\usepackage[T1]{fontenc}

\usepackage{amssymb,mathptmx,eucal,array,enumitem,setspace,geometry,color,multirow}
\usepackage{mathtools} % needed for cases*, dcases, bsmallmatrix, etc...
\usepackage[noadjust]{cite}

\usepackage{tikz}
\usetikzlibrary{decorations.pathmorphing,cd} %cd for nice commutative diagrams
\tikzset{>=to} % Sets arrows so that they look like \rightarrow etc...
\tikzset{squig/.style={->, line join=round, decorate,
                       decoration={zigzag, segment length=4, amplitude=.9, post=lineto, post length=2pt}}} % like \rightsquigarrow
                       
\usepackage{mathrsfs}

\definecolor{rouge}{rgb}{0.95,0.1,0.15}
\definecolor{vertforet}{rgb}{0.2,0.64,0.2}
\definecolor{bleu}{rgb}{0.1,0.1,0.9}

\usepackage[pdfpagemode=UseNone, pdfstartview={XYZ null null null}]{hyperref}
\hypersetup{
colorlinks=true,
citecolor=red,
linkcolor=blue,
urlcolor=blue
}

\usepackage[capitalise,noabbrev]{cleveref} % now this is a package to be proud of!

\DeclareSymbolFont{largesymbols}{OMX}{zplm}{m}{n} %Replaces summation symbol in times by the palatino one...

\numberwithin{equation}{section}

\newcolumntype{C}{>{$}c<{$}}              % Defines math mode in tabular (array package)...
\newcolumntype{D}{>{$\displaystyle}l<{$}} % Defines math mode with display big-operators and left-justification
\allowdisplaybreaks

\newcommand{\qbin}[3]{\genfrac{[}{]}{0pt}{}{#1}{#2}_{#3}}

\newcommand{\tl}[1]{\mathsf{TL}_{#1}}
\newcommand{\atl}[1]{\mathsf{a}\hskip-1.8pt\mathsf{TL}_{#1}}
\newcommand{\tlb}[1]{\mathsf{TLb}_{#1}}
\newcommand{\uqsl}{U_q(\mathfrak{sl}_2)}
\newcommand{\luszt}{\mathcal LU_q}
\newcommand{\lusztpm}{\mathcal LU_{q^{\pm 1}}}
\newcommand{\lusztI}{\mathcal LU_{q^{-1}}}

\newcommand{\Cell}[1]{\mathsf{W}_{#1}} % cellulaire
\newcommand{\Irre}[1]{\mathsf{L}_{#1}} % simple
\newcommand{\pg}[1]{\mathsf{GP}_{#1}}  % generic part
\newcommand{\XXZ}[2]{\mathsf{X}_{#1}^{#2}} % the aTL_n-representation XXZ

\newcommand*\Eval[3]{\left.#1\right\rvert_{#2}^{#3}}

\newcommand{\Bnu}{\mathfrak{B}} % base generique
\newcommand{\BCell}[1]{\mathfrak{B}^{\Cell{}}_{#1}} % base des cellulaires
\newcommand{\BXXZ}[1]{\mathfrak{B}^{\XXZ{}{}}_{#1}}   % base des sous-espaces propres de S^z

\newcommand{\ii}{F} % les f de Th\'eo
\newcommand{\jj}{E} % les g de Th\'eo
\newcommand{\kk}{\mu} % les f de Th\'eo dans les cas q^d=z^\pm 2
\newcommand{\mm}{\nu} % les g de Th\'eo dans les cas q^d=z^\pm 2
\newcommand{\ay}[1]{i_{#1}} % le I_N,d d'Alexis et Théo
\newcommand{\gl}[1]{\theta_{#1}} % l'inclusion de Graham-Lehrer

\newcommand{\modY}{\textup{\,mod\,}} % Yvan's version

\DeclareMathOperator{\im}{im\,}
\DeclareMathOperator{\rad}{rad\,}
\DeclareMathOperator{\soc}{soc\,}
\DeclareMathOperator{\head}{top\,}

\DeclareMathOperator{\vspn}{span}

\DeclareMathOperator{\Hom}{Hom}
\DeclareMathOperator{\End}{End}
\DeclareMathOperator{\Ext}{Ext}

\DeclareMathOperator{\id}{\text{id}}  % identite dans l'algebre affine aTL
\newcommand{\myoplus}[2]{\underset{#1}{\overset{#2}{\oplus{}{\hskip+0pt}'}}}\newcommand{\mybigoplus}[2]{\underset{#1}{\overset{#2}{\bigoplus{}{\hskip-1pt}'}}}

\theoremstyle{plain}
\newtheorem{theorem}{Theorem}[section]
\newtheorem{lemma}[theorem]{Lemma}
\newtheorem{proposition}[theorem]{Proposition}
\newtheorem{corollary}[theorem]{Corollary}
\newtheorem{definition}[theorem]{Definition}
\newtheorem{rem}[theorem]{Remark}

\begin{document}

\title[Spin chains as modules over $\atl N$]{Spin chains as modules over \\ the affine Temperley-Lieb algebra}

\author[T Pinet]{Th\'eo Pinet}

\address[Th\'eo Pinet]{{\rm The work was done at }D\'{e}partement de math\'ematiques et statistique \\
Universit\'{e} de Mont\-r\'{e}al \\
Qu\'{e}bec, Canada, H3C~3J7. {\rm Now also at } Universit\'e de Paris and Sorbonne Universit\'e, CNRS, IMJ-PRG, F-75006 Paris, France.
}

\email{theo.pinet@umontreal.ca}

\author[Y Saint-Aubin]{Yvan Saint-Aubin}

\address[Yvan Saint-Aubin]{
D\'{e}partement de math\'{e}matiques et de statistique \\
Universit\'{e} de Montr\'{e}al \\
Qu\'{e}bec, Canada, H3C~3J7.
}

\email{yvan.saint-aubin@umontreal.ca}

\date{\today}


\begin{abstract} The affine Temperley-Lieb algebra $\atl N(\beta)$ is an infinite-dimensional algebra parametrized by a number $\beta \in \mathbb{C}$ and an integer $N\in \mathbb{N}$. It naturally acts on $(\mathbb{C}^2)^{\otimes N}$ to produce a family of representations labeled by an additional parameter $z\in\mathbb C^\times$. The structure of these representations, which were first introduced by Pasquier and Saleur \cite{PS} in their study of spin chains, is here made explicit. They share their composition factors with the cellular $\atl N(\beta)$-modules of Graham and Lehrer \cite{GL-Aff}, but differ from the latter representations by the direction of about half of the arrows of their Loewy diagrams. The proof of this statement uses a morphism introduced by Morin-Duchesne and Saint-Aubin \cite{MDSA} as well as new maps that intertwine various $\atl N(\beta)$-actions on the XXZ chain and generalize applications studied by Deguchi {\em et al} \cite{Deguchi} and after in \cite{MDSAloop}. 

\end{abstract}

\maketitle

\tableofcontents

\onehalfspacing

\newpage

\begin{center}{\scshape List of common symbols}\end{center}

\medskip

\begin{center}
\begin{tabular}{rl}
$\atl N(\beta)$ & the affine Temperley-Lieb algebra (subsection \ref{sec:atln})\\
$\tl N(\beta)$ & the (usual) Temperley-Lieb algebra (subsection \ref{sec:atln})\\
$\beta$			& complex parameter of the algebras $\atl N(\beta)$ and $\tl{N}(\beta)$ ($\beta=-q-q^{-1}$)\\
$\lambda, \lambda_N, \Lambda, \Lambda_N, \Delta, \Delta_N$ & sets of pairs $(d,z)$ parametrizing $\atl N$-modules (section \ref{sec:main})\\
$\preceq, \unlhd$	& partial orders on $\lambda$ and $\Delta$ respectively (section \ref{sec:main})\\
problematic pairs 	& the pairs $(0,\pm q)$ when $q+q^{-1}=0$ and $N$ is even (subsection \ref{sec:atln})\\
$(m,n)$-diagrams& elements of diagrammatic basis of $\atl N$ (subsection \ref{sec:atlOnXXZ})\\
$\Bnu_{m,n}$	& diagrammatric basis of the vector space of $(m,n)$-diagrams (subsection \ref{sec:atln})\\
$|v|$ 		& rank of the diagram $v$ (subsection \ref{sec:atln})\\
$\Cell{N;d,z}$	& cellular modules over the algebra $\atl N$ (subsection \ref{sec:atln})\\
$\langle\ ,\ \rangle_{N;d,z}$ & $\mathbb{C}$-valued bilinear form on $\Cell{N;d,z}\times \Cell{N;d,z^{-1}}$ (subsection \ref{sec:morphismesGL})\\
$\Irre{N;d,z}$	& simple $\atl N$-module associated to the pair $(d,z)\in\Lambda_N$ (sections \ref{sec:main} and \ref{sec:atln})\\
$\gl{(d,z);(t,x)}$	& Graham-Lehrer morphisms between cellular modules (subsection \ref{sec:morphismesGL})\\
$H_{XXZ}$		& Hamiltonian acting on the XXZ spin chain (subsection \ref{sec:atlOnXXZ})\\
$z$				& twist parameter $\in\mathbb C^\times$ (subsection \ref{sec:atlOnXXZ})\\
$S^z$			& $z$-component of the spin operator (subsection \ref{sec:atlOnXXZ})\\
$\XXZ{N;z}{\pm}$& spin chains seen as a $\atl N$-module (subsection \ref{sec:atlOnXXZ}) or as $\lusztpm$-module (subsection \ref{sub:homoXXZ})\\
$\XXZ{N;d,z}{\pm}$& $\atl N$-submodule of $\XXZ{N;z}{\pm}$ seen as eigenspace of $S^z$ with eigenvalue $d$ (subsection \ref{sec:atlOnXXZ})\\
$\ay{N;d,z}$	& $\atl N$-morphism $\Cell{N;d,z}\to \XXZ{N;d,z}{+}$ (subsection \ref{sec:ind})\\
$\pg{N;d,z}$ 		& generic part of the cellular module $\Cell{N;d,z}$ (subsection \ref{sec:ind})\\
$\BXXZ N$		& spin basis for $\XXZ N{}=(\mathbb C^2)^{\otimes N}$ (subsection \ref{sec:atlOnXXZ})\\
$\star$- and $\circ$-dualities	& functors $\modY \atl N\to \modY \atl N$ (subsection \ref{sec:dualities})\\
$\ell$ 		& smallest positive integer (if it exists) such that $q^{2\ell}=1$ (section \ref{sec:main})\\
$\luszt$ 		& Lusztig's quantum group (subsection \ref{sec:luqsl2})\\
$L_q(i),\Delta_q(i),P_q(i)$ & simple, Weyl and projective modules over $\luszt$ (subsection \ref{sec:luqsl2})\\
$\jj_{(t,x);(d,z)}^\pm$ and $\ii_{(d,z);(t,x)}^\pm$ & $\atl N$-morphisms between eigenspaces $\XXZ{N;d,z}\pm$ and $\XXZ{N;t,x}\pm$ (subsection \ref{sub:homoXXZ})\\
$\kk^{(n)}$ and $\mm^{(n)}$	& maps $\XXZ{N;d+2n\ell,zq^{n\ell}}{+}\rightarrow \XXZ{N;d,z}{+}$ and  $\XXZ{N;d,z}{+}\rightarrow \XXZ{N;d+2n\ell,zq^{n\ell}}{+}$ (subsection \ref{sub:homoXXZ})\\
\end{tabular}
\end{center}
\vspace*{0.5cm}

%%%%%%%%%%%
%
% intro
%
%%%%%%%%%%%

%%%%%%%%%%%%%
%%%%%%%%%%%%%
%%%%%   %%%%%
%%%%% 1 %%%%%
%%%%%   %%%%%
%%%%%%%%%%%%%
%%%%%%%%%%%%%

\begin{section}{Introduction} \label{sec:Intro}
One of the simplest incarnation of the classical Schur-Weyl duality ties the actions of the Lie algebra of $2\times 2$ traceless matrices $\mathfrak{sl}_2(\mathbb{C})$ and of the symmetric group on $N$ letters $\mathfrak{S}_N$ on $V=(\mathbb{C}^2)^{\otimes N}$. The key result is then that the enveloping algebra $U\mathfrak{sl}_2$ and the group algebra $\mathbb{C}\mathfrak{S}_N$ are mutual centralizers on $V$, that is,
$$ \End_{U\mathfrak{sl}_2}V \simeq \mathbb{C}\mathfrak{S}_N \quad \text{and} \quad \End_{\mathbb{C}\mathfrak{S}_N}V \simeq U\mathfrak{sl}_2.$$
The quantum analogue of this result replaces the actions of $\mathfrak{sl}_2(\mathbb{C})$ and $\mathfrak{S}_N$ by actions of the quantum group $U_q\mathfrak{sl}_2$ and of the Temperley-Lieb algebra $\tl N(\beta = -q-q^{-1})$: $$ \End_{U_q\mathfrak{sl}_2}V \simeq \tl{N}(\beta) \quad \text{and} \quad \End_{\tl{N}(\beta)}V \simeq U_q\mathfrak{sl}_2.$$ 
It is a remarkable fact that this action of $\tl{N}(\beta)$ on $V$ extends to a family of actions of the affine Temperley-Lieb algebra $\atl{N}(\beta)$. Of course, the resulting $\atl{N}(\beta)$-actions no longer commute with the action of the whole quantum group $U_q\mathfrak{sl}_2$ and the quantum Schur-Weyl duality is thus broken in the affine case. Still the action of $U_q\mathfrak{sl}_2$, or of Lusztig extension $\luszt$ when $q$ is a root of unity, is a key tool to probe the structure of the $\atl N(\beta)$-action.

To our knowledge, the above-mentioned action of $\atl N(\beta)$ on $V=(\mathbb{C}^2)^{\otimes N}$ appeared first in the physics literature where states of a periodic XXZ chain with $N$ spin-$\frac 12$ are vectors in $V$. Their evolution is described through 
%Y le "an" a ete remplace par un "a"
a $\mathbb{C}$-endomorphism of $V$, the Hamiltonian, which is parametrized by
two non-zero complex numbers: $q$ measures the anisotropy along the Z direction and $z$ the twist in the interaction between the extreme spins of the chain. Pasquier and Saleur \cite{PS} introduced this latter parameter by analogy with the anisotropic limit of the transfer matrix for the six-vertex model with twisted boundary conditions. They also recognized the rich algebraic structure of the periodic XXZ chain. In fact, their work leads to four different actions of the affine Temperley-Lieb algebra $\atl N(\beta)$ on $V$ such that the generators $e_i$ of $\atl N(\beta)$ are represented by nearest-neighbor interactions. The four corresponding $\atl N(\beta)$-modules are denoted by $\XXZ{N;z^{\pm1}}\pm$ (where the two $\pm$'s are independent) and will refered to as periodic XXZ spin chains to underline their physical origin. The goal of this article is to reveal the structure of these modules as a function of $q$ and $z$.

The affine Temperley-Lieb algebra is not the symmetry algebra of the periodic XXZ chain. In fact the Hamiltonian is simply an element of $\atl N(\beta)$ (the sum of the generators $e_i$) and commutes only with a part of the algebra. Knowing the structure of the modules $\XXZ{N;z^{\pm 1}}{\pm}$ is however still important. It may indeed help, from a physical point of view, partially diagonalize the Hamiltonian by identifying the simple $\atl{N}(\beta)$-modules intervening as composition factors of the overall spin chain and could also play a crucial role in studying the continuum limit of the model. These reasons have already led others to probe the structure of the chain as a $\atl N(\beta)$-module and to discover specific examples \cite{PS,GRS1, GRS2} of the general result presented here. 

There are also mathematical reasons to pursue this study. The representation theory of $\atl N(\beta)$ was first studied by Graham and Lehrer \cite{GL-Aff}. Many have pursued their work, but a satisfactory understanding of this theory is still to come. For example a full description of the finite-dimensional indecomposable modules remains to be completed. The structure of the modules $\XXZ{N;z^{\pm1}}\pm$, which is the goal of the present work, is of course linked to this description. This structure has not been described systematically before; it is somewhat complicated and provides a further indication of the richness of the representation theory of $\atl N(\beta)$. Interestingly it is also reminiscent of that of certain infinite-dimensional modules over the Virasoro algebra $\mathfrak{Vir}$, namely the so-called Feigin-Fuchs modules. Our results thus belong to the body of works initiated by Koo and Saleur in \cite{koo1994representations} on the link between the representation theory of the algebra $\mathfrak{Vir}$ and that of Temperley-Lieb families. This link between the (non-associative) Lie algebra $\mathfrak{Vir}$ and the (associative) $\atl N(\beta)$ has recently attracted interest from the mathematical community. Finally the knowledge of this structure could also help characterize the morphism-spaces of the spin chains and could thus lead to a better understanding of the broken quantum Schur-Weyl duality described above. 

Our tools in probing the structure of the $\atl{N}(\beta)$-action on the XXZ chain mainly come from the finite-dimensional representation theory of $\atl{N}(\beta)$ and of $\luszt$, Lusztig's extension of $\uqsl$. For the quantum group $\luszt$, many authors have contributed to the theory and several monographs \cite{Jantzen, ChariPressley, Klimyk} with some recent work \cite{AndersenTubbenhauer} provide quick reviews that fit our needs. For the affine algebra $\atl{N}(\beta)$, we mainly use the pioneering paper of Graham and Lehrer with in particular their cellular modules and their classification of simple finite-dimensional $\atl{N}(\beta)$-modules. 

The article is organized as follows. Section \ref{sec:main} states the main results by relating the Loewy diagrams of the cellular modules of Graham and Lehrer to those of the $S^z$-eigenspaces $\XXZ{N;d,z^{\pm1}}\pm$ of the XXZ chains $\XXZ{N;z^{\pm 1}}{\pm}$. It also recalls Graham and Lehrer's classification of the simple finite-dimensional representations of $\atl{N}(\beta)$ and the partial order on the pairs $(d,z)$ labelling the cellular modules. Sections \ref{sec:xxzatln} and \ref{sec:xxzLUq} then study the periodic XXZ chains as representations of $\atl{N}(\beta)$ and $\luszt$ respectively. Beside recalling some basic definitions related to the algebra $\atl{N}(\beta)$, section \ref{sec:xxzatln} gives the morphisms between cellular modules \cite{GL-Aff} and constructs isomorphisms (e.g.~spin flips) and dualities between distinct eigenspaces $\XXZ{N;d,z^{\pm 1}}{\pm}$. Finally, it recalls a morphism, first defined by Morin-Duchesne and Saint-Aubin 
in \cite{MDSA}, connecting cellular modules to the eigenspaces $\XXZ{N;d,z^{\pm1}}\pm$. This morphism already leads to a description of some of the simplest modules $\XXZ{N;d,z^{\pm1}}\pm$ (for example when the parameter $q$ is not a root of unity). Section \ref{sec:xxzLUq} recalls the basic definition of Lusztig's quantum group $\luszt$ for $q$ a root of unity and describes its basic modules: the Weyl, the indecomposable projective and the simple ones. A new explicit realization of the indecomposable projective modules is introduced. It reveals clearly two Weyl modules that build this projective, the first as a submodule, the second as a quotient. Some fusion rules are also given. At last, this section constructs one of the key tools used in this paper, namely $\atl N(\beta)$-morphisms $\XXZ{N;t,x}+\to\XXZ{N;d,z}+$ built out of elements of $\luszt$. Section \ref{sec:struc} contains the proof of the main result for the cases where the parameter $q$ is a root of unity. The conclusion follows with a quick summary of the arguments leading to the main result, some of its corollaries and possible future research avenues. Appendices complete the text: appendix \ref{app:a} recalls the basic of $q$-numbers, appendix \ref{app:b} proves some results on $\luszt$, and appendix \ref{app:problematic} covers cases that were excluded of the body text for the sake of clarity.
\end{section}

%%%%%%%%%%%%%
%%%%%%%%%%%%%
%%%%%   %%%%%
%%%%% 2 %%%%%
%%%%%   %%%%%
%%%%%%%%%%%%%
%%%%%%%%%%%%%

\begin{section}{The main results} \label{sec:main}

This section states the main results. While it uses notations and objects that will be defined only in the following sections, a reader familiar with the affine Temperley-Lieb algebra and its action on the periodic XXZ spin-$\frac12$ chain will be able to understand the results and easily refer to them when needed. Hopefully, for the other readers, this section will serve as a roadmap to the goal to be achieved.

The affine Temperley-Lieb algebras $\atl N(\beta)$ form a family of infinite-dimensional associative unital algebras labeled by a positive integer $N$ and a complex number $\beta = -q-q^{-1}$ where $q\in\mathbb C^\times$. Graham and Lehrer \cite{GL-Aff} gave a complete classification of simple finite-dimensional $\atl N(\beta)$-modules as a function of $N$ and $q$. A first step toward this classification is the definition of a family of modules $\Cell{N;d,z}$ that they called {\em cellular}. Two further indices are used to label the family: $d$ is a non-negative integer $\leq N$ of the parity of $N$ and $z$ is a non-zero complex number. The modules $\Cell{N;d,z}$ are not simple in general, but their quotients by their radicals are (except for the problematic pairs $(d,z)$ defined after \eqref{eq:Lambda}). Moreover any simple finite-dimensional $\atl N(\beta)$-module can be obtained through such a quotient and the correspondence between cellular modules and isomorphism classes of finite-dimensional simple ones is almost one-to-one. The two sets
\begin{equation}\label{eq:lambda}
\lambda=\mathbb{Z}_{\geq 0}\times \mathbb{C}^{\times}\qquad\text{and}\qquad 
\lambda_N=\{(d,z)\in\lambda\,|\, d\leq N,\, d\equiv N\modY 2\}
\end{equation}
describe the possible pairs $(d,z)$ without and with the constraint of parity and $d\leq N$. They will be used to describe the action of $\atl{N}(\beta)$ on the XXZ chains. However, finer sets are required to obtain a one-to-one labeling of the simple finite-dimensional modules. Let $\sim$ be the equivalence relation on $\lambda$ that identifies $(0,z)$ and $(0,z^{-1})$. (When $d\neq 0$, the pairs $(d,z)$ and $(d,z^{-1})$ are not equivalent under $\sim$.) With this relation, define
\begin{equation}\label{eq:Lambda}\Lambda=\{(d,z)\in\lambda\}/\sim\qquad\text{and}\qquad 
\Lambda_N=\{(d,z)\in\Lambda\,|\, d\leq N,\, d\equiv N\modY 2\}\end{equation}
where the pairs $(0,q)\sim(0,q^{-1})$ are to be removed\footnote{Note that these pairs $(0,q)$ and $(0,q^{-1})$ are not excluded from the set $\lambda_N$.} from $\Lambda_N$ when $N$ is even with $q+q^{-1}=0$. The study of the chains related to these pairs $(0,q)$ and $(0,q^{-1})$ will require a finer analysis. For that reason, we call them {\em problematic pairs}.
\begin{theorem}[\cite{GL-Aff}]\label{thm:listOfSimples} The family of $\Irre{N;d,z}=\Cell{N;d,z}/\rad\Cell{N;d,z}$ indexed by $(d,z)\in\Lambda_N$ is a complete set of non-isomorphic simple finite-dimensional $\atl N(\beta)$-modules. 
\end{theorem}
The structure of the cellular modules $\Cell{N;d,z}$ depends on $q$ (or more specifically on $\beta$). To describe this dependence, Graham and Lehrer ordered the set $\lambda$ by specifying that a pair $(t,x)$ succeeds $(d,z)$ whenever there exists a non-negative integer $m$ such that $t=d+2m$ with
\begin{equation}\label{eq:conditionsAB}
\text{(A)}\ \ z^2=q^t\text{ and }x=zq^{-m}\qquad\text{ or }\qquad
\text{(B)}\ \ z^2=q^{-t}\text{ and }x=zq^{m}.
\end{equation}
The successor $(t,x)$ is said to be {\em strict} if $m$ is positive and a {\em direct successor through condition} A (or B) if $m$ is the smallest positive integer solving condition A (resp.~B). Let $\preceq$ be the weakest partial order on $\lambda$ that contains $(d,z)\preceq(t,x)$ for all successions\footnote{This partial order induces one on $\Lambda$. Indeed, the only possible difficulty would be for the pairs $(0,z)$ and $(0,z^{-1})$ that are identified by the equivalence relation $\sim$. However, these pairs have the same successors as the direct successor of $(0,z)$ through condition A (or B) is the direct successor of $(0,z^{-1})$ through condition B (resp.~A). The same symbol $\preceq$ can thus be used non-ambiguously for the induced partial order on $\Lambda$.} in $\lambda$. 

The set of successors of a given pair $(d,z)\in\lambda$ can be described explicitly as follows (see also Appendix B of \cite{BSAfusion}).

\noindent (a) If $q$ is not a root of unity, then there is at most one element succeeding strictly $(d,z)$ in $\lambda$.
 
\noindent (b) Let $q$ be a root of unity and $\ell$ be the smallest positive integer such that $q^{2\ell}=1$.
\begin{itemize}
\item[(i)] If $q^d=z^2$ and $z^4=1$, then the successors of $(d,z)$ in $\lambda$ are 
$$(d,z)\preceq(d+2\ell,zq^\ell)\preceq(d+4\ell,z)\preceq(d+6\ell,zq^\ell)\preceq\cdots$$
and conditions (A) or (B) give the same immediate successor starting from any element $(d+2m\ell, zq^{m\ell})$.
\item[(ii)] If $q^d=z^2$ or $z^{-2}$ with $z^4\neq 1$, either condition A or B has a solution $(t,x)$ with $d<t<d+2\ell$ and the successors of $(d,z)$ in $\lambda$ are 
$$(d,z)\preceq(t,x)\preceq(d+2\ell,zq^\ell)\preceq(t+2\ell,xq^\ell)\preceq(d+4\ell,z)\preceq\cdots$$
\item[(iii)] If $q^d$ is distinct from both $z^2$ and $z^{-2}$, then the successors of $(d,z)$ in $\lambda$ are organized as follows
\begin{equation*}
\begin{tikzpicture}[baseline={(current bounding box.center)},scale=0.45]
\node (k0) at (0,0) [] {$(d_0,z_0)$};
\node (j0) at (4,0) [] {$(s_0,y_0)$};
\node (k1) at (8,0) [] {$(d_1,z_1)$};
\node (j1) at (12,0) [] {$(s_1,y_1)$};
\node (i0) at (4,2) [] {$(t_0,x_0)$};
\node (h0) at (8,2) [] {$(h_0,v_0)$};
\node (i1) at (12,2) [] {$(t_1,x_1)$};
\node (k2) at (15.5,0) [] {$\dots $};
\node (h1) at (15.5,2) [] {$\dots $};
\draw[<-, dashed] (k0) -- (j0);\draw[<-, dashed] (k0) -- (i0);
\draw[<-, dashed] (j0) -- (k1);\draw[<-, dashed] (j0) -- (h0);
\draw[<-, dashed] (i0) -- (k1);\draw[<-, dashed] (i0) -- (h0);
\draw[<-, dashed] (k1) -- (j1);\draw[<-, dashed] (k1) -- (i1);
\draw[<-, dashed] (h0) -- (j1);\draw[<-, dashed] (h0) -- (i1);
\draw[dashed,<-] (j1) -- (k2);\draw[dashed,<-] (j1) -- (h1);
\draw[dashed,<-] (i1) -- (k2);\draw[dashed,<-] (i1) -- (h1);
\draw[dotted,thick] (-1.5,-0.75) -- (-1.5,0.5) -- (3,2.75) -- (9.25,2.75) -- (9.25,1.5) -- (5,-0.75) -- (-1.5,-0.75);
\end{tikzpicture}
\end{equation*} 
with the direct succession $(d_0,z_0)\preceq(t_0,x_0)$ written as $(d_0,z_0)\ \begin{tikzpicture}[baseline=-1.5,scale=0.65]\draw[<-, dashed] (0,0) -- (1,0);\end{tikzpicture}\ (t_0,x_0)$. The following notation is used. Let $s$ be the smallest integer larger than $d$ solving $z^2=q^s$ and set $k=\frac12(s-d)$. Then, the successors of $(d,z)$ are 
\begin{alignat*}{3}(d_a,z_a)&=(d+2a\ell,zq^{a\ell}),
\qquad\quad &&(t_a,x_a)=(-s+\delta_t+2a\ell, z^{-1}q^{a\ell+k+\delta_t/2}),\\
(s_a,y_a)&=(s+2a\ell, zq^{a\ell-k}),
\qquad && (h_a,v_a)=(-d+\delta_h+2a\ell, z^{-1}q^{a\ell+\delta_h/2}),
\end{alignat*}
where $a\geq 0$ and with $\delta_t$ and $\delta_h$ the smallest integer multiples of $2\ell$ such that $t_0$ is larger than $d_0$, and $h_0$ is larger than both $s_0$ and $t_0$. Direct successors tied horizontally are obtained through condition A, those tied diagonally through condition B. The partial order $\preceq$ can be constructed from the four successors in the dotted polygon by translation, the $(d_0,z_0)$ going onto the $(d_i,z_i)$'s, etc. The successors are then obtained by adding to the integer ($d_0, t_0, s_0$ or $h_0$) a multiple of $2\ell$ and by multiplying the complex number ($z_0, x_0, y_0$ or $v_0$) by a power of $q^\ell=\pm 1$.
\end{itemize}
Using this description, the structure of the modules $\Cell{N;d,z}$ can be revealed through Loewy diagrams. In such a diagram, nodes are composition factors of the module under consideration. The shortened $(d,z)$ will stand here for the irreducible $\Irre{N;d,z}$. An arrow pointing from $(d,z)$ to $(t,x)$ indicates that any submodule containing the composition factor $(d,z)$ will also contain $(t,x)$ as a composition factor. A connected Loewy diagram represents an indecomposable module.
\begin{theorem}[\cite{GL-Aff}\footnote{This theorem is not shown explicitly in \cite{GL-Aff} but follows directly from the results proven there (see also the discussion given in \cite{BSAfusion}).}]\label{thm:GL} Let $q\in\mathbb C^\times$ and $(d,z)\in\Lambda_N$ be distinct from the problematic pairs. The description of the cellular modules $\Cell{N;d,z}$ is broken into the same subcases as those of the weakest partial order $\preceq$.

\noindent(a) {\bfseries\itshape The case when $q$ is not a root of unity.} If there is no (strict) successor of $(d,z)$ in $\lambda_N$, then $\Cell{N;d,z}$ is simple. If there is one, say $(t,x)\in \lambda_N$, then the Loewy diagram of $\Cell{N;d,z}$ is
$$
\begin{tikzpicture}[baseline={(current bounding box.center)},scale=1/3]
		\node (a) at (0,0) {$(d,z)$};
		\node (b) at (0,-2.75) {$(t,x)$};
		\draw[->] (a) -- (b);
	\end{tikzpicture}
$$

\noindent(b) {\bfseries\itshape The case when $q$ is a root of unity.} Here, $\ell$ is the smallest positive integer such that $q^{2\ell} =1$ and all pairs $(i,u)$ with $i>N$, that is $(i,u)\in\lambda\backslash \lambda_N$, must be removed. The Loewy diagrams of $\Cell{N;d,z}$ for the subcases {\em (i), (ii)} and {\em (iii)} are
$$\begin{tikzpicture}[baseline={(current bounding box.center)},scale=1/3]
		\node (a) at (0,0) {$(d,z)$};
		\node (b) at (0,-2.75) {$(d+2\ell,zq^\ell)$};
		\node (c) at (0,-5.5) {$(d+4\ell,z)$};
		\node (d) at (0,-8.25) {$(d+6\ell,zq^\ell)$};
		\node at (0,-10.95) {$\vdots$};
		\node (e) at (0,-10.6) { };
		\draw[->] (a) -- (b); \draw[->] (b) -- (c); \draw[->] (c) -- (d); \draw[->] (d) -- (e);
		\node at (0,-16.5) {(i)};
	\end{tikzpicture}
\qquad\qquad\qquad
\begin{tikzpicture}[baseline={(current bounding box.center)},scale=1/3]
		\node (a) at (0,0) {$(d,z)$};
		\node (b) at (0,-2.75) {$(t,x)$};
		\node (c) at (0,-5.5) {$(d+2\ell,zq^\ell)$};
		\node (d) at (0,-8.25) {$(t+2\ell,xq^\ell)$};
		\node (e) at (0,-11) {$(d+4\ell,z)$};
		\node at (0,-13.7) {$\vdots$};
		\node (f) at (0,-13.4) { };
		\draw[->] (a) -- (b); \draw[->] (b) -- (c); \draw[->] (c) -- (d); 
		\draw[->] (d) -- (e); \draw[->] (e) -- (f);
		\node at (0,-16.5) {(ii)};
	\end{tikzpicture}
\qquad\qquad\qquad
\begin{tikzpicture}[baseline={(current bounding box.center)},scale=1/3]
\node (k0) at (0,8.25) [] {$(d_0,z_0)$};
\node (j0) at (0,5.5) [] {$(s_0,y_0)$};
\node (k1) at (0,2.75) [] {$(d_1,z_1)$};
\node (j1) at (0,0) [] {$(s_1,y_1)$};
\node (i0) at (4,5.5) [] {$(t_0,x_0)$};
\node (h0) at (4,2.75) [] {$(h_0,v_0)$};
\node (i1) at (4,0) [] {$(t_1,x_1)$};
\node (k2) at (0,-2.75) [] {$(d_2,z_2)$};
\node (h1) at (4,-2.75) [] {$(h_1,v_1)$};
\node (k3) at (0,-5.6) [] {$\vdots $};
\node (h2) at (4,-5.6) [] {$\vdots $};
\draw[->] (k0) -- (j0);\draw[->] (k0) -- (i0);
\draw[->] (j0) -- (k1);\draw[->] (j0) -- (h0);
\draw[->] (i0) -- (k1);\draw[->] (i0) -- (h0);
\draw[->] (k1) -- (j1);\draw[->] (k1) -- (i1);
\draw[->] (h0) -- (j1);\draw[->] (h0) -- (i1);
\draw[->] (j1) -- (k2);\draw[->] (i1) -- (k2);
\draw[->] (j1) -- (h1);\draw[->] (i1) -- (h1);
\draw[->] (k2) -- (0,-5);\draw[->] (h1) -- (4,-5);
\draw[->] (k2) -- (3,-5);\draw[->] (h1) -- (1,-5);
\node at (1.5,-8.3) {(iii)};
\end{tikzpicture}$$
In the Loewy diagram associated to subcase (iii), the vertical arrows (the diagonal ones) point to a pair that solves condition A (resp.~B) starting from the arrow's source. The  notation is otherwise that of the description of the order $\preceq$.
\end{theorem}
\noindent This theorem has been recalled not only since it fixes the notation for our results but because the structure of the $\atl N(\beta)$-action on the XXZ spin-$\frac12$ chain is strangely similar to the above one. The Hamiltonian $H_{\text{XXZ}}$ is a map from $\XXZ N{}=(\mathbb C^2)^{\otimes N}$ into itself and preserves the eigenspaces of the operator $S^z=\frac12\sum_{1\leq i\leq N}\sigma_i^z$ (see section \ref{sec:atlOnXXZ}). By describing how the chain closes periodically, that is how the spin at site $1$ interacts with the one at site $N$, Pasquier and Saleur \cite{PS} introduced a family of $\atl{N}(\beta)$-actions on $\XXZ N{}$ labeled by a sign $\pm$ and a {\em twist} parameter $z\in\mathbb C^\times$. The ensuing module will here be noted $\XXZ{N;z}\pm$. Like the Hamiltonian, any of these actions stabilizes the eigenspaces of $S^z$. Thus, a given $\XXZ{N;z}\pm$ decomposes, as a $\atl{N}(\beta)$-module, into a direct sum $\XXZ{N;z}\pm={\bigoplus{}{\hskip-1pt}'{\hskip-2pt}}_{-N\leq d\leq N}\ \XXZ{N;d,z}\pm$ where $\XXZ{N;d,z}\pm$ is the eigenspace of $S^z$ with eigenvalue $\frac d2$ and where ${\bigoplus{}{\hskip-1pt}'}$ denotes a direct sum with steps of $2$, that is $d\in\{-N,-N+2,\dots,N\}$. The submodules $\XXZ{N;d,z}{\pm}$ are related by spin flip (see section \ref{sec:atlOnXXZ}) and by two involutive endofunctors $\star$ and $\circ$ (contravariant and covariant, respectively) of the category of finite-dimensional $\atl{N}$-modules. These functors preserve irreducibility, dimensions and act on Loewy diagrams of modules in a simple way (described in section \ref{sec:dualities}).

\begin{proposition}\label{prop:dualMerge} Spin flip induces $\atl{N}$-linear isomorphisms $\XXZ{N;d,z}\pm\simeq \XXZ{N;-d,z^{-1}}\mp$ while the involutive endofunctors $\star$ and $\circ$ give $(\XXZ{N;d,z}\pm)^\star\simeq\XXZ{N;d,z^{-1}}\pm$ and $(\XXZ{N;d,z}\pm)^\circ\simeq\XXZ{N;d,z^{-1}}\mp$ as $\atl{N}$-modules. 
Finally, $(\Irre{N;d,z}^*)^{\circ} \simeq \Irre{N;d,z}$ and the Loewy diagram of $\XXZ{N;d,z}{-}$ may be obtained from that of $\XXZ{N;d,z}+$ by flipping all arrows (if any) without changing the composition factors.
\end{proposition}
\noindent The spin flip isomorphisms and the last part of proposition \ref{prop:dualMerge} show that it is sufficient to study the submodules $\XXZ{N;d,z}{+}$ with non-negative $d$ to deduce the structure of all possible submodules $\XXZ{N;d,z}{\pm}$. The main result of the present article is an explicit description of the structure of these modules.
\begin{theorem}\label{thm:main} Let $q\in\mathbb C^\times$ and $(d,z)\in\lambda_N$ with $d\geq 0$. The composition factors of $\XXZ{N;d,z}{\pm}$ coincide with those of $\Cell{N;d,z}$. The description of the $\atl{N}(\beta)$-modules $\XXZ{N;d,z}+$ is broken into the same subcases as those of the weakest partial order $\preceq$ and uses the same notation. The Loewy diagram of $\XXZ{N;d,z}-$ is obtained from that of $\XXZ{N;d,z}+$ by flipping all arrows (if any).

\noindent(a) {\bfseries\itshape The case when $q$ is not a root of unity.} If there is no (strict) successor of $(d,z)$ in $\lambda_N$, $\XXZ{N;d,z}{+}$ is simple. If there is one, say $(t,x)\in \lambda_N$, then the Loewy diagram of $\XXZ{N;d,z}{+}$ is either 
$$
\begin{tikzpicture}[baseline={(current bounding box.center)},scale=1/3]
		\node (a) at (0,0) {$(t,x)$};
		\node (b) at (0,-2.75) {$(d,z)$};
		\draw[->] (a) -- (b);
	\end{tikzpicture}
\qquad\qquad\text{or}\qquad\qquad
\begin{tikzpicture}[baseline={(current bounding box.center)},scale=1/3]
		\node (a) at (0,0) {$(d,z)$};
		\node (b) at (0,-2.75) {$(t,x)$};
		\draw[->] (a) -- (b);
	\end{tikzpicture}
$$
depending on whether $(t,x)$ is obtained from $(d,z)$ by condition A (left diagram) or B (right one).

\noindent(b) {\bfseries\itshape The case when $q$ is a root of unity.} In what follows, all pairs $(i,u)$ with $i>N$, that is $(i,u)\in\lambda\backslash\lambda_N$, are to be removed. The structure of $\XXZ{N;d,z}+$ is
$$
\begin{tikzpicture}[baseline={(current bounding box.center)},scale=1/3]
		\node (a) at (0,0) {$(d,z)$};
		\node (b) at (0,-2.75) {$(d+2\ell;zq^\ell)$};
		\node (c) at (0,-5.5) {$(d+4\ell;z)$};
		\node (d) at (0,-8.25) {$(d+6\ell;zq^\ell)$};
		\node (e) at (0,-11) {$\vdots$};
		\node at (0,-16.5) {(i)};
	\end{tikzpicture}
\qquad\qquad\qquad
\begin{tikzpicture}[baseline={(current bounding box.center)},scale=1/3]
		\node (a) at (0,0) {$(d,z)$};
		\node (b) at (0,-2.75) {$(t,x)$};
		\node (c) at (0,-5.5) {$(d+2\ell,zq^\ell)$};
		\node (d) at (0,-8.25) {$(t+2\ell,xq^\ell)$};
		\node (e) at (0,-11) {$(d+4\ell,z)$};
		\node at (0,-13.7) {$\vdots$};
		\node (f) at (0,-13.4) { };
		\draw[<-,very thick] (a) -- (b); \draw[->] (b) -- (c); \draw[<-,very thick] (c) -- (d); 
		\draw[->] (d) -- (e); \draw[<-,very thick] (e) -- (f);
\node at (4,-7.5) {or};
		\node (aa) at (8,0) {$(d,z)$};
		\node (bb) at (8,-2.75) {$(t,x)$};
		\node (cc) at (8,-5.5) {$(d+2\ell,zq^\ell)$};
		\node (dd) at (8,-8.25) {$(t+2\ell,xq^\ell)$};
		\node (ee) at (8,-11) {$(d+4\ell,z)$};
		\node at (8,-13.7) {$\vdots$};
		\node (ff) at (8,-13.4) { };
		\draw[->] (aa) -- (bb); \draw[<-,very thick] (bb) -- (cc); \draw[->] (cc) -- (dd); 
		\draw[<-,very thick] (dd) -- (ee); \draw[->] (ee) -- (ff);
	\node at (4,-16.5) {(ii)};
	\end{tikzpicture}
\qquad\qquad\qquad
\begin{tikzpicture}[baseline={(current bounding box.center)},scale=1/3]
\node (k0) at (0,8.25) [] {$(d_0,z_0)$};
\node (j0) at (0,5.5) [] {$(s_0,y_0)$};
\node (k1) at (0,2.75) [] {$(d_1,z_1)$};
\node (j1) at (0,0) [] {$(s_1,y_1)$};
\node (i0) at (4,5.5) [] {$(t_0,x_0)$};
\node (h0) at (4,2.75) [] {$(h_0,v_0)$};
\node (i1) at (4,0) [] {$(t_1,x_1)$};
\node (k2) at (0,-2.75) [] {$(d_2,z_2)$};
\node (h1) at (4,-2.75) [] {$(h_1,v_1)$};
\node (k3) at (0,-5.6) [] {$\vdots $};
\node (h2) at (4,-5.6) [] {$\vdots $};
\draw[<-,very thick] (k0) -- (j0);\draw[->] (k0) -- (i0);
\draw[->] (j0) -- (k1);\draw[->] (j0) -- (h0);
\draw[<-,very thick] (i0) -- (k1);\draw[<-,very thick] (i0) -- (h0);
\draw[<-,very thick] (k1) -- (j1);\draw[->] (k1) -- (i1);
\draw[<-,very thick] (h0) -- (j1);\draw[->] (h0) -- (i1);
\draw[->] (j1) -- (k2);\draw[<-, very thick] (i1) -- (k2);
\draw[->] (j1) -- (h1);\draw[<-, very thick] (i1) -- (h1);
\draw[<-, very thick] (k2) -- (0,-5);\draw[->] (h1) -- (4,-5);
\draw[->] (k2) -- (3,-5);\draw[<-, very thick] (h1) -- (1,-5);
\node at (1.5,-8.3) {(iii)};
\end{tikzpicture}
$$
\noindent The absence of arrows in subcase (i) means that the module is semisimple: $\XXZ{N;d,z}{+}\simeq \bigoplus_{i=0}^m \Irre{N;d+2i\ell, zq^{i\ell}}$ with $m$ the largest integer such that $d+2m\ell\leq N$. The diagram for subcase (ii) is chosen according to whether $(t,x)$ is obtained from $(d,z)$ by condition A (left diagram) or B (right one). In subcase (iii), the direction of the arrows are such that the composition factors $(s_a, y_a)$ are sources and the $(t_a,x_a)$ targets. The arrows drawn thicker in subcases (ii) and (iii) are those that have been flipped compared to the corresponding ones in the Loewy diagram of $\Cell{N;d,z}$. In the case of the problematic pairs, that is for $\XXZ{N;0,q^{\pm 1}}{+}$ with $q+q^{-1}=0$ and $N$ even, the Loewy diagram is of type (iii), but the pair $(d_0,z_0)$ is always omitted and an arrow $(s_0,y_0) \to (t_0,x_0)$ is added if $N=2$. \end{theorem}

The proof of this result relies on new $\atl{N}(\beta)$-morphisms between $S^z$-eigenspaces $\XXZ{N;d,z}{\pm}$. We conclude this overview by describing some of them. Since the eigenvalues of $S^z$ range from $-N$ to $N$ for the chain with $N$ spins, two further sets of pairs $(d,z)$ are needed. They are
\begin{equation}
\Delta=\mathbb Z\times \mathbb C^\times\qquad\text{and}\qquad
\Delta_N=\{(d,z)\in\Delta\,|\ |d|\leq N,\, d\equiv N\text{\ mod\ }2\}.
\end{equation}
The partial order $\preceq$ on $\lambda$ can be extended to a partial order $\unlhd$ on $\Delta$. In $\Delta$, the pair $(t,x)$ succeeds $(d,z)$ if there exists a non-negative integer $m$ such that $t = d+2m$ and
\begin{equation}
\text{(A)}\quad  z^2=q^t\text{ and }x=zq^{-m}\qquad\text{ or }\qquad
\text{(B)}\quad z^2=q^{-t}\text{ and }x=zq^{m}.
\end{equation}
Then $\unlhd$ is defined as the weakest partial order on $\Delta$  containing $(d,z)\unlhd (t,x)$ for all such successions.
\begin{theorem} Let $(d,z),(t,x)\in\Delta$. Then, the following maps are $\atl N(\beta)$-linear under the stated condition:
\begin{align*}
\ii_{(d,z);(t,x)}^+ &:\XXZ{N;t,x}+\to\XXZ{N;d,z}+\quad &\text{if $(d,z)\unlhd (t,x)$ through condition B,}\\
\jj_{(t,x);(d,z)}^+ &:\XXZ{N;d,z}+\to\XXZ{N;t,x}+\quad &\text{if $(d,z)\unlhd (t,x)$ through condition A,}
\end{align*}
where the action of $\ii_{(d,z);(t,x)}^+$ on the usual spin basis $\{|x_1x_2\dots x_N\rangle\,|\, x_1,...,x_N \in\{+,-\}\}$ is through the divided power $F^{(a)}$ of the quantum group $\luszt$: $|x_1x_2\dots x_N\rangle_x\mapsto F^{(a)}|x_1x_2\dots x_N\rangle_z$ with $a=\frac12(t-d)$. The action of $\jj_{(t,x);(d,z)}^+$ is defined similarly, with $F^{(a)}$ replaced by the divided power $E^{(a)}$ of $\luszt$.
\end{theorem}
\noindent These $\atl N$-morphisms extend maps introduced in \cite{PS}, \cite{Deguchi} and \cite{MDSAloop} which were endomorphisms of a given spin chain. The new ones intertwine chains described by different pairs in $\Delta_N$ and satisfy remarkable properties (see theorem \ref{thm:f}).

\end{section}

%%%%%%%%%%%%%
%%%%%%%%%%%%%
%%%%%   %%%%%
%%%%% 3 %%%%%
%%%%%   %%%%%
%%%%%%%%%%%%%
%%%%%%%%%%%%%

\begin{section}{The XXZ spin chain as a $\atl N$-module}\label{sec:xxzatln}

The periodic XXZ spin-$\frac12$ chain is a remarkably rich mathematical object: the neighbor-to-neighbor interaction can be obtained through a representation of the affine Temperley-Lieb algebra $\atl{N}(\beta)$ and the chain is naturally acted upon by the quantum group $\luszt$ when the anisotropy parameter $q$ is a root of unity. Also, the actions of $\atl N(\beta)$ and $\luszt$ partially commute with one another. This section studies the chain as a $\atl N(\beta)$-module and the next as a $\luszt$-module.

%%%%%%%%%%%%%%%
%
% aTL_N and cellular : basic definitions
%
%%%%%%%%%%%%%%%
\begin{subsection}{The affine Temperley-Lieb algebra}\label{sec:atln}
This paragraph presents the affine Temperley-Lieb algebra $\atl{N}(\beta)$ in two ways: first, with generators and relations and, second, from a diagrammatic point of view. A family of finite-dimensional representations, the cellular modules $\Cell{N;d,z}$, are also introduced. Our presentation is based on \cite{GL-Aff,MS} (see also \cite{BSAfusion}.)

Fix $N \geq 2$, $q \in \mathbb{C}^{\times}$ and set $\beta = -q-q^{-1}$. The algebra $\atl{N}=\atl N(\beta)$ is defined through generators $\Omega_N,\Omega_N^{-1}$, $\id$ and $e_i$ where $0\leq i\leq N-1$. The associated defining relations depend on the choice of $N$. If $N \geq 3$, they are
\begin{align}\label{eq:genRel}
\begin{split}
e_ie_i &=\beta e_i, \qquad\qquad e_ie_{i\pm 1}e_i=e_i \quad \text{for all }1\leq i\leq N,\qquad\qquad
e_ie_j=e_je_i \quad \text{if }|i-j|\geq 2,\\
\Omega_Ne_i&=e_{i-1}\Omega_N, \qquad\qquad (\Omega_N^{\pm1} e_0)^{N-1}=\Omega_N^{\pm N}(\Omega_N^{\pm1}e_0), \qquad\qquad \Omega_N\Omega_N^{-1}=\Omega_N^{-1}\Omega_N=\id,
\end{split}
\end{align}
where the index on $e_i$ is understood modulo $N$: $e_{i+N}=e_i$. 
Note that the relation $\Omega_Ne_i=e_{i-1}\Omega_N$ implies $e_i=\Omega_N^{1-i}e_1\Omega_N^{i-1}$ and thus $\{e_1,\Omega_N,\Omega_N^{-1}\}$ is a smaller set of generators for $\atl{N}$. Observe also that this same relation shows that the relation $(\Omega_N^{\pm1} e_0)^{N-1}=\Omega_N^{\pm N}(\Omega_N^{\pm1}e_0)$ is equivalent to $(\Omega_N^{\pm1} e_i)^{N-1}=\Omega_N^{\pm N}(\Omega_N^{\pm1}e_i)$ for all $i$'s; this is welcome as that latter relation was the only one where one of the generators, the $e_0$, played a singular role. The defining relations of $\atl 2$ are
$$e_1e_1=\beta e_1, \qquad\qquad e_2=\Omega_2^{\pm 1}e_1\Omega_2^{\mp 1},\qquad\qquad
\Omega_2^2e_1=e_1\Omega_2^2=e_1 \quad \text{and} \quad \Omega_2\Omega_2^{-1}=\Omega_2^{-1}\Omega_2 = \id.$$
The limiting $\atl 1=\mathbb C[\Omega_1,\Omega_1^{-1}]$ is included for completeness.
The algebras $\atl N$ are all unital, associative and infinite-dimen\-sional (note that the subalgebra spanned by $\Omega_N$ and $\Omega_N^{-1}$ is  isomorphic to $\mathbb C[x,x^{-1}]$). The subalgebra generated by $\{\id, e_1,...,e_{N-1}\}$ is isomorphic to the (usual) Temperley-Lieb algebra $\tl{N}(\beta)$ (see \cite{ridout2014standard} for a detailed introduction).

There is another presentation of the algebra $\atl N$ that rests on the concept of {\em $(m,n)$-diagrams}. A $(m,n)$-diagram (with $m\equiv n\modY 2$) is a rectangle with $m$ points marked on its left side and $n$ on its right. The top and bottom sides of the rectangle are identified. The total $m+n$ points are linked pairwise by non-intersecting lines. These lines may cross the top side (or bottom one) to then reenter at the bottom side (or top one). The diagram may also contain closed loops which may be contractible or not (as when the loop crosses the top side an odd number of times). Two diagrams differing only by an isotopy are identified. Here are examples of $(4,2)$-diagrams:
$$\begin{tikzpicture}[baseline={(current bounding box.center)},scale = 1/3]
	\draw[thick] (0,-1) -- (0,4);
	\draw[thick] (3,-1) -- (3,4);
	\draw[dashed] (0,-1) -- (3,-1);
	\draw[dashed] (0,4) -- (3,4);
	\foreach \s in {1,...,4}
	{	
		\filldraw[black] (0,\s -1) circle (3pt);
	};
	\foreach \s in {2,...,3}
	{	
		\filldraw[black] (3,\s -1) circle (3pt);
	};
	\draw (0,1) .. controls (1,1) and (1,2) .. (0,2);
	\draw (0,3) .. controls (1,3) and (2,2) .. (3,2);
	\draw (0,0) .. controls (1,0) and (2,1) .. (3,1);
\end{tikzpicture}
\ ,\qquad\begin{tikzpicture}[baseline={(current bounding box.center)},scale = 1/3]
	\draw[thick] (0,-1) -- (0,4);
	\draw[thick] (3,-1) -- (3,4);
	\draw[dashed] (0,-1) -- (3,-1);
	\draw[dashed] (0,4) -- (3,4);
	\foreach \s in {1,...,4}
	{	
		\filldraw[black] (0,\s -1) circle (3pt);
	};
	\foreach \s in {2,...,3}
	{	
		\filldraw[black] (3,\s -1) circle (3pt);
	};
	\draw (0,1) .. controls (1,1) and (1,2) .. (0,2);
	\draw (0,3) .. controls (1,3) and (2,2) .. (3,2);
	\draw (0,0) .. controls (1,0) and (2,1) .. (3,1);
	\draw (2,-1) arc (0:180:0.5);
	\draw (1,4) arc (180:360:0.5);
	\end{tikzpicture}
\ ,\qquad
	\begin{tikzpicture}[baseline={(current bounding box.center)},scale = 1/3]
	\draw[thick] (0,-1) -- (0,4);
	\draw[thick] (3,-1) -- (3,4);
	\draw[dashed] (0,-1) -- (3,-1);
	\draw[dashed] (0,4) -- (3,4);
	\foreach \s in {1,...,4}
	{	
		\filldraw[black] (0,\s -1) circle (3pt);
	};
	\foreach \s in {2,...,3}
	{	
		\filldraw[black] (3,\s -1) circle (3pt);
	};
	\draw (2,2.25) arc (0:360:0.5);
	\draw (0,0) .. controls (1,0) and (1,1) .. (0,1);
	\draw (3,1) .. controls (2,1) and (2,2) .. (3,2);
	\draw (0,3) .. controls (1,3) and (1.5,3) .. (1.5,4);
	\draw (0,2) .. controls (1,2) and (1.5,0) .. (1.5,-1);
	\end{tikzpicture}
\ ,\qquad
	\begin{tikzpicture}[baseline={(current bounding box.center)},scale = 1/3]
	\draw[thick] (0,-1) -- (0,4);
	\draw[thick] (3,-1) -- (3,4);
	\draw[dashed] (0,-1) -- (3,-1);
	\draw[dashed] (0,4) -- (3,4);
	\foreach \s in {1,...,4}
	{	
		\filldraw[black] (0,\s -1) circle (3pt);
	};
	\foreach \s in {2,...,3}
	{	
		\filldraw[black] (3,\s -1) circle (3pt);
	};
	\draw (0,0) .. controls (1.5,0) and (1.5,3) .. (0,3);
	\draw (0,1) .. controls (1,1) and (1,2) .. (0,2);
	\draw (3,1) .. controls (2.5,1) and (2,0) .. (2,-1);
	\draw (3,2) .. controls (2.5,2) and (2,3) .. (2,4);
	\draw (1.5,-1) -- (1.5,4);
	\end{tikzpicture}
\ .
$$
The {\em rank} $|v|$ of a given diagram $v$ is the minimal number of lines crossing the top side among all diagrams isotopic to $v$. Lines that crosses from left to right are called {\em through lines}. A $(m,n)$-diagram that has $n$ through lines is called {\em monic}. For the examples above, the ranks are respectively $0, 0, 1, 2$ and there are $2, 2, 0, 0$ through lines. Of the four diagrams, only the first two are monic. There are monic $(m,n)$-diagrams if and only if $m\geq n$.

Given a $(m,n)$-diagram $v$ and a $(p,q)$-diagram $w$, we can define the composition $v*w$ if $n=p$. The result is then a $(m,q)$-diagram weighted by a complex number which depends on a parameter $q\in \mathbb C^\times$. More precisely, this composition is obtained by concatenating $v$ to the left of $w$, identifying their sides with $n=p$ points, and removing this common side with the points on it. Contractible loops need also to be removed. The resulting diagram is then multiplied formally by a factor of $\beta=-q-q^{-1}$ for each contractible loop omitted. Here is an example for a $(4,4)$-diagram and a $(4,2)$-diagram.
$$\begin{tikzpicture}[baseline={(current bounding box.center)},scale = 1/3]
	\draw[thick] (0,-1) -- (0,4);
	\draw[thick] (3,-1) -- (3,4);
	\draw[dashed] (0,-1) -- (3,-1);
	\draw[dashed] (0,4) -- (3,4);
	\foreach \s in {1,...,4}
	{	
		\filldraw[black] (0,\s -1) circle (3pt);
		\filldraw[black] (3,\s -1) circle (3pt);
	};
	\draw (0,3) .. controls (1,3) and (1.5,3) .. (1.5,4);
	\draw (3,0) .. controls (2,0) and (1.5,0) .. (1.5,-1);
	\draw (0,2) .. controls (1,2) and (2,3) .. (3,3);
	\draw (0,0) .. controls (1,0) and (1,1) .. (0,1);
	\draw (3,2) .. controls (2,2) and (2,1) .. (3,1);
	\end{tikzpicture}
\ *\ 
\begin{tikzpicture}[baseline={(current bounding box.center)},scale = 1/3]
	\draw[thick] (0,-1) -- (0,4);
	\draw[thick] (3,-1) -- (3,4);
	\draw[dashed] (0,-1) -- (3,-1);
	\draw[dashed] (0,4) -- (3,4);
	\foreach \s in {1,...,4}
	{	
		\filldraw[black] (0,\s -1) circle (3pt);
	};
	\foreach \s in {2,...,3}
	{	
		\filldraw[black] (3,\s -1) circle (3pt);
	};
	\draw (0,3) .. controls (1,3) and (2,1) .. (3,1);
	\draw (3,2) .. controls (2,2) and (1.5,3) .. (1.5,4);
	\draw (0,0) .. controls (1,0) and (1.5,0) .. (1.5,-1);
	\draw (0,2) .. controls (1,2) and (1,1) .. (0,1);
	\end{tikzpicture}
\ =\ 
\begin{tikzpicture}[baseline={(current bounding box.center)},scale = 1/3]
	\draw[thick] (0,-1) -- (0,4);
	\draw[thick] (3,-1) -- (3,4);
	\draw[dashed] (0,-1) -- (3,-1);
	\draw[dashed] (0,4) -- (3,4);
	\foreach \s in {1,...,4}
	{	
		\filldraw[black] (0,\s -1) circle (3pt);
		\filldraw[black] (3,\s -1) circle (3pt);
	};
	\draw (0,3) .. controls (1,3) and (1.5,3) .. (1.5,4);
	\draw (3,0) .. controls (2,0) and (1.5,0) .. (1.5,-1);
	\draw (0,2) .. controls (1,2) and (2,3) .. (3,3);
	\draw (0,0) .. controls (1,0) and (1,1) .. (0,1);
	\draw (3,2) .. controls (2,2) and (2,1) .. (3,1);
	\draw[thick] (3,-1) -- (3,4);
	\draw[thick] (6,-1) -- (6,4);
	\draw[dashed] (3,-1) -- (6,-1);
	\draw[dashed] (3,4) -- (6,4);
	\foreach \s in {1,...,4}
	{	
		\filldraw[black] (3,\s -1) circle (3pt);
	};
	\foreach \s in {2,...,3}
	{	
		\filldraw[black] (6,\s -1) circle (3pt);
	};
	\draw (3,3) .. controls (4,3) and (5,1) .. (6,1);
	\draw (6,2) .. controls (5,2) and (4.5,3) .. (4.5,4);
	\draw (3,0) .. controls (4,0) and (4.5,0) .. (4.5,-1);
	\draw (3,2) .. controls (4,2) and (4,1) .. (3,1);
	\end{tikzpicture}
\ =\ \beta\ 
\begin{tikzpicture}[baseline={(current bounding box.center)},scale = 1/3]
	\draw[thick] (0,-1) -- (0,4);
	\draw[thick] (3,-1) -- (3,4);
	\draw[dashed] (0,-1) -- (3,-1);
	\draw[dashed] (0,4) -- (3,4);
	\foreach \s in {1,...,4}
	{	
		\filldraw[black] (0,\s -1) circle (3pt);
	};
	\foreach \s in {2,...,3}
	{	
		\filldraw[black] (3,\s -1) circle (3pt);
	};
	\draw (0,3) .. controls (1,3) and (2,2) .. (3,2);
	\draw (0,2) .. controls (1,2) and (2,1) .. (3,1);
	\draw (0,0) .. controls (1,0) and (1,1) .. (0,1);
	\end{tikzpicture}
\ .
$$
Let $\Bnu_{m,n}$ be the set of all $(m,n)$-diagrams up to isotopy and without contractible loops. Denote also by $\mathsf D_{m,n}$ the vector space of formal $\mathbb C$-linear combinations of elements of $\Bnu_{m,n}$. Then, the space $\mathsf D_{N,N}$, together with the composition $*$ just introduced, defines a $\mathbb{C}$-algebra $D_{N,N}(\beta)$. Green \cite{Green} showed that this diagrammatic algebra $\mathsf D_{N,N}(\beta)$ and the algebra $\atl N(\beta)$ defined through generators and relations are isomorphic. The isomorphism is 
given through the identification:
\begin{equation}\label{eq:lesGen}
e_1\mapsto\ \begin{tikzpicture}[baseline={(current bounding box.center)},scale = 1/3]
	\draw[thick] (0,-1) -- (0,7);
	\draw[thick] (3,-1) -- (3,7);
	\draw[dashed] (0,-1) -- (3,-1);
	\draw[dashed] (0,7) -- (3,7);
	\foreach \s in {4,...,6}
	{	
		\filldraw[black] (0,\s ) circle (3pt);
		\filldraw[black] (3,\s ) circle (3pt);
	};
	\foreach \s in {0,...,1}
	{	
		\filldraw[black] (0,\s ) circle (3pt);
		\filldraw[black] (3,\s ) circle (3pt);
	};
	\draw (0,5) .. controls (1,5) and (1,6) .. (0,6);
	\draw (3,5) .. controls (2,5) and (2,6) .. (3,6);
	\draw (0,4) -- (3,4);
	\node at (1.5,3) {$\vdots$};
	\draw (0,1) -- (3,1);
	\draw (0,0) -- (3,0);
	\end{tikzpicture}
	\ ,\qquad
\Omega_N\mapsto\ \begin{tikzpicture}[baseline={(current bounding box.center)},scale = 1/3]
	\draw[thick] (0,-1) -- (0,7);
	\draw[thick] (3,-1) -- (3,7);
	\draw[dashed] (0,-1) -- (3,-1);
	\draw[dashed] (0,7) -- (3,7);
	\foreach \s in {4,...,6}
	{	
		\filldraw[black] (0,\s -4) circle (3pt);
		\filldraw[black] (3,\s ) circle (3pt);
	};
	\foreach \s in {0,...,1}
	{	
		\filldraw[black] (0,\s +5) circle (3pt);
		\filldraw[black] (3,\s ) circle (3pt);
	};
	\draw (2,7) arc (180:270:1);
	\draw (1,-1) arc (0:90:1);
	\draw (0,6) .. controls (1,6) and (2,5) .. (3,5);	
	\draw (0,5) .. controls (1,5) and (2,4) .. (3,4);	
	\node at (1.5,3.25) {$\vdots$};
	\draw (0,2) .. controls (1,2) and (2,1) .. (3,1);	
	\draw (0,1) .. controls (1,1) and (2,0) .. (3,0);	
	\end{tikzpicture}
\ ,\qquad
\Omega_N^{-1}\mapsto\ \begin{tikzpicture}[baseline={(current bounding box.center)},scale = 1/3]
	\draw[thick] (0,-1) -- (0,7);
	\draw[thick] (3,-1) -- (3,7);
	\draw[dashed] (0,-1) -- (3,-1);
	\draw[dashed] (0,7) -- (3,7);
	\foreach \s in {4,...,6}
	{	
		\filldraw[black] (0,\s ) circle (3pt);
		\filldraw[black] (3,\s -4) circle (3pt);
	};
	\foreach \s in {0,...,1}
	{	
		\filldraw[black] (0,\s ) circle (3pt);
		\filldraw[black] (3,\s +5) circle (3pt);
	};
	\draw (0,6) arc (270:360:1);
	\draw (3,0) arc (90:180:1);
	\draw (0,5) .. controls (1,5) and (2,6) .. (3,6);	
	\draw (0,4) .. controls (1,4) and (2,5) .. (3,5);	
	\node at (1.5,3.25) {$\vdots$};
	\draw (0,1) .. controls (1,1) and (2,2) .. (3,2);	
	\draw (0,0) .. controls (1,0) and (2,1) .. (3,1);	
	\end{tikzpicture}
\quad\text{and}\quad
\id\mapsto \begin{tikzpicture}[baseline={(current bounding box.center)},scale = 1/3]
	\draw[thick] (0,-1) -- (0,7);
	\draw[thick] (3,-1) -- (3,7);
	\draw[dashed] (0,-1) -- (3,-1);
	\draw[dashed] (0,7) -- (3,7);
	\foreach \s in {5,...,6}
	{	
		\filldraw[black] (0,\s ) circle (3pt);
		\filldraw[black] (3,\s ) circle (3pt);
	};
	\foreach \s in {0,...,1}
	{	
		\filldraw[black] (0,\s ) circle (3pt);
		\filldraw[black] (3,\s ) circle (3pt);
	};
	\draw (0,6) -- (3,6);
	\draw (0,5) -- (3,5);
	\node at (1.5,3.25) {$\vdots$};
	\draw (0,1) -- (3,1);
	\draw (0,0) -- (3,0);
	\end{tikzpicture}
\ .
\end{equation}
Note that using this isomorphism on the set $\{\id,e_1,...,e_{N-1}\}$ gives back the isomorphism described in \cite{ridout2014standard} for the (usual) Temperley-Lieb algebra. We thus recover $\tl{N}(\beta)$ as the subalgebra $\mathbb{C}$-spanned by the $(N,N)$-diagrams of zero rank.

Even though $\atl 0$ will not be used, the generator $\Omega_0$ will. It is the $(0,0)$-diagram
$$\Omega_0=\ \begin{tikzpicture}[baseline={(current bounding box.center)},scale = 1/3]
	\draw[thick] (0,-1) -- (0,3);
	\draw[thick] (3,-1) -- (3,3);
	\draw[dashed] (0,-1) -- (3,-1);
	\draw[dashed] (0,3) -- (3,3);
	\draw (1.5,-1) -- (1.5,3);
	\end{tikzpicture}\ .
$$

The next step is to introduce the cellular modules $\Cell{N;d,z}$. Because composing a $(N,N)$-diagram with a $(N,d)$-diagram gives back a $(N,d)$-diagram, the space $\mathsf D_{N,d}$ is actually a module for the algebra $\mathsf D_{N,N}(\beta)\simeq \atl N(\beta)$. This module is easily seen to be infinite-dimensional. For example,
$$\dots, \quad v\Omega_2^3,\quad v\Omega_2^2,\quad  v\Omega_2=\ \begin{tikzpicture}[baseline={(current bounding box.center)},scale = 1/3]
	\draw[thick] (0,-1) -- (0,4);
	\draw[thick] (3,-1) -- (3,4);
	\draw[dashed] (0,-1) -- (3,-1);
	\draw[dashed] (0,4) -- (3,4);
	\foreach \s in {1,...,4}
	{	
		\filldraw[black] (0,\s -1) circle (3pt);
	};
	\foreach \s in {2,...,3}
	{	
		\filldraw[black] (3,\s -1) circle (3pt);
	};
	\draw (0,0) .. controls (1,0) and (1,1) .. (0,1);
	\draw (0,3) .. controls (1,3) and (2,1) .. (3,1);
	\draw (0,2) .. controls (1,2) and (1.5,0) .. (1.5,-1);
	\draw (3,2) .. controls (2,2) and (1.5,3) .. (1.5,4);
	\end{tikzpicture}
\ ,\quad
v=\ \begin{tikzpicture}[baseline={(current bounding box.center)},scale = 1/3]
	\draw[thick] (0,-1) -- (0,4);
	\draw[thick] (3,-1) -- (3,4);
	\draw[dashed] (0,-1) -- (3,-1);
	\draw[dashed] (0,4) -- (3,4);
	\foreach \s in {1,...,4}
	{	
		\filldraw[black] (0,\s -1) circle (3pt);
	};
	\foreach \s in {2,...,3}
	{	
		\filldraw[black] (3,\s -1) circle (3pt);
	};
	\draw (0,0) .. controls (1,0) and (1,1) .. (0,1);
	\draw (0,3) .. controls (1,3) and (2,2) .. (3,2);
	\draw (0,2) .. controls (1,2) and (2,1) .. (3,1);
	\end{tikzpicture}
\ ,\quad
v\Omega_2^{-1}=\ \begin{tikzpicture}[baseline={(current bounding box.center)},scale = 1/3]
	\draw[thick] (0,-1) -- (0,4);
	\draw[thick] (3,-1) -- (3,4);
	\draw[dashed] (0,-1) -- (3,-1);
	\draw[dashed] (0,4) -- (3,4);
	\foreach \s in {1,...,4}
	{	
		\filldraw[black] (0,\s -1) circle (3pt);
	};
	\foreach \s in {2,...,3}
	{	
		\filldraw[black] (3,\s -1) circle (3pt);
	};
	\draw (0,0) .. controls (1,0) and (1,1) .. (0,1);
	\draw (0,2) -- (3,2);
	\draw (0,3) arc (270:360:1);
	\draw (3,1) .. controls (2,1) and (1.5,0) .. (1.5,-1);
	\end{tikzpicture}
\ ,\quad
v\Omega_2^{-2},\quad v\Omega_2^{-3},\quad  \dots
$$
are all distinct $(4,2)$-diagrams. Defining the modules $\Cell{N;d,z}$ from $\mathsf D_{N,d}$ requires two more steps. First, only the monic diagrams in $\Bnu_{N,d}$ are used. (This limits the range of $d$ to be $0\leq d\leq N$.) However, the composition of a $(N,N)$-diagram $a$ with a monic $(N,d)$-diagram $v$ may result in a $(N,d)$-diagram $a*v$ that is non-monic. The composition is thus modified by setting any non-monic diagram resulting from the $\atl N(\beta)$-action to zero. Unfortunately there are still an infinite number of monic diagrams in $\Bnu_{N,d}$ as the above example shows. This will be settled by the second step. Let $z\in \mathbb{C}^{\times}$ and consider the linear map $f_z$ defined on linear combinations of monic $(N,d)$-diagrams by $f_z(v)=v*(\Omega_d-z\, \id)$ if $d\neq 0$ and by $f_z(v)=v*(\Omega_0-(z+z^{-1})\,\id)$ if $d=0$. As this map acts from the right on $v$, it commutes with the left-action of $\atl N(\beta)$. The cokernel of $f_z$ is then the desired finite-dimensional cellular module $\Cell{N;d,z}$. Remark that all the monic $(4,2)$-diagrams shown in the example above are proportional in $\Cell{4;2,z}$, namely $v\Omega_2^i=z^iv$ for any $i \in \mathbb{Z}$. 

The dimension of the module $\Cell{N;d,z}$ is easily computed. Each monic $(N,d)$-diagram must have $r=(N-d)/2$ loops starting and ending on its left side. Because the number of times the through lines of a diagram cross its top or bottom sides may be ignored in $\Cell{N;d,z}$, the top positions of these $r$ loops completely determine a basis element. Indeed, there is only one way to close these loops on the left side without intersection with either themselves or the through lines. Any choice of these $r$ top positions among the $N$ positions on the left side gives rise to a basis element and
\begin{equation}\label{eq:dimW}\dim\Cell{N;d,z}=\left(\begin{smallmatrix}N\\ r\end{smallmatrix}\right)=\big(\begin{smallmatrix}N\\ (N-d)/2\end{smallmatrix}\big).\end{equation}
This basis just described will be noted $\BCell{N,d}$. It clearly does not depend on $z$. For example, the four elements 
$$\begin{tikzpicture}[baseline={(current bounding box.center)},scale = 1/3]
	\draw[thick] (0,-1) -- (0,4);
	\draw[thick] (3,-1) -- (3,4);
	\draw[dashed] (0,-1) -- (3,-1);
	\draw[dashed] (0,4) -- (3,4);
	\foreach \s in {1,...,4}
	{	
		\filldraw[black] (0,\s -1) circle (3pt);
	};
	\foreach \s in {2,...,3}
	{	
		\filldraw[black] (3,\s -1) circle (3pt);
	};
	\draw (0,2) .. controls (1,2) and (1,3) .. (0,3);
	\draw (0,1) .. controls (1,1) and (2,2) .. (3,2);
	\draw (0,0) .. controls (1,0) and (2,1) .. (3,1);
	\end{tikzpicture}\ ,
\qquad
\begin{tikzpicture}[baseline={(current bounding box.center)},scale = 1/3]
	\draw[thick] (0,-1) -- (0,4);
	\draw[thick] (3,-1) -- (3,4);
	\draw[dashed] (0,-1) -- (3,-1);
	\draw[dashed] (0,4) -- (3,4);
	\foreach \s in {1,...,4}
	{	
		\filldraw[black] (0,\s -1) circle (3pt);
	};
	\foreach \s in {2,...,3}
	{	
		\filldraw[black] (3,\s -1) circle (3pt);
	};
	\draw (0,1) .. controls (1,1) and (1,2) .. (0,2);
	\draw (0,3) .. controls (1,3) and (2,2) .. (3,2);
	\draw (0,0) .. controls (1,0) and (2,1) .. (3,1);
	\end{tikzpicture}\ ,
\qquad
\begin{tikzpicture}[baseline={(current bounding box.center)},scale = 1/3]
	\draw[thick] (0,-1) -- (0,4);
	\draw[thick] (3,-1) -- (3,4);
	\draw[dashed] (0,-1) -- (3,-1);
	\draw[dashed] (0,4) -- (3,4);
	\foreach \s in {1,...,4}
	{	
		\filldraw[black] (0,\s -1) circle (3pt);
	};
	\foreach \s in {2,...,3}
	{	
		\filldraw[black] (3,\s -1) circle (3pt);
	};
	\draw (0,0) .. controls (1,0) and (1,1) .. (0,1);
	\draw (0,3) .. controls (1,3) and (2,2) .. (3,2);
	\draw (0,2) .. controls (1,2) and (2,1) .. (3,1);
	\end{tikzpicture}
\ ,
\qquad
\begin{tikzpicture}[baseline={(current bounding box.center)},scale = 1/3]
	\draw[thick] (0,-1) -- (0,4);
	\draw[thick] (3,-1) -- (3,4);
	\draw[dashed] (0,-1) -- (3,-1);
	\draw[dashed] (0,4) -- (3,4);
	\foreach \s in {1,...,4}
	{	
		\filldraw[black] (0,\s -1) circle (3pt);
	};
	\foreach \s in {2,...,3}
	{	
		\filldraw[black] (3,\s -1) circle (3pt);
	};
	\draw (0,1) -- (3,1);
	\draw (0,2) -- (3,2);
	\draw (0,3) arc (270:360:1);
	\draw (1,-1) arc (0:90:1);
	\end{tikzpicture}
$$
form a basis of $\Cell{4;2,z}$. Note that the third diagram is the $v$ of the previous example. All the other $v\Omega_2^i$ with $i\in\mathbb Z$ would fall in the same class in $\text{coker} f_z$. Similarly, the other elements of the basis depicted here were chosen as the element in their classes in $\text{coker} f_z$ with the smallest rank. This choice will always be implied when the basis $\BCell{N;d}$ is used.

One of Graham and Lehrer's important contributions \cite{GL-Aff} was to show that all finite-dimensional simple $\atl N$-modules arise as quotients $\Cell{N;d,z}/\rad\Cell{N;d,z}$ for some $z\in \mathbb{C}^{\times}$ and $d$ satisfying $0\leq d\leq N$ with $d\equiv N\modY 2$. A complete list of non-isomorphic simple modules requires the introduction of the two sets alluded to in the section on main results:
\begin{equation}\label{eq:lesLambdas}\Lambda=\mathbb{Z}_{\geq 0}\times \mathbb{C}^{\times}\qquad \text{and}\qquad \Lambda_N=\{(d,z)\in\Lambda\,|\, d\leq N,\, d\equiv N\modY 2\}\end{equation}
where, in $\Lambda$, the pairs $(0,z)$ and $(0,z^{-1})$ are identified for every $z\in \mathbb C^\times$. Moreover the pairs $(0,q)$ and $(0,q^{-1})$ are removed from $\Lambda_N$ when both $N$ is even and $q+q^{-1}=0$, that is, $q=i$ or $-i$. These two are the {\em problematic pairs}.
Then, theorem \ref{thm:listOfSimples} provides the complete list of non-isomorphic finite-dimensional simple modules of $\atl{N}=\atl N(\beta)$, namely $\{\Irre{N;d,z}\,|\, (d,z)\in\Lambda_N\}$  where $\Irre{N;d,z}$ is the (simple) head of the cellular module $\Cell{N;d,z}$.
\end{subsection}

%%%%%%%%%%%%%%%
%
% XXZ comme aTL_N-module
%
%%%%%%%%%%%%%%%
\begin{subsection}{The action of $\atl N$ on $\XXZ N\pm$}\label{sec:atlOnXXZ}

The periodic XXZ spin-$\frac12$ chain that will be considered here was introduced by Pasquier and Saleur \cite{PS} as a limit of the six-vertex model (see also \cite{AGR}). It is described by a Hamiltonian $H_\text{XXZ}:\XXZ N{}\to\XXZ N{}$ where $\XXZ N{}=(\mathbb C^2)^{\otimes N}$. This Hamiltonian is expressed as a sum of nearest-neighbor interactions $e_i^+$ or $e_i^-$ as
\begin{equation}\label{eq:hamiltonian}H_\text{XXZ}=\sum_{i=1}^Ne_i^+=\sum_{i=1}^Ne_i^-
\end{equation}
where, for fixed parameters $q,z\in\mathbb C^\times$,
\begin{align*}\label{eq:lesEi}
e_i^\pm&=\sigma_i^-\sigma_{i+1}^+ +\sigma_i^+\sigma_{i+1}^-+(q+q^{-1})\sigma_i^+\sigma_i^-\sigma_{i+1}^+\sigma_{i+1}^- -q^{\pm 1} \sigma_i^+\sigma_i^- -q^{\mp 1} \sigma_{i+1}^+\sigma_{i+1}^- \quad  \text{for }1\leq i\leq N-1,\\
e_N^\pm&=z^2\sigma_N^-\sigma_{1}^+ +z^{-2}\sigma_N^+\sigma_{1}^-+(q+q^{-1})\sigma_N^+\sigma_N^-\sigma_{1}^+\sigma_{1}^- -q^{\pm 1} \sigma_N^+\sigma_N^- -q^{\mp 1} \sigma_{1}^+\sigma_{1}^-, 
\end{align*} 
and $\sigma_i^s=I_2\otimes\dots \otimes I_2\otimes \sigma^s\otimes I_2\otimes \dots \otimes I_2$ with $I_2$ the $2\times 2$ identity matrix, $s\in\{+, -, z\}$ and $\sigma^s$ the usual Pauli matrix. Note that, while the $e_i^+$ and $e_i^-$ differ by the factors of $q$ in the last two terms, they lead to the same Hamiltonian as \eqref{eq:hamiltonian} states. Note also that the appearance of the twist $z^2$ ($=e^{i\phi}$ in \cite{PS}) only in $\sigma_N^\pm$ raises the question of translational invariance of the spin chain. However, a simple change of basis\footnote{Such a basis was used in \cite{MDSA}. The twist there is parametrized by $v$ instead of $z$ (with any $v$ such that $v^N=z^{-1}$). If $T$ ($\tilde T$) represents an endomorphism of $\XXZ N{}$ in the basis used here (used in \cite{MDSA} resp.), then $T=O\tilde T\mathcal O^{-1}$ where $O=v^{\sum_{1\leq j\leq N}j\sigma_j^z}$.}
can be used to restore the explicit invariance. Here, we choose the basis used in \cite{PS} as it simplifies many expressions and as the set $\{e_i^+,1\leq i\leq N-1\}$ (or $\{e_i^-,1\leq i\leq N-1\}$) generates the usual action of the (regular) algebra $\tl N(\beta=-q-q^{-1})$ on $\XXZ N{}$. This in turn leads to simpler expressions for the action of the quantum group $\luszt$ on $\XXZ N{}$ when $q$ is a root of unity (see section \ref{sec:luqsl2}).

It is a well-known (and remarkable) fact that both sets $\{e_i^+,1\leq i\leq N\}$ and $\{e_i^-,1\leq i\leq N\}$ satisfy all the defining relations of $\atl N = \atl N(\beta)$ where $\Omega_N$ and $\Omega_N^{-1}$ do not appear. The following expressions for the latter generators
\begin{equation*}
\Omega_N=tz^{-\sigma_1^z}\qquad\text{and}\qquad \Omega_N^{-1}=t^{-1}z^{\sigma_N^z}
\end{equation*}
complete\footnote{A proof that the sets $\{e_i^\pm, \Omega_N, \Omega_N^{-1}\}$ given above do define $\atl N$-modules can be found in \cite{MDSA}.} the description of two representations $\XXZ{N;z}+$ and $\XXZ{N;z}-$ of $\atl N$ on $\XXZ N{}=(\mathbb C^2)^{\otimes N}$. Here, $t$ is the left translation of spins along the chain. More precisely, if one uses the standard basis $\{|+\rangle,|-\rangle\}$ for the space $\mathbb C^2$ and the induced basis $\BXXZ N=\{|x_1x_2\dots x_N\rangle\,|\, x_1,x_2,\dots,x_N\in \{+,-\}\}$ for the tensor product $\XXZ N{}$, then $t$ is defined by $t|x_1x_2\dots x_N\rangle=|x_2\dots x_Nx_1\rangle$. Even though the sets $\{e_i^+\}$ and $\{e_i^-\}$ lead to the same Hamiltonian $H_\text{XXZ}$, the modules $\XXZ N+$ and $\XXZ N-$ may not be isomorphic (in fact, they are typically not!). Moreover, changing $z$ to $-z$ provides two other representations for the same $H_\text{XXZ}$.

The expressions of the generators $e_i^\pm$ and $\Omega_N$ make it clear that they commute with the overall spin $S^z=\frac12\sum_{1\leq i\leq N}\sigma_i^z$. The module $\XXZ{N;z}\pm$ thus decomposes as the direct sum
$$\XXZ{N;z}\pm=\mybigoplus{d=-N}N \XXZ{N;d,z}\pm$$
where $\XXZ{N;d,z}\pm$ is the eigenspace of $S^z$ with eigenvalue $\frac d2$ and where $\myoplus{}{}$ denotes a sum with steps of $2$. 

The spin flip $s:\XXZ N{}\to\XXZ N{}$ is defined as usual on the basis $\BXXZ N$ by $s|x_1x_2\dots x_N\rangle=|(-x_1)(-x_2)\dots(-x_N)\rangle$. It clearly maps the eigenspace of $S^z$ with eigenvalue $\frac d2$ onto the eigenspace with eigenvalue $-\frac d2$.
\begin{proposition}\label{prop:spinflip} The spin flip $s$ defines a $\atl N$-isomorphism between $\XXZ{N;d,z}+$ and $\XXZ{N;-d,z^{-1}}-$.
\end{proposition}
\begin{proof} It is sufficient to check that $s$ intertwines the actions of $e_1$ and $\Omega_N$ on the two modules. The simple formula
\begin{equation}\label{eq:e1}e_1^\pm|x_1x_2\dots x_N\rangle=\delta_{x_1+x_2,0}(|x_2x_1x_3\dots x_N\rangle-q^{\pm x_1}|x_1x_2\dots x_N\rangle)\end{equation}
can be verified starting from the definition of the $e_i^\pm$. Then, 
\begin{align*}
e_1^- s|x_1x_2\dots x_N\rangle&=e_1^-|(-x_1)(-x_2)\dots(-x_N)\rangle\\
	&=\delta_{x_1+x_2,0}(|(-x_2)(-x_1)(-x_3)\dots(-x_N)\rangle-q^{x_1}|(-x_1)(-x_2)\dots(-x_N)\rangle)\\
	&=s(\delta_{x_1+x_2,0}(|x_2x_1x_3\dots x_N\rangle-q^{x_1}|x_1x_2\dots x_N\rangle))=se_1^+|x_1x_2\dots x_N\rangle.
\end{align*}
Similarly (here, it is useful to add symbols on the basis elements to remind oneself which action is at play),
\begin{align*}
\Omega_{N}s|x_1x_2\dots x_N\rangle_{z}^+
	&=\Omega_{N}|(-x_1)(-x_2)\dots(-x_N)\rangle_{z^{-1}}^-=z^{x_1}|(-x_2)\dots(-x_N)(-x_1)\rangle_{z^{-1}}^-\\
	&=z^{x_1}s|x_2\dots x_Nx_1\rangle_z^+= s\Omega_{N}|x_1x_2\dots x_N\rangle_z^+
\end{align*}
which ends the proof.
\end{proof}
\end{subsection}

%%%%%%%%%%%%%%%
%
% GL morphisms
%
%%%%%%%%%%%%%%%
\begin{subsection}{Morphisms between cellular modules}\label{sec:morphismesGL}

Theorem \ref{thm:GL} revealed the structure of the cellular modules $\Cell{N;d,z}$. The main tool in the proof of this theorem  is the existence of an inclusion of the cellular $\Cell{N;t,x}$ into $\Cell{N;d,z}$ whenever $(t,x)\in\lambda_N$ succeeds $(d,z)$. The explicit expression of this morphism is known \cite{GL-Aff} and will be used in paragraph \ref{sec:ind}. Appendix \ref{app:a} contains the definition of the $q$-numbers $[n]_q$ and the $q$-factorials $[n]_q!$ with their basic properties.

For $w\in\BCell{t,d}$, denote by $\text{beg }w\subseteq\{1,2,\dots, t\}$ the set of positions for the beginnings of the loops on the left side of $w$ (with the highest point on that side corresponding to position 1). For each $s\in\text{beg }w$, let also $\#_s$ be the number of loops in the convex hull of the loop starting at $s$, including that loop. Finally, let $h_w$ be the rational function 
$$h_w(q)=\frac{[r]_q!}{\prod_{s\in\text{beg }w}[\#_s]_q}$$
where $r=\frac12(t-d)$. 
\begin{proposition}[\cite{GL-Aff}]\label{prop:grandThm} Fix $(d,z)$ and $(t,x)$ two pairs in $\lambda_N$ such that the latter succeeds the former. Set $a=1$ ($a=-1$) if $(t,x)$ succeeds $(d,z)$ through condition A (resp.~condition B). Then, $\gl{}=\gl{(d,z);(t,x)}:\Cell{N;t,x}\to\Cell{N;d,z}$ given by
\begin{equation}\label{eq:gl}
\gl{(d,z);(t,x)}(v)=\sum_{w\in\BCell{t,d}}q^{ai_w}z^{r-|w|}h_w(q)vw
\end{equation}
is an injective $\atl N$-linear map. Here, $r=\frac12(t-d)$, $|w|$ is the rank of $w$ and $2i_w=t(|w|-r)+\frac d2(t+1)-\zeta_w$ with $\zeta_w$ the sum of the positions (on the left side of $w$) of the through lines.
\end{proposition}
Several important results, beside theorem \ref{thm:GL}, are tied to this family of inclusions. First, $\Hom(\Cell{N;t,x},\Cell{N;d,z})\simeq \mathbb C$ if $(d,z)\preceq (t,x)$ and is $0$ otherwise. Second, $\Irre{N;t,x}$ is a composition factor of $\Cell{N;d,z}$ if and only if $(d,z)\preceq (t,x)$ with $(t,x)\in \Lambda_N$. Finally, the multiplicity of $\Irre{N;t,x}$ in $\Cell{N;d,z}$ is at most one. All these results are proved in \cite{GL-Aff}.
\end{subsection}

%%%%%%%%%%%%%%%
%
% duality : aTL, simple* and XXZ*
%
%%%%%%%%%%%%%%%
\begin{subsection}{Dualities}\label{sec:dualities}
Let $\mathsf A$ be an algebra over $\mathbb C$. If $\mathsf M$ is a left $\mathsf A$-module, then the space $\Hom_{\mathbb C}(\mathsf M,\mathbb C)$ of linear functionals on $\mathsf M$ inherits a natural right $\mathsf A$-module structure through $(fa)(m)\overset{\text{def}}{=}f(am)$ for $f\in\Hom_{\mathbb C}(\mathsf M,\mathbb C)$, $a\in\mathsf A$ and $m\in \mathsf M$. Left- and right-modules are different algebraic structures and it is in general impossible to relate them, for example with morphisms. Fortunately, the algebra $\atl{N}$ enjoys an additional structure that allows to define left-modules from right ones. This structure is the anti-involution $\dagger: \atl N\to\atl N^\text{op}$ that takes an $(N,N)$-diagram $v$ to its reflection $v^\dagger$ with respect to a vertical line. Note that this procedure is defined for any $(N,d)$-diagram and produces a bijection $\dagger : \mathsf D_{N,d}\to \mathsf D_{d,N}$. The anti-involution obtained for $N=d$ allows for a left $\atl N$-action to be defined on $\Hom_{\mathbb C}(\mathsf M,\mathbb C)$, namely
$$(a\star f)(m)=f(a^\dagger m), \qquad a\in \atl N,\, f\in\Hom_{\mathbb C}(\mathsf M,\mathbb C),\, m\in \mathsf M.$$
It is indeed a proper action:
$$(a\star(b\star f))(m)=(b\star f)(a^\dagger m)=f(b^\dagger a^\dagger m)=f((ab)^\dagger m)=((ab)\star f)(m).$$
This left-module $(\Hom_{\mathbb C}(\mathsf M,\mathbb C), \star)$ will be denoted $\mathsf M^\star$.

This correspondence $\star$ of left $\atl N$-modules to left $\atl N$-modules can be extended to morphisms between them. If $\varphi\in\Hom_{\atl N}(\mathsf M,\mathsf N)$, then $\varphi^\star\in\Hom_{\atl N}(\mathsf N^\star,\mathsf M^\star)$ is defined as $[\varphi^\star(h)](m)=h\circ \varphi(m)$ for $h\in\mathsf N^\star$ and $m\in \mathsf M$. Clearly $[\varphi^\star(a\star h)](m)=(a\star h)(\varphi(m))=h(a^\dagger \varphi(m))=h\circ\varphi(a^\dagger m)=[a\star(h\circ\varphi)](m)=[a\star(\varphi^\star(h))](m)$ for all $h\in \mathsf N^\star$, $m\in \mathsf M$, $a\in\atl N$ and $\varphi^\star\in\Hom_{\atl N}(\mathsf N^\star,\mathsf M^\star)$. This shows that $\star:\modY \atl N\to\modY \atl N$ is a contravariant endofunctor of the category $\text{mod } \atl N$ of finitely-generated left $\atl N$-modules. One can further show that it is a duality in the categorical sense, that is, there exists a functorial isomorphism $\Psi:\id_{\modY \atl N}\overset{\sim}\longrightarrow \star^2$. As a consequence, $\star$ is an exact functor and a short exact sequence $0\to \mathsf M\to\mathsf N\to\mathsf P\to 0$ induces another short exact sequence $0\to \mathsf P^\star\to\mathsf N^\star\to\mathsf M^\star\to 0$ that splits if and only if the original one does. For the purpose of this text, the key feature of $\star$ is that the Loewy diagram of $\mathsf M^\star$ is exactly the one obtained by flipping all arrows in the diagram of $\mathsf M$ while replacing its composition factors by their $\star$-duals. For example, here are possible Loewy diagrams for a pair of $\star$-dual left-modules:
$$\mathsf M\simeq\ 
\begin{tikzpicture}[baseline={(current bounding box.center)},scale = 1/3]
    \node (i1) at (0,0) {$\mathsf I_1$};
    \node (i2) at (3,2) {$\mathsf I_2$};
    \node (i3) at (6,0) {$\mathsf I_3$};
    \node (i4) at (9,2) {$\mathsf I_4$};
    \node (i5) at (12,0) {$\mathsf I_5$};
	\draw[->] (i2) -- (i1); \draw[->] (i2) -- (i3);
	\draw[->] (i4) -- (i5); \draw[->] (i4) -- (i3);	
	\end{tikzpicture}
\qquad\text{and}\qquad
\mathsf M^\star\simeq\ 
\begin{tikzpicture}[baseline={(current bounding box.center)},scale = 1/3]
    \node (i1) at (0,2) {$\mathsf I_1^\star$};
    \node (i2) at (3,0) {$\mathsf I_2^\star$};
    \node (i3) at (6,2) {$\mathsf I_3^\star$};
    \node (i4) at (9,0) {$\mathsf I_4^\star$};
    \node (i5) at (12,2) {$\mathsf I_5^\star$};
	\draw[<-] (i2) -- (i1); \draw[<-] (i2) -- (i3);
	\draw[<-] (i4) -- (i5); \draw[<-] (i4) -- (i3);	
	\end{tikzpicture}\ .
$$
The proof that the functor $\star$ is a duality and exact will be omitted. They are straightforward exercises in category theory.

Our next step is to identify the images under $\star$ of the simple module $\Irre{N;d,z}$ and of the eigenspace $\XXZ{N;d,z}\pm\subseteq\XXZ{N;z}\pm$. A bilinear form $\langle\ ,\ \rangle=\langle\ ,\ \rangle_{N;d,z}:\Cell{N;d,z}\times \Cell{N;d,z^{-1}}\to \mathbb C$ will play a role toward this goal \cite{GL-Aff}. It is defined diagrammatially using the bijection $\dagger$ on $(N,d)$-diagrams. If $w$ and $v$ are $(N,d)$-diagrams in $\BCell{N,d}$, then the composition $v^\dagger w$, as defined in paragraph \ref{sec:atln}, is a $(d,d)$-diagram. If $v^\dagger w$ does not have $d$ through lines, then $\langle w, v\rangle $ is set to be $0$. If $v^\dagger w$ has $d$ through lines, then there exist two integers $n\in\mathbb Z_{\geq0}$ and $m\in\mathbb Z$ that describe it completely, that is $v^\dagger w$ is a certain power $\Omega_d^m$ plus $n$ closed contractible loops. In this case, $\langle w, v\rangle $ is defined to be
$$\langle w, v\rangle =\begin{cases}\beta^nz^m,&\text{if }d\neq0,\\
	\beta^n(z+z^{-1})^m,&\text{if }d=0.
\end{cases}$$
The following drawings give examples for pairs of $(5,1)$- and $(6,0)$-diagrams:
$$\langle w,v\rangle=\Big\langle \ \begin{tikzpicture}[baseline={(current bounding box.center)},scale = 1/3]
	\draw[thick] (0,-1) -- (0,5);
	\draw[thick] (3,-1) -- (3,5);
	\draw[dashed] (0,-1) -- (3,-1);
	\draw[dashed] (0,5) -- (3,5);
	\foreach \s in {0,...,4}
	{	
		\filldraw[black] (0,\s ) circle (3pt);
	};
	\filldraw[black] (3,2) circle (3pt);
	\draw (0,4) .. controls (1,4) and (2,2) .. (3,2);
	\draw (0,0) .. controls (1.5,0) and (1.5,3) .. (0,3);
	\draw (0,1) .. controls (1,1) and (1,2) .. (0,2);
	\end{tikzpicture}
\ ,\ \begin{tikzpicture}[baseline={(current bounding box.center)},scale = 1/3]
	\draw[thick] (0,-1) -- (0,5);
	\draw[thick] (3,-1) -- (3,5);
	\draw[dashed] (0,-1) -- (3,-1);
	\draw[dashed] (0,5) -- (3,5);
	\foreach \s in {0,...,4}
	{	
		\filldraw[black] (0,\s ) circle (3pt);
	};
	\filldraw[black] (3,2) circle (3pt);
	\draw (0,4) arc (270:360:1);
	\draw (3,2) .. controls (2,2) and (1.5,0) .. (1.5,-1);
	\draw (0,2) .. controls (1,2) and (1,3) .. (0,3);
	\draw (0,0) .. controls (1,0) and (1,1) .. (0,1);
	\end{tikzpicture}\ 
\Big\rangle= \beta z\qquad\text{since}\qquad
v^\dagger w=\ \begin{tikzpicture}[baseline={(current bounding box.center)},scale = 1/3]
	\draw[thick] (-3,-1) -- (-3,5);
	\draw[thick] (0,-1) -- (0,5);
	\draw[thick] (3,-1) -- (3,5);
	\draw[dashed] (-3,-1) -- (3,-1);
	\draw[dashed] (-3,5) -- (3,5);
	\foreach \s in {0,...,4}
	{	
		\filldraw[black] (0,\s ) circle (3pt);
	};
	\filldraw[black] (3,2) circle (3pt);
	\filldraw[black] (-3,2) circle (-3pt);
	\draw (-1,5) arc (180:270:1);
	\draw (-3,2) .. controls (-2,2) and (-1.5,0) .. (-1.5,-1);
	\draw (0,2) .. controls (-1,2) and (-1,3) .. (0,3);
	\draw (0,0) .. controls (-1,0) and (-1,1) .. (0,1);
	\draw (0,4) .. controls (1,4) and (2,2) .. (3,2);
	\draw (0,0) .. controls (1.5,0) and (1.5,3) .. (0,3);
	\draw (0,1) .. controls (1,1) and (1,2) .. (0,2);
	\end{tikzpicture}
\ = \ 
\begin{tikzpicture}[baseline={(current bounding box.center)},scale = 1/3]
	\draw[thick] (0,-1) -- (0,1);
	\draw[thick] (3,-1) -- (3,1);
	\draw[dashed] (0,-1) -- (3,-1);
	\draw[dashed] (0,1) -- (3,1);
	\filldraw[black] (3,0) circle (3pt);
	\filldraw[black] (0,0) circle (-3pt);
	\draw (2,1) arc (180:270:1);
	\draw (1,-1) arc (0:90:1);
	\draw (2,0) arc (0:360:0.5);
	\end{tikzpicture}
\ = \beta z
$$ 
and
$$\langle w',v'\rangle=\Big\langle \ \begin{tikzpicture}[baseline={(current bounding box.center)},scale = 1/3]
	\draw[thick] (0,-1) -- (0,6);
	\draw[thick] (3,-1) -- (3,6);
	\draw[dashed] (0,-1) -- (3,-1);
	\draw[dashed] (0,6) -- (3,6);
	\foreach \s in {0,...,5}
	{	
		\filldraw[black] (0,\s ) circle (3pt);
	};
	\draw (0,5) arc (270:360:1);
	\draw (1,-1) arc (0:90:1);
	\draw (0,2) .. controls (1,2) and (1,1) .. (0,1);
	\draw (0,3) .. controls (1,3) and (1,4) .. (0,4);
	\end{tikzpicture}
\ ,\ \begin{tikzpicture}[baseline={(current bounding box.center)},scale = 1/3]
	\draw[thick] (0,-1) -- (0,6);
	\draw[thick] (3,-1) -- (3,6);
	\draw[dashed] (0,-1) -- (3,-1);
	\draw[dashed] (0,6) -- (3,6);
	\foreach \s in {0,...,5}
	{	
		\filldraw[black] (0,\s ) circle (3pt);
	};
	\draw (0,0) .. controls (1.5,0) and (1.5,5) .. (0,5);
	\draw (0,2) .. controls (1,2) and (1,1) .. (0,1);
	\draw (0,3) .. controls (1,3) and (1,4) .. (0,4);
	\end{tikzpicture}\ 
\Big\rangle= \beta^2 (z+z^{-1})\qquad\text{since}\qquad
v'^\dagger w'=\ \begin{tikzpicture}[baseline={(current bounding box.center)},scale = 1/3]
	\draw[thick] (-3,-1) -- (-3,6);
	\draw[thick] (0,-1) -- (0,6);
	\draw[thick] (3,-1) -- (3,6);
	\draw[dashed] (-3,-1) -- (3,-1);
	\draw[dashed] (-3,6) -- (3,6);
	\foreach \s in {0,...,5}
	{	
		\filldraw[black] (0,\s ) circle (3pt);
	};
	\draw (0,0) .. controls (-2,0) and (-2,5) .. (0,5);
	\draw (0,2) .. controls (-1,2) and (-1,1) .. (0,1);
	\draw (0,3) .. controls (-1,3) and (-1,4) .. (0,4);
	\draw (0,5) arc (270:360:1);
	\draw (1,-1) arc (0:90:1);
	\draw (0,2) .. controls (1,2) and (1,1) .. (0,1);
	\draw (0,3) .. controls (1,3) and (1,4) .. (0,4);
	\end{tikzpicture}
\ = \ 
\begin{tikzpicture}[baseline={(current bounding box.center)},scale = 1/3]
	\draw[thick] (0,-1) -- (0,1);
	\draw[thick] (3,-1) -- (3,1);
	\draw[dashed] (0,-1) -- (3,-1);
	\draw[dashed] (0,1) -- (3,1);
	\draw (1,-1) -- (1,1);
	\draw (2, -0.5) arc (0:360:0.25);
	\draw (2,0.5) arc (0:360:0.25);
	\end{tikzpicture}
\ = \beta^2 (z+z^{-1}).
$$
Note that the factors $z^m$ and $(z+z^{-1})^m$ are consistent with the definition of the cellular module $\Cell{N;d,z}$ as the cokernel of $f_z$. Note also that $\langle\ ,\ \rangle$ is never identically zero if $(d,z)\in\Lambda_N$. Indeed, it is clearly non-zero when $d=N$. For $0<d<N$, there are $w$ and $v$ whose pairing is $1$: draw a through line coming from the left, crisscrossing the line between $v^\dagger$ and $w$ as many times needed to draw $(N-d)/2$ arcs on both sides, and continue the line to the right side before filling up the remaining positions of both $w^\dagger$ and $v$ with horizontal through lines (an example is given below for $N=5$ and $d=1$). Finally, for $\langle\ ,\ \rangle$ to vanish when $d=0$, both numbers $\beta=-q-q^{-1}$ and $z+z^{-1}$ need to be zero. That forces both $q$ and $z$ to be either $i$ or $-i$, that is, $(d,z)$ should be $(0,q)$ or $(0,q^{-1})$ with $q+q^{-1}=0$, precisely the problematic pairs that are removed from the set $\Lambda_N$ labelling the simple modules $\Irre{N;d,z}$ (see the lines after \eqref{eq:lesLambdas}).
$$\begin{tikzpicture}[baseline={(current bounding box.center)},scale = 1/3]
  \draw[thick] (-3,-1) -- (-3,5);
	\draw[thick] (0,-1) -- (0,5);
	\draw[thick] (3,-1) -- (3,5);
	\draw[dashed] (0,-1) -- (3,-1);
	\draw[dashed] (0,5) -- (3,5);
	\draw[dashed] (-3,-1) -- (0,-1);
	\draw[dashed] (-3,5) -- (0,5);
	\foreach \s in {0,...,4}
	{	
		\filldraw[black] (0,\s ) circle (3pt);
	};
	\filldraw[black] (3,2) circle (3pt);
	\filldraw[black] (-3,2) circle (3pt);
	\draw (0,4) .. controls (1,4) and (2,2) .. (3,2);
	\draw (-3,2) .. controls (-2,2) and (-1,0) .. (0,0);
	\draw (0,0) .. controls (1.5,0) and (1.5,3) .. (0,3);
	\draw (0,1) .. controls (-1.5,1) and (-1.5,4) .. (0,4);
	\draw (0,1) .. controls (1,1) and (1,2) .. (0,2);
	\draw (0,2) .. controls (-1,2) and (-1,3) .. (0,3);
	\end{tikzpicture}$$
The diagrammatic definition of $\langle\ ,\ \rangle$ and the associativity of the composition of diagrams give $\langle w,a^\dagger v\rangle=\langle a w,v\rangle$ for any $a\in\atl N$. The bilinear form is thus said to be {\em invariant}. Even though this bilinear form $\langle\ ,\ \rangle = \langle \ ,\ \rangle_{N;d,z}$ is never zero when $(d,z)\in\Lambda_N$, it may have a radical, that is, $\rad\langle\ ,\ \rangle_{N;d,z}=\{w\in\Cell{N;d,z}\ |\ \langle w,v\rangle_{N;d,z}=0\text{ for all }v\in\Cell{N;d,z^{-1}}\}$ might be larger than $\{0\}$. The invariance of $\langle \ ,\ \rangle_{N;d,z}$ proves that $\rad\langle\ ,\ \rangle_{N;d,z}$ is a submodule of $\Cell{N;d,z}$. The following result ties this radical with the Jacobson radical $\rad \Cell{N;d,z}$, that is the intersection of all maximal submodules of $\Cell{N;d,z}$.
\begin{proposition}[\cite{GL-Aff}]\label{prop:Radrad} If $(d,z)\in\Lambda_N$, the radical $\rad \langle\ ,\ \rangle_{N;d,z}$ is equal to the Jacobson radical $\rad \Cell{N;d,z}$ of $\Cell{N;d,z}$.
\end{proposition}
This result leads to the identification of the $\star$-dual of the simple $\atl N$-modules and of the submodules $\XXZ{N;d,z}\pm\subseteq\XXZ{N;z}\pm$.
\begin{proposition}\label{thm:isoByDual} (a) The modules $(\XXZ{N;d,z}{\pm})^\star$ and $\XXZ{N;d,z^{-1}}\pm$ are isomorphic if $z\in\mathbb C^\times$ and $-N\leq d\leq N$ with $d\equiv N\modY 2$.

\noindent(b) The modules $\Irre{N;d,z}^\star$ and $\Irre{N;d,z^{-1}}$ are isomorphic if $(d,z)\in\Lambda_N$.
\end{proposition}
\begin{proof} (a) Let $(\BXXZ N)^\star=\{\langle y_1y_2\dots y_N|\ \ |\ y_1,...,y_N\in\{+,-\}\}$ be the basis dual to $\BXXZ N=\{| x_1x_2\dots x_N\rangle \ |\ x_1,...x_N\in \{+,-\}\}$, that is $\langle y_1y_2\dots y_N|x_1x_2\dots x_N\rangle=\prod_{1\leq i\leq N}\delta_{x_i,y_i}$. The left-right reflections of $e_1^\pm$ and $\Omega_N$ are simply given by $(e_1^\pm)^\dagger=e_1^\pm$ and $\Omega_N^\dagger=\Omega_N^{-1}$. The $\star$-action of $e_1^\pm$ on $\langle y_1y_2\dots y_N|$, seen as linear functional in $\Hom_{\mathbb C}(\XXZ N{},\mathbb C)$, is computed using \eqref{eq:e1}:
\begin{align*}
(e_1^\pm \star \langle y_1y_2\dots y_N|)(|x_1x_2\dots x_N\rangle)
	&=\langle y_1y_2\dots y_N|((e_1^\pm)^\dagger |x_1x_2\dots x_N\rangle)\\
	&=\langle y_1y_2\dots y_N|(e_1^\pm |x_1x_2\dots x_N\rangle)\\
	&=\langle y_1y_2\dots y_N|[\delta_{x_1+x_2,0}(|x_2x_1x_3\dots x_N\rangle-q^{\pm x_1}|x_1x_2\dots x_N\rangle)]\\
	&=\delta_{y_1+y_2,0}(\langle y_2y_1y_3\dots y_N|-q^{\pm y_1}\langle y_1y_2\dots y_N|)|x_1x_2\dots x_N\rangle.
\end{align*}
Similarly, the $\star$-action of $\Omega_n$ is
\begin{align*}
(\Omega_N \star \langle y_1y_2\dots y_N|)(|x_1x_2\dots x_N\rangle)
	&=\langle y_1y_2\dots y_N|((\Omega_N)^\dagger |x_1x_2\dots x_N\rangle)\\
	&=\langle y_1y_2\dots y_N|(\Omega_N^{-1} |x_1x_2\dots x_N\rangle)\\
	&=\langle y_1y_2\dots y_N|z^{x_N}|x_Nx_1\dots x_2\rangle\\
	&=z^{x_N}\langle y_2\dots y_Ny_1|x_1x_2\dots x_N\rangle.
\end{align*}
The $\atl{N}$-action on the eigenspace $\XXZ{N;d,z^{-1}}{\pm}$ is thus directly recovered if the element $\langle y_1y_2\dots y_N|$ of the dual basis $(\BXXZ N)^\star$ is identified to the element $|x_1x_2\dots x_N\rangle$ of $\BXXZ N$.

\noindent (b) Let $(\BCell{N,d})^*$ be the basis dual to the basis $\BCell{N,d}$ of $\Cell{N;d,z}$, that is, $w^*(v) = \delta_{w,v}$ for any linear functional $w^*\in(\BCell{N,d})^*$ and any $v\in\BCell{N,d}$. Define the $\mathbb C$-linear map $\psi:\Cell{N;d,z^{-1}}\to \Cell{N;d,z}^\star$ by $\psi(w)=\sum_v \langle w,v\rangle_{N;d,z^{-1}}v^*$ where the sum (and all others in this proof) is over $\BCell{N,d}$. It is non-zero since $\langle\ ,\ \rangle = \langle\ ,\ \rangle_{N;d,z^{-1}}\neq 0$. It is also $\atl N$-linear. Indeed, if $a\in\atl N$, $w\in \Cell{N;d,z^{-1}}$ and $u\in\BCell{N,d}$, then
\begin{align*}
[\psi(aw)](u)&=\sum_v\langle aw,v\rangle v^*(u)=\langle aw,u\rangle\qquad \text{since }v^*(u)=\delta_{v,u},\\
	&=\langle w,a^\dagger u\rangle\qquad \text{by the invariance of }\langle\ ,\ \rangle.
\intertext{If $\gamma_v\in\mathbb C$ are the coefficients of $a^\dagger u$ in the basis $\BCell{N,d}$, then}
	&=\langle w,\sum_v\gamma_v v\rangle=\sum_{v,x}\langle w,x\rangle x^*(\gamma_v v)
	=\sum_x\langle w,x\rangle x^*(a^\dagger u)\\
	&=[\psi(w)](a^\dagger u)=[a\star \psi(w)](u).
\end{align*}
Finally, $\psi(w)=0$ if and only if $\langle w,v\rangle=0$ for all $v\in\Cell{N;d,z}$, that is if and only if $w\in\rad \langle \, ,\, \rangle_{N;d,z^{-1}}=\rad \Cell{N;d,z^{-1}}$ by the last proposition. Thus, $\ker \psi=\rad \Cell{N;d,z^{-1}}$ and $\im\psi\subseteq \Cell{N;d,z}^\star$ must be isomorphic to $\Cell{N;d,z^{-1}}/\ker\psi=\Irre{N;d,z^{-1}}$. The image of $\psi$ is hence irreducible and must lie in the socle of $\Cell{N;d,z}^\star$. As the duality $\star$ flips all arrows in Loewy diagrams and preserves irreducibility, it exchanges the socle and head of modules. Accordingly, theorem \ref{thm:GL} gives in this case $\soc\Cell{N;d,z}^\star\simeq (\head\Cell{N;d,z})^\star=\Irre{N;d,z}^\star$ so that $\Irre{N;d,z^{-1}}\simeq \Irre{N;d,z}^\star$ by irreducibility of $\Irre{N;d,z}^\star$.
\end{proof}
There is yet another duality that stems from the symmetry of the defining relations \eqref{eq:genRel}: these relations are clearly invariant under the exchange of generators $e_i\mapsto e_{N-i}$ and $\Omega_N^{\pm 1}\mapsto\Omega_N^{\mp 1}$. This correspondence defines an automorphism $\Phi$ of the algebra $\atl N$. If one sees the indices of the generators $e_i$'s as labelling evenly-spaced dots on a circle, this map amounts to a reflection of the circle onto itself that keeps the site $N$ fixed. The diagrammatic definition of $\atl N$ offers another concrete visualisation of this automorphism: it is simply the top-to-bottom reflection of all diagrams. This map $\Phi:\atl N\to\atl N$ is of interest here as it can twist the action on a module to define a new one. More precisely, if $\mathsf M$ is a left-module, then $(a,m) \mapsto \Phi(a)m$ defines another left-$\atl{N}$-module $\mathsf M^\circ$ over the vector space underlying $\mathsf M$. The endofunctor $\circ:\modY \atl N\to \modY \atl N$ acting on modules as $\mathsf M\mapsto \mathsf M^\circ$ and as the identity on morphisms is covariant, exact and its square is clearly the identity. It thus preserves short exact sequences (and their split or non-split character) and the irreducibility of modules. In particular, the Loewy diagram of $\mathsf M^\circ$ may be obtained from the one of $\mathsf M$ by simply replacing the composition factors by their $\circ$-duals (without inverting the arrows).
\begin{proposition}\label{prop:foncteurcirc}(a) The modules $(\XXZ{N;d,z}{\pm})^\circ$ and $\XXZ{N;d,z^{-1}}\mp$ are isomorphic for $z\in\mathbb C^\times$ and $-N\leq d\leq N$ with $d\equiv N\modY 2$.

\noindent(b) The modules $\Cell{N;d,z}^\circ$ and $\Cell{N;d,z^{-1}}$ are isomorphic if $(d,z)\in\Lambda_N$ and so are the simple ones $\Irre{N;d,z}^\circ$ and $\Irre{N;d,z^{-1}}$.
\end{proposition}
\begin{proof} (a) Let $\varrho:\XXZ{N;d,z^{-1}}{\pm}\to(\XXZ{N;d,z}{\mp})^{\circ}$ be given by $\varrho(|x_1x_2\dots x_N\rangle_{z^{-1}}^{\pm})=|x_N\dots x_2 x_1\rangle_{z}^{\mp}$ (here we add symbols on the basis elements as in the proof of proposition \ref{prop:spinflip}). This is clearly an invertible $\mathbb C$-linear map. It is also a $\atl N$-morphism. Only the most delicate verification, namely $\Phi(e_1) \varrho=\varrho e_1$, is provided:
\begin{align*}
\Phi(e_1) \varrho |x_1x_2\dots x_N\rangle_{z^{-1}}^{\pm} 
	&= e_{N-1}|x_N\dots x_2 x_1\rangle_{z}^{\mp},\qquad \text{definition of }\Phi(e_1)=e_{N-1},\\
	&= \Omega_N^2e_1\Omega_N^{-2}|x_N\dots x_2 x_1\rangle_{z}^{\mp}\,\qquad \text{defining relations \eqref{eq:genRel}},\\
	&= \Omega_N^2e_1 \Omega_N^{-1}z^{x_1}|x_1x_N\dots x_2 \rangle_{z}^{\mp}=\Omega_N^2e_1 z^{x_1+x_2}|x_2x_1x_Nx_{N-1}\dots x_3 \rangle_{z}^{\mp}, \qquad \text{action of }\Omega_N^{-1},\\
	&= \delta_{x_1+x_2,0}z^{x_1+x_2}\Omega_N^2(|x_1x_2x_Nx_{N-1}\dots x_3\rangle_{z}^{\mp} -q^{\mp x_2}|x_2x_1x_Nx_{N-1}\dots x_3 \rangle_{z}^{\mp}), \qquad \text{by }\eqref{eq:e1},\\
	&= \delta_{x_1+x_2,0}z^{x_1+x_2}z^{-x_1-x_2}(|x_Nx_{N-1}\dots x_3x_1x_2\rangle_{z}^{\mp} -q^{\pm x_1}|x_N\dots x_2x_1\rangle_{z}^{\mp}),\qquad \text{action of }\Omega_N,\\
	&= \varrho(\delta_{x_1+x_2,0}(|x_2x_1x_3\dots x_N\rangle_{z^{-1}}^{\pm} -q^{\pm x_1}|x_1x_2\dots x_N\rangle_{z^{-1}}^{\pm}))=\varrho e_1|x_1x_2\dots x_N\rangle_{z^{-1}}^{\pm}.
\end{align*}
(b) The isomorphism $\Cell{N;d,z}^\circ\simeq\Cell{N;d,z^{-1}}$ follows from the visualisation of the automorphism $\Phi$ as flipping all diagrams vertically and is left to the reader. The other isomorphism then follows from the first one by the fact that $\circ$ preserves Loewy diagrams, so that $\Irre{N;d,z^{-1}}=\head \Cell{N;d,z^{-1}}\simeq \head (\Cell{N;d,z})^\circ = \Irre{N;d,z}^{\circ}$.
\end{proof}
There are now two dualities, the contravariant $\star$ and the covariant $\circ$. Their action on $\modY \atl N$ commute in the sense that $(\mathsf M^\star)^\circ\simeq (\mathsf M^\circ)^\star$ if $\mathsf M\in\modY \atl N$. Indeed, first note that $\Phi(a)^\dagger=\Phi(a^\dagger)$ if $a\in \atl{N}$ (again it suffices to check this on $e_1$ and $\Omega_N$). Then, for $a\in\atl N, m^*\in\Hom(\mathsf M,\mathsf N)$ and $m'\in\mathsf M$, we get $(\Phi(a)m^*)(m')=m^*(\Phi(a)^\dagger m')=m^*(\Phi(a^\dagger)m')$.

We now prove a part of theorem \ref{thm:main}. We say that two Loewy diagrams are {\it reciprocal} if one of these diagrams can be deduced from the other by flipping all the arrows without modifying the composition factors.
\begin{corollary}\label{cor:doubleDual} Let $\mathsf{M}$ be a $\atl{N}$-module. Then, the Loewy diagram of $(\mathsf{M}^{*})^{\circ}$ is reciprocal to that of $\mathsf{M}$. In particular, the Loewy diagrams of $\XXZ{N;d,z}+$ and $\XXZ{N;d,z}-$ are reciprocal whenever $z \in \mathbb{C}^{\times}$ and $-N \leq d\leq N$ with $d\equiv N \modY 2$.
\end{corollary}
\begin{proof} By precedent results, the Loewy diagram of $(\mathsf{M}^{*})^{\circ}$ may be obtained from that of $\mathsf{M}$ by flipping all the arrows while applying successively the endofunctors $\star$ and $\circ$ to the composition factors: $\Irre{N;d,z}\mapsto(\Irre{N;d,z}^*)^{\circ}$. The first statement of the corollary then follows from propositions \ref{thm:isoByDual} and \ref{prop:foncteurcirc} as these propositions give $(\Irre{N;d,z}^*)^{\circ}\simeq \Irre{N;d,z^{-1}}^{\circ}\simeq \Irre{N;d,z}$. The second statement also follows from the first as $((\XXZ{N;d,z}+)^*)^{\circ}\simeq (\XXZ{N;d,z^{-1}}+)^{\circ} \simeq \XXZ{N;d,z}-$ again by propositions \ref{thm:isoByDual} and \ref{prop:foncteurcirc}.
\end{proof}
\end{subsection}

%%%%%%%%%%%%%%%
%
% i_nd
%
%%%%%%%%%%%%%%% 
\begin{subsection}{Morphism between cellular modules and $\XXZ{N;d,z}\pm$}\label{sec:ind} Let $(d,z)\in \lambda_N$. A remarkable morphism $\ay{N;d,z}:\Cell{N;d,z}\to \XXZ{N;d,z}{+}$ 
between cellular modules and XXZ-eigenspaces was introduced in \cite{MDSA}. An element $w\in\BCell{N,d}$ can be associated with the $r=\frac12(N-d)$ positions $i_1,i_2,\dots, i_r$ where its $r$ arcs begin and by the positions $j_1,j_2,\dots, j_r$ where they end. Each $j_n$ is taken to satisfy $i_n<j_n\leq N+i_n-1$ and, clearly, the arc crosses the boundary if and only if $j_n>N$. The set $\psi(w)=\{(i_n,j_n),1\leq n\leq r\}$ describes completely $w$. The morphism $\ay{N;d,z}$ acts on it as\footnote{As noted in paragraph \ref{sec:atlOnXXZ}, the basis used in \cite{MDSA} is different from the one used here. Consequently, the expression of $\ay{N;d,z}$ given there is also different.}
$$\ay{N;d,z}(w)=\prod_{(i,j)\in\psi(w) \atop j\leq N}(u^{-1}\sigma_i^-+u\sigma_j^-)\cdot
\prod_{(i,j)\in\psi(w) \atop j> N}(zu^{-1}\sigma_i^-+z^{-1}u\sigma_j^-)|0\rangle$$
where $u=(-q)^{\frac12}$ and $|0\rangle=|++\cdots+\rangle$. (The two choices for the square root $(-q)^{\frac12}$ may lead to two distinct morphisms. However, the morphisms differ in this case only by an overall sign.) The $z$ in $\ay{N;d,z}$ will be omitted when there is no risk of confusion. An analogous morphism $\Cell{N;d,z}\to\XXZ{N;d,z}-$ exists but will not be used in this paper.
\begin{theorem}[\cite{MDSA}]\label{thm:mdsa} The map $\ay{N;d,z}$ is one-to-one if and only if $(d,z)$ has no strict
successor $(s,y)$ via condition A in $\lambda_N$.
\end{theorem}
\noindent Recall that the dimensions of $\Cell{N;d,z}$ and $\XXZ{N;d,z}{+}$ are equal. Hence, if $\ay{N;d,z}$ is not one-to-one, its kernel must be non-zero. However, even though $\ay{N;d,z}$ may be singular, it is never trivial.
\begin{lemma}\label{thm:iNonZero} The morphism $\ay{N;d,z}$ is non-zero.
\end{lemma}
\begin{proof} If $d=N$, then $\dim\Cell{N;N,z}=\dim\XXZ{N;N,z}+=1$ and the only element $w$ of $\BCell{N,N}$ is such that $\psi(w)=\emptyset$. The morphism $\ay{N;N,z}$ then maps $w$ on $|0\rangle$. If $d<N$ and thus $r=\frac12(N-d)\geq 1$, there is at least one arc. Consider the vector $w$ in $\BCell{N,d}$ whose $r$ arcs are all nested at the topmost positions, that is $\psi(w)=\{(1,2r),(2,2r-1),\dots , (r,r+1)\}$. Its image by $\ay{N;d,z}$ contains the vector $|--\cdots-++\cdots +\rangle$ (where all the ``$-$" are at the beginning) with coefficient $u^{-r}=(-q)^{-r/2}\neq 0$. \end{proof}
Because the morphism $\gl{(d,z);(t,x)}$ is injective, the next proposition shows that, if $q^2$ is an $\ell$-th primitive root of unity,  
then, whenever the pair $(d,z)$ has an immediate successor $(t,x)$ through condition B with $t-d<2\ell$, the image of $\ay{N;d,z}$ contains at least the two composition factors $(d,z)$ and $(t,x)$ (if, of course, $t\leq N$) . 
\begin{proposition}\label{thm:Bsurvit} Let $q^2$ be an $\ell$-th primitive root of unity. Let $(d,z)\in \lambda_N$ be a pair with a direct successor $(t,x)$ through condition B with $t-d<2\ell$ and $t\leq N$. Then, the morphism $\ay{N;d,z}\circ\gl{(d,z);(t,x)}:\Cell{N;t,x}\to \XXZ{N;d,z}{+}$ is non-zero.
\end{proposition}
\begin{proof} The strategy of the proof is to choose two elements, one $w$ in $\BCell{N,t}$, the other $v$ in $\BXXZ{N,d}$, such that the component of $v$ in $\ay{N,d}\circ\gl{(d,z);(t,x)}(w)$ is non-zero. With the notation
$$r_t=\frac12(N-t),\quad r_d=\frac12(N-d)\quad \text{and}\quad k=\frac12(t-d)<\ell$$
that will be used throughout this proof, our choices are the following:
$$\begin{tikzpicture}[baseline={(current bounding box.center)},scale=0.5]
\draw (3,0) -- (19,0);
\draw[dashed] (3,2.5) -- (3,0) ;
\draw[dashed] (19,2.5) -- (19,0) ;
 \draw (3,2.5) -- (19,2.5);
 \draw[dotted] (-0.75,-1.7) -- (19,-1.7);
% tout d'abord le w
\node at (0.75,1.25) {$w\in\BCell{N,t}$};
\draw[line width=0.3mm] (4,0) .. controls(4,2)and(9,2) .. (9,0);
\node at (5,0.35) {$\dots$}; \node at (8,0.35) {$\dots$};
\draw[line width=0.3mm] (7,0) arc (0:180:0.5);
\node at (4,-0.6) {$_1$};
\draw[dotted, line width=0.3mm] (6,0) -- (6,-0.8); \node at (6,-1) {$_{r_t}$};
\node at (7,-0.5) {$_{r_t+1}$};
\draw[dotted, line width=0.3mm] (9,0) -- (9,-1); \node at (9,-1.2) {$_{2r_t}$};
\draw[line width=0.3mm] (10,0) -- (10,2.5);
\draw[line width=0.3mm] (12,0) -- (12,2.5);
\node at (11,0.35) {$\dots$};
\node at (10,-0.5) {$_{2r_t+1}$};
\draw[dotted, line width=0.3mm] (12,0) -- (12,-1); \node at (12,-1.2) {$_{2r_t+i-1}$};
\draw[line width=0.3mm] (13,0) -- (13,2.5);
\node at (14,0.35) {$\dots$};
\draw[line width=0.3mm] (15,0) -- (15,2.5);
\draw[line width=0.3mm] (16,0) -- (16,2.5);
\node at (17,0.35) {$\dots$};
\draw[line width=0.3mm] (18,0) -- (18,2.5);
\node at (13,-0.5) {$_{2r_t+i}$};
\draw[dotted, line width=0.3mm] (15,0) -- (15,-1); \node at (15,-1.2) {$_{2r_t+k}$};
\node at (16,-0.5) {$_{\ \ 2r_t+k+1}$};
\draw[dotted, line width=0.3mm] (18,0) -- (18,-1); \node at (18,-1.2) {$_{N}$};
% et enfin l'objet du desir : v
\node at (0.75,-2.5) {$v\in\BXXZ{N,d}$};
\node at (3.5,-2.5) {$|$};
\node at (4,-2.5) {$-$};\node at (5,-2.5) {$\dots$};\node at (6,-2.5) {$-$};
\node at (5,-3) {$\underbrace{\phantom{MMMM}}_{r_t}$};
\node at (7,-2.5) {$+$};\node at (8,-2.5) {$\dots$};\node at (9,-2.5) {$+$};
\node at (10,-2.5) {$-$};\node at (11,-2.5) {$\dots$};\node at (12,-2.5) {$-$};
\node at (13,-2.5) {$-$};\node at (14,-2.5) {$\dots$};\node at (15,-2.5) {$-$};
\node at (12.5,-3) {$\underbrace{\phantom{MMMMMMMMI}}_{k}$};
\node at (16,-2.5) {$+$};\node at (17,-2.5) {$\dots$};\node at (18,-2.5) {$+$};
\node at (18.5,-2.5) {$\rangle$};
\end{tikzpicture}$$
We have rotated the $(N,t)$-diagram $w$ so that what usually appears at the top is now on the left. With the explicit form of $\gl{}$ given in \eqref{eq:gl}, the image of $w$ under study reads
$$\ay{N,d}\circ\gl{(d,z);(t,x)}(w)=\sum_{y\in\BCell{t,d}}q^{-i_y}z^{k-|y|}h_{y}(q)\ay{N,d}(wy).$$
Note that the composition of $w$ on the right by a $y\in \BCell{t,d}$ will not change the $r_t$ nested arcs in $w$. Note also that the minus signs in $v$ come from the action of $\sigma^-_i$ and $\sigma^-_j$ in the products $\prod_{j\leq N}(u^{-1}\sigma_i^-+u\sigma_j^-)\cdot\prod_{j> N}(zu^{-1}\sigma_i^-+z^{-1}u\sigma_j^-)|0\rangle$ which are over all pairs $(i,j)\in\psi(wy)$. The $r_t$ leftmost ``$-$'' in $v$ will thus come from the beginning $i$ of the arcs $(i,j)$ in $\psi(w)=\{(1,2r_t),(2,2r_t-1),\dots , (r_t,r_t+1)\}\subseteq\psi(wy)$. The $k$ following ``$-$" will need to come from arcs created by the composition with the $(t,d)$-diagram $y$. 

$$\begin{tikzpicture}[baseline={(current bounding box.center)},scale=0.5]
\draw (3,0) -- (26,0); \draw (3,2.5) -- (26,2.5);
\draw (3,5) -- (26,5);
\draw[dotted] (-0.75,-1.7) -- (26,-1.75);
\draw[dashed] (3,0) -- (3,2.5);
\draw[dashed] (26,0) -- (26,2.5);
\draw[dashed] (3,2.5) -- (3,5);
\draw[dashed] (26,2.5) -- (26,5);
% tout d'abord le w
\node at (0.75,1.25) {$w\in\BCell{N,t}$};
\draw[line width=0.3mm] (4,0) .. controls(4,2)and(9,2) .. (9,0);
\node at (5,0.35) {$\dots$}; \node at (8,0.35) {$\dots$};
\draw[line width=0.3mm] (7,0) arc (0:180:0.5);
\node at (4,-0.6) {$_1$};
\draw[dotted, line width=0.3mm] (6,0) -- (6,-0.8); \node at (6,-1) {$_{r_t}$};
\node at (7,-0.5) {$_{r_t+1}$};
\draw[dotted, line width=0.3mm] (9,0) -- (9,-1); \node at (9,-1.2) {$_{2r_t}$};
\draw[line width=0.3mm] (10,0) -- (10,2.5);
\draw[line width=0.3mm] (12,0) -- (12,2.5);
\node at (11,0.35) {$\dots$};
\node at (10,-0.5) {$_{2r_t+1}$};
\draw[dotted, line width=0.3mm] (12,0) -- (12,-1); \node at (12,-1.2) {$_{2r_t+i-1}$};
\draw[line width=0.3mm] (13,0) -- (13,2.5);
\node at (14,0.35) {$\dots$};
\draw[line width=0.3mm] (15,0) -- (15,2.5);
\draw[line width=0.3mm] (16,0) -- (16,2.5);
\node at (17,0.35) {$\dots$};
\draw[line width=0.3mm] (18,0) -- (18,2.5);
\node at (13,-0.5) {$_{2r_t+i}$};
\draw[dotted, line width=0.3mm] (15,0) -- (15,-1); \node at (15,-1.2) {$_{2r_t+k}$};
\node at (16,-0.5) {$_{\ \ 2r_t+k+1}$};
\draw[dotted, line width=0.3mm] (18,0) -- (18,-1); \node at (18,-1.2) {$_{2r_d-i+1}$};
\draw[line width=0.3mm] (19,0) -- (19,2.5);
\node at (20,0.35) {$\dots$};
\draw[line width=0.3mm] (21,0) -- (21,2.5);
\draw[line width=0.3mm] (22,0) -- (22,2.5);
\node at (23,0.35) {$\dots$};
\draw[line width=0.3mm] (24,0) -- (24,2.5);
\node at (19,-0.5) {$_{\ \ 2r_d-i+2}$};
\draw[dotted, line width=0.3mm] (21,0) -- (21,-1); \node at (21,-1.2) {$_{N-i+1}$};
\node at (22,-0.5) {$_{N-i}$};
\draw[dotted, line width=0.3mm] (24,0) -- (24,-1); \node at (24,-1.2) {$_{N}$};
% et maintenant le y
\node at (0.75,3.75) {$y\in\BCell{t,d}$};
\draw[line width=0.3mm] (10,2.5) .. controls(10,3.5)and(5,3.5) .. (3,3.5);
\draw[line width=0.3mm] (12,2.5) .. controls(12,4)and(5,4) .. (3,4);
\node at (11,2.85) {$\dots$}; 
\draw[line width=0.3mm] (13,2.5) .. controls(13,4.5)and(18,4.5) .. (18,2.5);
\node at (14,2.85) {$\dots$}; 
\draw[line width=0.3mm] (16,2.5) arc (0:180:0.5);
\node at (17,2.85) {$\dots$}; 
\draw[line width=0.3mm] (19,2.5) -- (19,5);
\node at (20,2.85) {$\dots$};
\draw[line width=0.3mm] (21,2.5) -- (21,5);
\draw[line width=0.3mm] (22,2.5) .. controls(22,4)and(25,4) .. (26,4);
\node at (23,2.85) {$\dots$};
\draw[line width=0.3mm] (24,2.5) .. controls(24,3.5)and(25,3.5) .. (26,3.5);
% et enfin l'objet du desir : v
\node at (0.75,-2.5) {$v\in\BXXZ{N,d}$};
\node at (3.5,-2.5) {$|$};
\node at (4,-2.5) {$-$};\node at (5,-2.5) {$\dots$};\node at (6,-2.5) {$-$};
\node at (5,-3) {$\underbrace{\phantom{MMMM}}_{r_t}$};
\node at (7,-2.5) {$+$};\node at (8,-2.5) {$\dots$};\node at (9,-2.5) {$+$};
\node at (10,-2.5) {$-$};\node at (11,-2.5) {$\dots$};\node at (12,-2.5) {$-$};
\node at (13,-2.5) {$-$};\node at (14,-2.5) {$\dots$};\node at (15,-2.5) {$-$};
\node at (12.5,-3) {$\underbrace{\phantom{MMMMMMMMI}}_{k}$};
\node at (16,-2.5) {$+$};\node at (17,-2.5) {$\dots$};\node at (18,-2.5) {$+$};
\node at (19,-2.5) {$+$};\node at (20,-2.5) {$\dots$};\node at (21,-2.5) {$+$};
\node at (22,-2.5) {$+$};\node at (23,-2.5) {$\dots$};\node at (24,-2.5) {$+$};
\node at (24.5,-2.5) {$\rangle$};
\end{tikzpicture}$$

Which $y\in\BCell{t,d}$ can contribute a non-zero coefficient to $v$? The above diagram contains the answer. In this diagram, the chosen $w$ and $v$ have been reproduced without changes, together with a possible candidate for $y$. (Some through lines that were included in the $\dots$ have been added to $w$ and similarly for some ``$+$'' in the $v$.) Independently of the candidate drawn, a $y$ will contribute a non-zero coefficient to $v$ only if it can add arcs to $w$ with beginnings or ends at all positions from $2r_t+1$ to $2r_t+k$, hereafter refered to as the ``desired'' positions. The number $k$ of ``$-$'' to be created is equal to the number of arcs in $y\in\BCell{t,d}$ and therefore each new arc, that is, each arc that was not already in $w$, will have to produce one of these ``$-$''. Moreover, if such a new arc had both its beginning and end among the desired positions, a minus sign would be missing in $v$ as only one of the two terms in $(u^{-1}\sigma_i^-+u\sigma_j^-)$ (or $(zu^{-1}\sigma_i^-+z^{-1}u\sigma_j^-)$) can be used to obtain a given element in $\BXXZ{N,d}$. Thus, all arcs in $y$ must start or end at a desired position, but must not have both their beginning and end there. Note finally that, if an arc in $y$ ends at a desired position and crosses the left boundary, all the other arcs ending on its left in $y$ must also cross this boundary. Similarly, if an arc in $y$ begins at a desired position and closes to its right without crossing the right boundary, then all other arcs starting at a desired position on its right in $y$ will also close to their right without crossing this boundary. In other words, $y$ can contribute to a non-zero component along $v$ if and only if there is a $i$, with $1\leq i\leq k+1$, such that $y$ has $i-1$ contiguous arcs ending at desired positions and crossing the left boundary with $k+1-i$ ones starting at desired positions and closing to their right without reaching the boundary. The diagram shows such a $y$ and there are precisely $k+1$ one them. They will be labeled $y_i$ (for $1\leq i\leq k+1$).

The rest of the proof simply studies their contribution. The exercise consists in computing the rank $|y_i|$, the exponent $i_{y_i}$ of $q$, the weight $h_{y_i}(q)$ and the factor of $q$ and $z$ stemming from the products $\prod_{(i,j)\in\psi(wy_i)}$. The rank $|y_i|$, that is the number of arcs in $y_i$ that cross the boundary, is $i-1$ and the sum $\zeta_{y_i}$ of the positions of the through lines of $y_i$ is easily computed to be $\frac d2(2t-d+3)-di$. The exponent $i_{y_i}$ of $q$ may thus be expressed, after some simplifications, as
\begin{align*}
2i_{y_i}&=t(i-1-\frac12(t-d))+\frac d2(t+1)-\frac d2(2t-d+3)+di = (t+d)(i-(k+1))
\end{align*}
The diagram for $y_i$ also allows for the computation of the rational function $h_{y_i}(q)$:
$$h_{y_i}(q)=\left[\begin{matrix}k\\i-1\end{matrix}\right]_q.$$
At last, to compute the coefficient of $v$, only the terms in the products $\prod_{(i,j)\in\psi(wy_i)}$ that create a minus sign at a desired position need to be accounted. This means the beginning of the $r_t$ arcs of $w$ (that give a factor $(u^{-1})^{r_t}$), the end of the arcs in $y_i$ that cross the boundary (with a factor of $(z^{-1}u)^{i-1}$) and the beginning of the last $k+1-i$ arcs in $y_i$ (with a factor $(u^{-1})^{k+1-i}$). These three contributions amount to a factor of $z^{-(i-1)}u^{2(i-1)-r_d}$. Using the fact that $u$ is $(-q)^{\frac12}$, the sum of the contributions of the $y_i$ can hence be written as
$$(\text{coefficient of $v$ in $\ay{N,d}\circ\gl{(d,z);(t,x)}(w)$})=
q^{\frac12(t+d)(k+1)}(-q)^{-\frac12r_d-1}z^{k+2}\sum_{i=1}^{k+1}
q^{-\frac{i}{2}(t+d)}(-q)^i z^{-2i}\left[\begin{matrix}k\\i-1\end{matrix}\right]_q.$$
The factor outside the sum is non-zero. The coefficient of $v$ will thus be non-zero if and only if the sum is. The hypothesis that $(d,z)\preceq (t,x)$ through condition B is finally used here. This hypothesis forces $q^tz^2=1$ and the sum becomes
$$\sum_{i=1}^{k+1}(-1)^iq^{(k+1)i}\left[\begin{matrix}k\\i-1\end{matrix}\right]_q=
-q^{k+1}\prod_{i=1}^k(1-q^{2i})$$
where the $q$-binomial theorem \ref{thm:qBinome} has been used. Since by hypothesis $k<\ell$ and $q^2$ is a $\ell$-th primitive root of unity, none of the factors in this last product vanishes and the coefficient of $v$ in $\ay{N,d}\circ\gl{(d,z);(t,x)}(w)$ is indeed non-zero.
\end{proof}

Remark that the hypothesis of $q$ being a specific root of unity is used only at the end to assert the non-vanishing of the product $\prod_{1\leq i\leq k}(1-q^{2i})$. If $q$ is {\em not} a root of unity, this product is clearly non-zero. Thus:
\begin{proposition}\label{thm:BsurvitGen} Let $q\in\mathbb C^\times$ be generic, that is, not a root of unity. If $(d,z)\in\lambda_N$ has a direct successor $(t,x)$ through condition $B$ with $t\leq N$, then the morphism $\ay{N;d,z}\circ\gl{(d,z);(t,x)}:\Cell{N;t,x}\to \XXZ{N;d,z}{+}$ is non-zero.
\end{proposition}
The two propositions \ref{thm:Bsurvit} and \ref{thm:BsurvitGen} seem to give a privileged role to condition B compared to condition A. However, similar statements hold, with B replaced by A and $q$ by $q^{-1}$. Indeed, if $q$ is replaced by $q^{-1}$ in the definition of both conditions and if $\preceq_{q^{-1}}$ denotes the new partial order thus defined, then a quick check gives the following lemma.
\begin{lemma}\label{thm:partialOrders} Let $(d,z), (t,x)\in\lambda$. The following statements are equivalent.

\noindent(a) $(d,z)\preceq (t,x)$ directly through condition A (or B);

\noindent(b) $(d,z^{-1})\preceq (t,x^{-1})$ directly through condition B (resp.~A);

\noindent(c) $(d,z)\preceq_{q^{-1}} (t,x)$ directly through condition B (resp.~A).
\end{lemma}
Theorem \ref{thm:mdsa} with propositions \ref{thm:Bsurvit} and \ref{thm:BsurvitGen} underline the role played by the successors of $(d,z)\in \lambda_N$ through either condition A or B in the description of the image and kernel of $\ay{N;d,z}$. Let the {\em generic part} $\pg{N;d,z}$ of $\Cell{N;d,z}$ be 
\begin{equation*}
\pg{N;d,z}=\begin{cases}
	\Cell{N;d,z}/\im \gl{(d,z);(s,y)}& \text{if a direct successor $(s,y)$ of $(d,z)$ through condition A exists,}\\
	\Cell{N;d,z} & \text{if such a successor $(s,y)$ through A does not exist.}
\end{cases}
\end{equation*}
\begin{corollary}\label{thm:gpEstUnQuotient} Let $q,z\in\mathbb C^\times$ and $0\leq d\leq N$ with $d\equiv_2N$. Then $\pg{N;d,z}$ is isomorphic to a quotient of $\im \ay{N;d,z}$.
\end{corollary}
\begin{proof}
Assume that $(d,z)$ is not a problematic pair. Then, theorem \ref{thm:GL} shows that the generic part $\pg{N;d,z}$ has at most two composition factors. If $\pg{N;d,z}$ is simple, then $\pg{N;d,z}\simeq \Irre{N;d,z} = \head \Cell{N;d,z}$ is clearly a quotient of $\im \ay{N;d,z}$ as $\ay{N;d,z}$ is a non-zero homomorphism by lemma \ref{thm:iNonZero}. 

If $\pg{N;d,z}$ rather has two composition factors, then its Loewy diagram must be of the form
$$ (d,z)\rightarrow (t,x) $$
where $(t,x)\in \lambda_N$ is a direct successor of $(d,z)$ through condition B. If $q$ is generic, then $\pg{N;d,z}$ is isomorphic to $\Cell{N:d,z}$ by proposition \ref{thm:BsurvitGen} and the result follows. Suppose then that $q^2$ is an $\ell$th-primitive root of unity. We can also assume that $t \geq d+2\ell$ as the desired result is a direct consequence of %follows from
 proposition \ref{thm:Bsurvit} when $t<d+2\ell$. In this case, $t$ must be equal to $d+2\ell$ to be a direct successor and we must have $q^d = q^{t-2\ell} = z^{-2}$. The structure of $\Cell{N;d,z}$ is thus given by the Loewy diagram (i) or (ii) of theorem \ref{thm:GL} (with the direct successor $(s,y)$ of $(d,z)$ through %via
  condition A being $(d+2\ell,zq^\ell)$ in the case of diagram (i) and being the second highest node in the case of diagram (ii)) and the generic part $\pg{N;d,z}$ must then be isomorphic to the simple module $\Irre{N;d,z}$. We can therefore rule out this case as it was already treated above. 

The above discussion left only the case of the problematic pairs to be studied. This is done in appendix \ref{app:problematic}.
\end{proof}

We shall show later that the generic part is indeed the image of $\ay{N;d,z}$. Toward this goal, the remaining of this section is devoted to the exploration of the structure of the module $\XXZ{N;d,z}+$ when the radical of $\Cell{N;d,z}$ is semisimple and, thus, its Loewy diagram has at most two layers. This exercise, which will be important in the final proofs of section \ref{sec:struc}, will show how the results of the present section and the previous one can be used to probe the structure of the eigenspaces $\XXZ{N;d,z}+$. We avoid the case of problematic pairs here and defer their study to appendix \ref{app:problematic}.
\begin{proposition}\label{prop:troisnoeuds} Let $q,z\in\mathbb C^\times$ and $(d,z)\in\lambda_N$ which is not a problematic pair. Note $(s,y)$ and $(t,x)\in\lambda$ the immediate successors of $(d,z)$ through condition A and B, if they exist. If $q^2$ is a primitive $\ell$-th root of unity, assume also $N<d+2\ell$. Then, if $\rad\Cell{N;d,z}$ is semisimple (and thus its Loewy diagram has at most two layers), the Loewy diagram of $\XXZ{N;d,z}+$ is
$$
\begin{tikzpicture}[baseline={(current bounding box.center)},scale=1/3]
\node (k0) at (0,8) [] {$(s,y)$};
\node (i0) at (4,6) [] {$(d,z)$};
\node (k1) at (0,4) [] {$(t,x)$};
\draw[->] (k0) -- (i0);\draw[->] (i0) -- (k1);
\end{tikzpicture}
$$
with the nodes $(s,y)$ and $(t,x)$ omitted if they are not in $\lambda_N$.
\end{proposition}
\begin{proof} Again this is accomplished by studying several cases.

\noindent{\bf\itshape Case 1: $(s,y)\not\in \lambda_N$ or does not exist}. In this case, theorem \ref{thm:mdsa} states that $\ay{N;d,z}$ is an isomorphism. The Loewy diagram of $\XXZ{N;d,z}+$ thus coincides with that of $\Cell{N;d,z}$ and is either
$$\begin{tikzpicture}[baseline={(current bounding box.center)},scale=1/3]
\node (i0) at (4,6) [] {$(d,z)$};
\node (k1) at (0,4) [] {$(t,x)$};
\draw[->] (i0) -- (k1);
\end{tikzpicture}
\qquad \text{or}\qquad
\begin{tikzpicture}[baseline={(current bounding box.center)},scale=1/3]
\node (i0) at (4,6) [] {$(d,z)$};
\end{tikzpicture}
$$
depending on whether $(t,x)$ exists and belongs to $\lambda_N$ or not. The statement follows.

\noindent{\bf\itshape Case 2: $(s,y)\in \lambda_N$ and $(t,x)\not\in \lambda_N$ or does not exist}. In this case, $s\leq N$ and $(s,y^{-1})$ directly succeeds $(d,z^{-1})$ through condition B by lemma \ref{thm:partialOrders}. Moreover, $(d,z^{-1})$ has no successor (in $\lambda_N$) through condition A since $t>N$ or $(t,x)$ does not exist. Hence, $\ay{N;d,z^{-1}}$ is an isomorphism and $\XXZ{N;d,z^{-1}}+$ has the Loewy diagram
$$
\begin{tikzpicture}[baseline={(current bounding box.center)},scale=1/3]
\node (i0) at (5,6) [] {$(d,z^{-1})$};
\node (k1) at (0,4) [] {$(s,y^{-1})$};
\draw[->] (i0) -- (k1);
\end{tikzpicture}
$$
However, $(\XXZ{N;d,z^{-1}}+)^\star\simeq \XXZ{N;d,z}+$ and $(\Irre{N;d,z^{-1}})^\star\simeq\Irre{N;d,z}$ by proposition \ref{thm:isoByDual} so, as the $\star$-duality flips Loewy diagrams (see section \ref{sec:dualities}), the diagram of $\XXZ{N;d,z}+$ ends up being
$$
\begin{tikzpicture}[baseline={(current bounding box.center)},scale=1/3]
\node (k0) at (0,8) [] {$(s,y^{-1})^\star$};
\node (k1) at (5,6) [] {$(d,z^{-1})^\star$};
\draw[->] (k0) -- (i0);
\end{tikzpicture}\simeq \begin{tikzpicture}[baseline={(current bounding box.center)},scale=1/3]
\node (k0) at (0,8) [] {$(s,y)$};
\node (k1) at (4,6) [] {$(d,z)$};
\draw[->] (k0) -- (i0);
\end{tikzpicture}
$$
\noindent{\bf\itshape Case 3: $(s,y)$ and $(t,x)\in \lambda_N$}. This may happen only when $q$ is a root of unity. Also, $q^d \neq z^2=q^s$ as the contrary would imply $s=d+2n\ell$ for a $n\in\mathbb{N}$ which is impossible since $s \leq N< d+2\ell$. In the same way, $q^d \neq z^{-2}=q^t$ and the structure of $\Cell{N;d,z}$ is thus given by the Loewy diagram (iii) of theorem \ref{thm:GL}. Lemma \ref{thm:partialOrders}  then allows us to link this Loewy diagram to that of $\Cell{N;d,z^{-1}}$. The two resulting diagrams are
$$\Cell{N;d,z}\ :
\begin{tikzpicture}[baseline={(current bounding box.center)},scale=1/3]
\node (k0) at (0,4) [] {$(s,y)$};
\node (i0) at (0,7) [] {$(d,z)$};
\node (k1) at (4,4) [] {$(t,x)$};
\draw[->] (i0) -- (k0);\draw[->] (i0) -- (k1);
\end{tikzpicture}
\qquad\text{and}\qquad
\Cell{N;d,z^{-1}}\ :
\begin{tikzpicture}[baseline={(current bounding box.center)},scale=1/3]
\node (k0) at (4,4) [] {$(s,y^{-1})$};
\node (i0) at (0,7) [] {$(d,z^{-1})$};
\node (k1) at (0,4) [] {$(t,x^{-1})$};
\draw[->] (i0) -- (k0);\draw[->] (i0) -- (k1);
\end{tikzpicture}
$$
where, as before, vertical arrows depict successors via condition A and diagonal ones via condition B. (Observe also that $(s,y)\neq (t,x)$ since an easy computation, given in case (i) of lemma \ref{lemma:IntFam}, show that the contrary would imply $q^d =z^2$.) The diagrams of the associated generic parts are obtained by quotienting by the successor via condition A and are thus
$$\pg{N;d,z}\ :
\begin{tikzpicture}[baseline={(current bounding box.center)},scale=1/3]
\node (i0) at (0,7) [] {$(d,z)$};
\node (k1) at (4,4) [] {$(t,x)$};
\draw[->] (i0) -- (k1);
\end{tikzpicture}
\qquad\text{and}\qquad
\pg{N;d,z^{-1}}\ :
\begin{tikzpicture}[baseline={(current bounding box.center)},scale=1/3]
\node (k0) at (4,4) [] {$(s,y^{-1})$};
\node (i0) at (0,7) [] {$(d,z^{-1})$};
\draw[->] (i0) -- (k0);
\end{tikzpicture}
$$
Theorem \ref{thm:mdsa} and proposition \ref{thm:Bsurvit} then give $\im \ay{N;d,z}\simeq \pg{N;d,z}$ so $\XXZ{N;d,z}+$ must contain a copy of $\pg{N;d,z}$ as a submodule. Similarly, $\XXZ{N;d,z^{-1}}+$ must have a submodule isomorphic to $\pg{N;d,z^{-1}}$. Again, proposition \ref{thm:isoByDual} and the $\star$-duality show that the Loewy diagram $\XXZ{N;d,z}+$ must contain the factor $\Irre{N;s,y}$ with an arrow from it to the factor $\Irre{N;d,z}$. Hence, $\XXZ{N;d,z}+$ shares its composition factors with $\Cell{N;d,z}$ as $\dim \XXZ{N;d,z}{+} = \dim \Cell{N;d,z}$ and the diagram of $\XXZ{N;d,z}+$ must be the desired one.
\end{proof}
Note that case 1 and case 2 above give a proof of the main result, theorem \ref{thm:main}, for any generic $q$ (that is any $q$ which is not a root of unity). The next sections thus concentrate on the study of the XXZ chains for $q$ a root of unity.
\end{subsection}
\end{section}

%%%%%%%%%%%%%
%%%%%%%%%%%%%
%%%%%   %%%%%
%%%%% 4 %%%%%
%%%%%   %%%%%
%%%%%%%%%%%%%
%%%%%%%%%%%%%

\begin{section}{The XXZ spin chain as $\luszt$-module}\label{sec:xxzLUq}
This section introduces the XXZ spin chain as a module over Lusztig's quantum group $\luszt$. We suppose that $q^2$ is an $\ell$th-primitive root of unity with $\ell \geq 2$, the limiting case $\ell =1$ being postponed to appendix \ref{app:problematic}. The first subsection presents the algebra $\luszt$ and states the key features of its representation theory while leaving most of the technical proofs to appendix \ref{app:b}. The second one gives $\atl{N}$-morphisms between spin chains and studies their image and kernel.
%%%%%%%%%%%%%%%
%
% aTL_N and LU_q(sl_2)
%
%%%%%%%%%%%%%%%\
\begin{subsection}{Lusztig's quantum group}\label{sec:luqsl2}
For $t$ a formal parameter, the \textit{rational form} $U_t\mathfrak{sl}_2$ is defined as the $\mathbb{Q}(t)$-algebra with generators $\{E,F,K^{\pm 1},\id\}$ and relations 
$$ KEK^{-1}=t^2E, \qquad KFK^{-1} = t^{-2}F, \qquad K^{\pm 1}K^{\mp 1} = \id \quad \text{and} \quad K-K^{-1}=(t-t^{-1})[E,F]. $$
To specialize the parameter $t$ to a root of unity $q\neq \pm 1$, it is customary to consider \textit{Lusztig's integral form} $U_{\text{res}}$ which is the $\mathbb{Z}[t,t^{-1}]$-subalgebra of $U_t\mathfrak{sl}_2$ generated by $K^{\pm 1}$ and the \textit{divided powers}
$$ E^{(n)} = \frac{1}{[n]_t!}E^n \quad \text{and} \quad F^{(n)} = \frac{1}{[n]_t!}F^n $$
for $n\in \mathbb{Z}_{\geq 0}$. The resulting specialization is called \textit{Lusztig's quantum group} and is denoted $\luszt$. The generators $K$, $E$ and $F$ are still elements of $\luszt$ as $\id = E^{(0)} = F^{(0)}$, $E=E^{(1)}$ and $F=F^{(1)}$. They satisfy
$E^{\ell} = F^{\ell} = 0$ and $K^{2\ell}=\id$, but $E^{(n)}$ and $F^{(n)}$ are non-zero and well defined even for $n \geq \ell$. There is also an involutive isomorphism $\luszt \simeq \lusztI$ which fixes the divided powers and sends $K$ to $K^{-1}$. This involution will be denoted $\iota$.

The quantum group $\luszt$ is a Hopf algebra with coproduct $\Delta:\luszt \rightarrow \luszt^{\otimes 2}$ given by $\Delta(K) = K\otimes K$ as well as\footnote{This is the coproduct obtained by conjugating the coproduct given in \cite{AndersenTubbenhauer} by (tensor products of) the isomorphism $\iota$ between $\luszt$ and $\lusztI$.} 
$$ \Delta(E^{(n)}) =  \sum_{m=0}^n q^{-m(n-m)}E^{(n-m)}K^{-m}\otimes E^{(m)} \quad\text{and} \quad \Delta(F^{(n)}) = \sum_{m=0}^n q^{m(n-m)}F^{(m)}\otimes K^m F^{(n-m)}. $$
It also admits a remarkable element $H$, the \textit{unrolled generator}, which satisfies $\Delta(H)= H\otimes 1+1\otimes H$. To define this element, let $\Phi_{\ell}(t)$ be the $\ell$th-cyclotomic polynomial and consider the non-zero complex number 
$$ \epsilon_q = \lim_{t\rightarrow q}\frac{t^{2\ell}-1}{\Phi_{\ell}(t^2)} = \prod_{\substack{1\leq k\leq \ell \\ \gcd(k,\ell)\neq 1}}(q^2-e^{2i\pi{k}/{\ell}})\in\mathbb C. $$
Then, $H = \frac{1}{\epsilon_q}H'$ where $H'$ is the specialization to $\luszt$ of the element
$$ \frac{K^{2\ell}-\id}{\Phi_{\ell}(t^2)} $$
of Lusztig's integral form $U_{\text{res}}$. The proof that $H'$ is a non-zero element of $U_{\text{res}}$ and that $\Delta(H) = H\otimes 1+1\otimes H$ can be found in \cite{Lentner}. The generator $H$ is interesting as its action on a representation $V$ of $\luszt$ is related to the $q$-logarithm of the associated $K$-action. A first example of this remarkable phenomenon is given via the following family of $\luszt$-modules. 
\begin{definition}[Weyl modules, \cite{AndersenTubbenhauer}]\label{def:Weyl} Let $i \in \mathbb{Z}_{\geq 0}$. The \textit{Weyl module} $\Delta_q(i)$ is the $\luszt$-module with $\mathbb{C}$-basis $\{m_0,...,m_i\}$ whose action is 
$$ Km_k = q^{i-2k}m_k, \qquad E^{(n)}m_k =\qbin{i-k+n}{n}{q}m_{k-n} \quad \text{and} \quad F^{(n)}m_k = \qbin{k+n}{n}{q}m_{k+n} $$
for $k \in \{0,...,i\}$, $n\in\mathbb{Z}_{\geq 0}$, and where
$m_s = 0$ if $s < 0$ or $s > i$. This action satisfies $Hm_k = (i-2k)m_k$.
\end{definition}
\begin{proposition}[\cite{AndersenTubbenhauer}]\label{prop:Weyl} Let $i,r,s\in \mathbb{Z}_{\geq 0}$ be such that $s<\ell$ and $i=r\ell+s$. Then, $L_q(i) = \head \Delta_q(i)$ is simple. Also,
\begin{itemize}
\item[(i)] $\Delta_q(i)=L_q(i)$ if $r=0$ or $s=\ell-1$.
\item[(ii)] If $s\neq \ell-1$, there is a non-split short exact sequence 
$$ 0\rightarrow L_q(i)\rightarrow \Delta_q(j)\rightarrow L_q(j) \rightarrow 0 $$
where $j= i+2(\ell-s-1)$. In particular, $L_q(i)=\soc \Delta_q(j)$.
\item[(iii)] $\{L_q(i)\}_{i\in \mathbb{Z}_{\geq 0}}$ is a complete set of finite-dimensional non-isomorphic simple $\luszt$-modules (of type I)\footnote{See \cite{AndersenTubbenhauer} and \cite{AndersenPolo} for the definition of $\luszt$-modules of types I and II. The distinction between these two types is not crucial in the present text.}.
\end{itemize}
\end{proposition} 
\noindent The second part of proposition \ref{prop:Weyl} makes it easy to obtain an explicit realization of the simple module $L_q(i)$. Indeed, for $j,r,s$ as above with $s\neq \ell-1$, one can show that the subspace $V_j = \vspn_{\mathbb{C}}\{m_{a\ell+b}\,|\,0\leq a \leq r,\, \ell-s-1\leq b<\ell\}\subsetneq \Delta_q(j)$ is closed under the $\luszt$-action so that the proposition gives $L_q(i)\simeq\soc \Delta_q(j) \simeq V_j$. In particular\footnote{We use here the fact that $0\leq a\ell+b\leq (r+1)\ell-1 \leq (r+2)\ell-s-2= i+2(\ell-s-1) = j$ for every $0\leq a\leq r$ and $\ell-s-1\leq b < \ell$ as $s< \ell$.}, $\dim L_q(i) = (r+1)(s+1)$.

Another remarkable family of modules is given by the projective\footnote{Note that these modules are probably not projective in the category of all $\luszt$-modules.} covers of the simple modules $L_q(i)$ in the category of finite-dimensional $\luszt$-modules (of type I). These modules are implicitly described in the next proposition.
\begin{proposition}[\cite{AndersenTubbenhauer}]\label{prop:ProjAT} Let $i,r,s \in \mathbb{Z}_{\geq 0}$ satisfy $s<\ell$ with $i=r\ell+s$. Denote by $P_q(i)$ the projective cover of $L_q(i)$. Then, 
\begin{itemize}
\item[(i)] $P_q(i)$ is an indecomposable (type I) $\luszt$-module which is also the injective envelope of $L_q(i)$.
\item[(ii)] $P_q(i) = \Delta_q(i) = L_q(i)$ if $s = \ell-1$.
\item[(iii)] If $s\neq \ell-1$, there is a non-split short exact sequence
$$ 0\rightarrow \Delta_q(j)\rightarrow P_q(i)\rightarrow \Delta_q(i)\rightarrow 0 $$
where $j = i+2(\ell-s-1)$. In particular, $\dim P_q(i) = 2\ell(r+1)$.
\item[(iv)] Let $\mathsf M$ be projective in the category of finite-dimensional (type I) $\luszt$-modules. Then, for each indecomposable summand $\mathsf M'$ of $\mathsf M$, there is a unique $i\in \mathbb{Z}_{\geq 0}$ such that $\mathsf M'\simeq P_q(i)$ as $\luszt$-modules.
\end{itemize}
\end{proposition}
For $i=r\ell+s$ as above with $s\neq \ell-1$, there is another non-split short exact sequence (cf. \cite{AndersenTubbenhauer})
\begin{equation*}
0\rightarrow \nabla_q(i) \rightarrow P_q(i) \rightarrow \nabla_q(j)\rightarrow 0\end{equation*}
where again $j=i-2(\ell-s-1)$ and where $\nabla_q(i)$ with $\nabla_q(j)$ are \textit{co-Weyl modules} (i.e. dual Weyl modules defined by using the antipode of $\luszt$). In particular, the modules $P_q(i)$ are self-dual and admit the following Loewy diagram 
\begin{equation}\label{fig:LoewyPq}
\begin{tikzpicture}[baseline={(current bounding box.center)},scale = 0.5]
\node[scale = 1] (A) at (0,0) {$L_q(i)$};
\node[scale = 1] (B) at (0,-4) {$L_q(i)$};
\node[scale = 1] (C) at (-2,-2) {$L_q(i_2)$};
\node[scale = 1] (D) at (2,-2) {$L_q(j)$};
\draw[->, thick, color = red]  (A) edge (D) (C) edge (B) ;
\draw[->, thick, color = blue] (A) edge (C)  (D) edge (B);
\end{tikzpicture}
\end{equation}
where, again, $i = r\ell+s$ with $s\neq \ell-1$, $j = i+2(\ell-s-1)$ and $i_2 = i-2(s+1)$. The composition factor $L_q(i_2)$ must be removed if $i_2 < 0$. The exact sequence of proposition \ref{prop:ProjAT} and its dual are illustrated with blue and red arrows (resp.).

The next theorem, proved in appendix \ref{app:b}, gives an explicit realization of the projective covers $P_q(i)$ when $P_q(i)\neq \Delta_q(i)$. The proof uses the fact that $\Ext^1_{\luszt}(\Delta_q(i),\Delta_q(j))\simeq \mathbb{C}$ (cf.~lemma \ref{lemma:ExtWeyl}) and starts by defining a module $T_t(i)$ over the integral form $U_{\text{res}}$ whose specialization at $t=q$ is the desired realization of $P_q(i)$. To our knowledge, this realization is new. It makes explicit the non-split short exact sequence of proposition \ref{prop:ProjAT} (ii). A different realization is known when $q = e^{i\pi/\ell}$ (cf.~\cite{bushlanov2009lusztig}) but does not exhibit this exact sequence. However, this other realization makes explicit the action of a remarkable subalgebra $\mathscr{U}\subseteq \luszt$, isomorphic to $U\mathfrak{sl}_2$, on the composition factors of $P_q(i)$.
\begin{theorem}\label{thm:ProjReal} Let $i,r,s \in \mathbb{Z}_{\geq 0}$ satisfy $s< \ell-1$ and $i=r\ell+s$. The linear space $T_q(i)$ with $\mathbb{C}$-basis $\{m_0,...,m_j,n_0,...,n_i\}$ is a $\luszt$-module such that $P_q(i)\simeq T_q(i)$ for the action given by $Km_k = q^{j-2k}m_k$, $Kn_p = q^{i-2p}n_p$ and
\begin{alignat*}{5} E^{(v)}m_k &= \qbin{j-k+v}{v}{q}m_{k-v},& \qquad &E^{(v)}n_p &=&\ \qbin{i-p+v}{v}{q}n_{p-v}+\gamma_{p,v}m_{\ell-s+p-v-1},\\
 F^{(v)}m_k &= \qbin{k+v}{v}{q}m_{k+v},& \qquad &F^{(v)}n_p &=&\ \qbin{p+v}{v}{q}n_{p+v}+\omega_{p,v}m_{\ell-s+p+v-1}
\end{alignat*}
where the complex numbers $\gamma_{p,v}$ and $\omega_{p,v}$ are
\begin{align*} 
\gamma_{p,v} & = \lim_{t\rightarrow q}\frac{1}{[v]_t!}\sum_{u=0}^{v-1} \qbin{\ell-s+p-u-2}{p-u}{t}\ \prod_{a=1}^u[i-p+a]_t\prod_{b=u+1}^{v-1}[i+\ell-s-p+b]_t, \\ 
\omega_{p,v} & = \begin{cases}
{\displaystyle \lim_{t\rightarrow q}\frac{1}{[\ell-s+i]_t}\qbin{\ell-s+p+v-1}{p+v}{t}\qbin{p+v}{v}{t},} & \text{if } v > i-p, \\
0, & \text{if } v \leq i-p.
\end{cases}
\end{align*}
Moreover the submodule $\mathsf M$ of $T_q(i)$ generated by $\{m_0,...,m_j\}$ is isomorphic to $\Delta_q(j)$ and $T_q(i)/\mathsf M \simeq \Delta_q(i)$. Also, the unrolled generator $H$ acts on $T_q(i)$ as on the direct sum $\Delta_q(i)\oplus\Delta_q(j)$, that is $Hm_k = (j-2k)m_k$ and $Hn_p = (i-2p)n_p$.
\end{theorem}
\textit{Fusion rules}, that is, decompositions of tensor products of $\luszt$-modules into indecomposable summands, will be used below and are proved in appendix \ref{app:b}. Note that, in light of these rules, we can identify the covers $P_q(i)$ with the non-simple modules studied by Pasquier and Saleur \cite{PS} in their work\footnote{In \cite{PS}, the authors found the decomposition of the modules $(L_q(1))^{\otimes N}$ in terms of indecomposable $\luszt$-modules. They did not prove the projectivity of the non-irreducible indecomposable summands appearing in this decomposition, but they did study their structure and found in particular preliminary versions of the fusion rules presented in theorem \ref{thm:fusLUq} and of the upcoming proposition \ref{prop:SECWeylProj}. 
} on the $\luszt$-module $(L_q(1))^{\otimes N}$. Note also that our result reproduces some of the fusion rules given in \cite{bushlanov2009lusztig,bushlanov2012lusztig} when $q = e^{i\pi/\ell}$.
\begin{theorem}\label{thm:fusLUq} Let $i,r,s\in \mathbb{Z}_{\geq 0}$ satisfy $s<\ell$ and $i=r\ell+s$. Then,
\begin{equation*}
L_q(i)\otimes L_q(1) \simeq (1-\delta_{s,0})M_q(i-1)\oplus (1-\delta_{s,\ell-1})L_q(i+1)\quad\text{and}
\end{equation*}
\begin{equation*}
P_q(i)\otimes L_q(1) \simeq (1+\delta_{s,\ell-2})P_q(i+1)\oplus (1-\delta_{r,0}\delta_{s,0})P_q(i-1)\oplus \delta_{s,0}P_q(i+2\ell-1)
\end{equation*}
where (only) the first rule applies when the projective modules $P_q(i)$ is isomorphic to $L_q(i)$ (i.e. if $s=\ell-1$) and where
$$M_q(i-1) = \begin{cases}
L_q(i-1), & \text{if } s \neq \ell-1, \\
P_q(i-1), & \text{if } s = \ell-1.
\end{cases}
$$
\end{theorem}
\begin{corollary}\label{coro:seulWeylProj}Let $L_q(1)\simeq \Delta_q(1)$ be the two-dimensional Weyl module of definition \ref{def:Weyl} and let $k$ be a positive integer. Then the $k$-fold tensor product $(L_q(1))^{\otimes k}$ is a direct sum of Weyl and projective modules.
\end{corollary}
\begin{proof} As proposition \ref{prop:Weyl} states, the simple $L_q(i)$ and the Weyl $\Delta_q(i)$ coincide when either $i\leq \ell-1$ or $i+1\equiv0\text{ mod }\ell$. Hence, the tensor product $(L_q(1))^{\otimes k}$ remains a direct sum of Weyl modules as long $k< \ell-1$. When $k$ reaches $\ell-1$, the tensor product $(L_q(1))^{\otimes (\ell-1)}$ contains $L_q(\ell-2)\otimes L_q(1)$ which gives rise to the projective $L_q(\ell-1)\simeq P_q(\ell-1)$. As the second rule of theorem \ref{thm:fusLUq} shows that the product $P_q(i)\otimes L_q(1)$ decomposes into a direct sum of projectives, the only concern appears when $P_q(i)$ is isomorphic to $L_q(i)$, that is when $i+1\equiv 0\text{ mod }\ell$. In this case, the first rule of theorem \ref{thm:fusLUq} must be used and this rule shows that only the term $M_q(i-1)$, which turns out to be the projective $P_q(i-1)$, survives.
\end{proof}
This section closes with important propositions that characterize the action of the divided powers on the projectives $P_q(i)$ and Weyl modules $\Delta_q(i)$. Specific forms of these propositions were already proven in \cite{PS}. 
\begin{proposition}\label{prop:AutWeylProj} Let $i,n\in \mathbb{Z}_{\geq 0}$ and $V  = \Delta_q(i)$ or $V=P_q(i)$. Then, the action of $e^nf^n$ (or $f^ne^n$) with $e = E^{(\ell)}$ and $f = F^{(\ell)}$ defines a $\mathbb{C}$-linear automorphism of the $H$-eigenspace $\Eval{V}{H=d}{}$ when $|d-2n\ell|\leq|d|$ (or $|d+2n\ell|\leq|d|$, respectively).
\end{proposition}
\begin{proof} We only consider the action of $e^nf^n$ as the proof for $f^ne^n$ is similar. Suppose $|d-2n\ell|\leq |d|$ and let first $V = \Delta_q(i)$ so that $\Eval{V}{H=d}{} = \mathbb{C}m_k$ for $k = \frac{1}{2}(i-d)\in \{0,...,i\}$ (in the notation of definition \ref{def:Weyl}). A direct computation shows that 
\begin{equation*} e^nf^nm_k = \prod_{u=1}^n \qbin{k+u\ell}{\ell}{q}\qbin{i-k+(1-u)\ell}{\ell}{q}m_k
\end{equation*}
which is zero only if $i-k<u\ell$ for a certain $1\leq u\leq n$ by theorem \ref{thm:qLucas}. However, this is impossible since our assumptions imply $i-2(k+u\ell)\geq i-2(k+n\ell)=d-2n\ell \geq -|d| \geq -i$ as $0\leq k=\frac{1}{2}(i-d) \leq i$. The result is thus proved if $V= \Delta_q(i)$. 

When $V = P_q(i)$ and $V\neq\Delta_q(i)$, the eigenspace $\Eval{V}{H=d}{}$ is the $\mathbb C$-span of $\{m_k,n_p\}$ with $k =\frac{1}{2}(j-d) \in \{0,\dots,j\}$ (where $j=i-2(\ell-s-1)$ as usual) and with $p = \frac{1}{2}(i-d)$. Note that $p$ may satisfy $p<0$ or $p>i$ as $j > i$ when $P_q(i) \neq \Delta_q(i)$. If $p$ is not in $\{0,\dots, i\}$, $n_p$ is understood to be $0$ and $\{m_k\}$ is then a basis of $\Eval{V}{H=d}{}$. As $m_k$ is contained in the submodule $\mathsf M$ isomorphic to $\Delta_q(j)$, we can use the first part of the proof to conclude that $e^nf^nm_k$ is a non-zero multiple of $m_k$. This ends the proof if $n_p=0$. Suppose now $n_p\neq 0$. Again, by the first part, $e^nf^n(n_p+\mathsf M)$ is a non-zero multiple of $n_p+\mathsf M$ as $n_p+\mathsf M \in P_q(i)/\mathsf M\simeq \Delta_q(i)$. Thus, the elements $e^nf^nn_p$ and $e^nf^nm_k$ are linearly independent and the action of $e^nf^n$ induces an automorphism of $\Eval{V}{H=d}{}$ when $V=P_q(i)$ and $|d-2n\ell|\leq|d|$. The result follows.
\end{proof}
\begin{proposition}\label{prop:SECWeylProj} 
Let $V$ be either $\Delta_q(i)$ or $P_q(i)$ for some $i\geq 0$. Fix also $d> 0$ an eigenvalue of $H$ on $V$ and $x \in \Eval{V}{H=d}{}$. Fix finally $a_1\in\mathbb{Z}_{\geq 0}$ and $1\leq a_2<\ell$ such that $d \geq a_1\ell+a_2$. Then, $F^{(a_1\ell+a_2)}x = 0$ if and only if $x = F^{(\ell-a_2)}y$ for some $y\in \Eval{V}{H=d+2(\ell-a_2)}{}$. In other words, the kernel of $F^{(a_1\ell+a_2)}$ on $\Eval{V}{H=d}{}$ is the image of $\Eval{V}{H=d+2(\ell-a_2)}{}$ by $F^{(\ell-a_2)}$. 
\end{proposition}
\begin{proof}
The fact that the image of $\Eval{V}{H=d+2(\ell-a_2)}{}$ by $F^{(\ell-a_2)}$ is contained in the kernel of $F^{(a_1\ell+a_2)}$ on $\Eval{V}{H=d}{}$ follows from the fact that Lemma \ref{prop:RelationsUres} gives
\begin{equation*}
F^{(a_1\ell+a_2)}F^{(\ell-a_2)} = \qbin{(a_1+1)\ell}{\ell-a_2}{q} F^{((a_1+1)\ell)}
\end{equation*}
and the $q$-binomial on the right-hand side is zero by theorem \ref{thm:qLucas}.
We thus only need to show the other inclusion. 

Consider first $V = \Delta_q(i)$. Then $\Eval{V}{H=d}{} = \mathbb{C}m_k$ (with the notation of definition \ref{def:Weyl}) for $k = \frac{1}{2}(i-d)\in\{0,...,i\}$. Set $k=k_1\ell+k_2$ with $k_1,k_2\in \mathbb{Z}_{\geq 0}$ and
$k_2<\ell$. Then, the action of $F^{(a_1\ell+a_2)}$ on $m_k$ is
\begin{equation*} F^{(a_1\ell+a_2)}m_k = \qbin{k+a_1\ell+a_2}{a_1\ell+a_2}{q}m_{k+a_1\ell+a_2} = \qbin{(a_1+k_1)\ell+k_2+a_2}{a_1\ell+a_2}{q}m_{k+a_1\ell+a_2}
\end{equation*}
with $m_{k+a_1\ell+a_2}\neq 0$ as $0\leq k+a_1\ell+a_2 \leq k+d = i-k \leq i$. Suppose from now on that $m_k$ is in the kernel of $F^{(a_1\ell+a_2)}$. Then $F^{(a_1\ell+a_2)}m_k=0$ gives $k_2+a_2>\ell$ by theorem \ref{thm:qLucas} and this same theorem  implies that $\rho \vcentcolon= \qbin{k}{\ell-a_2}{q}\neq 0$. Moreover, the inequality $k_2+a_2>\ell$ also implies $0\leq k_1\ell\leq k-\ell+a_2< k\leq i$ and thus $m_{k-\ell+a_2}$ is non-zero in $V$. The element $y=\frac{1}{\rho} m_{k-\ell+a_2}\in \Eval{V}{H=d+2(\ell-a_2)}{}$ is hence well-defined and satisfies $F^{(\ell-a_2)}y = m_k $ as desired. Let now $V= P_q(i)$ with $P_q(i)\neq \Delta_q(i)$ and suppose $d \leq i$. Then, with $p=\frac{1}{2}(i-d)\geq 0$ and the notation of theorem \ref{thm:ProjReal}, we have $\omega_{p,a_1\ell+a_2} = 0$ since $a_1\ell+a_2 \leq d \leq d+p = i-p$, and the divided power $F^{(a_1\ell+a_2)}$ thus acts on the $H$-eigenspace $\Eval{V}{H=d}{}$ as on the corresponding $H$-eigenspace of the direct sum $\Delta_q(i)\oplus \Delta_q(j)$. This still holds when $d > i$ as in this case $\Eval{V}{H=d}{}\subseteq \mathsf M\subseteq P_q(i)$ with $\mathsf M\simeq \Delta_q(j)$. The construction of a $y$ as in the first part of the proof goes through and this concludes the proof.
\end{proof}
\end{subsection}

%%%%%%%%%%%%%%%
%
% XXZ comme LU_q(sl_2)-module
%
%%%%%%%%%%%%%%%
\begin{subsection}{Morphisms between XXZ spin chains}\label{sub:homoXXZ}
The basis $\{m_0,m_1\}$ was used to define the $\luszt$-action on the Weyl module $L_q(1)\simeq \Delta_q(1)$ (see definition \ref{def:Weyl}). This basis is hereafter identified with the one used in section \ref{sec:atlOnXXZ} for $\mathbb C^2$ via $m_0\mapsto |+\rangle$ and $m_1\mapsto |-\rangle$. We also define recursively a family of maps $\{\Delta_n : \luszt\rightarrow \luszt^{\otimes n}\}_{n\geq 2}$ by $\Delta_2 = \Delta$ and $\Delta_{n+1} = (\id \otimes \Delta)\circ \Delta_{n}$ where $\Delta$ is the coproduct of section \ref{sec:luqsl2} and with $\id$ the identity operator on $\luszt^{\otimes(n-1)}$. We then associate the XXZ chain $X_{N;z}^+\simeq (\mathbb C^2)^{\otimes N}$ to the $\luszt$-module $(L_q(1))^{\otimes N}$ with the action induced by $\Delta_N$. The XXZ chain $X_{N;z}^-$ is also identified with the pullback of $(L_q(1))^{\otimes N}$ by the involution $\iota : \lusztI \rightarrow \luszt$. This latter chain can thus be seen as a $\lusztI$-module.

Corollary \ref{coro:seulWeylProj} has shown that, as a $\luszt$-module, the chain $\XXZ{N;z}{+}$ decomposes into a direct sum of Weyl modules $\Delta_q(i)$ and indecomposable projectives $P_q(i)$. Analogously, as the pullback by $\iota$ sends\footnote{This follows from definition \ref{def:Weyl}, proposition \ref{thm:ProjReal} and the fact that $[x]_q = [x]_{q^{-1}}$ for each $x \in \mathbb{Z}$.} $P_q(i)$ and $\Delta_q(i)$ to the $\lusztI$-modules $P_{q^{-1}}(i)$ and $\Delta_{q^{-1}}(i)$, the periodic chain $\XXZ{N;z}{-}$, viewed as a $\lusztI$-module, also decomposes into Weyl and indecomposable projective modules. This implies that propositions \ref{prop:AutWeylProj} and \ref{prop:SECWeylProj} apply to both $\XXZ{N;z}{+}$ and $\XXZ{N;z}{-}$. Furthermore, as the unrolled generator $H$ satisfies $\Delta(H) = H\otimes 1 + 1\otimes H$ with $Hm_0 = m_0$ and $Hm_1 = -m_1$ on $L_q(1)$, the preceding construction gives
$$ H|x_1\dots x_N\rangle_z^+ = \sum_{j=1}^N x_j|x_1\dots x_N\rangle_z^+ = 2S^z|x_1\dots x_N\rangle_z^+ $$
where we use the basis $\BXXZ{N;z}$ for $\XXZ{N;z}{+}$ (see the proof of proposition \ref{prop:spinflip}). The action of $H$ on $\XXZ{N;z}{+}$ thus corresponds to that of the total spin $2S^z$. Moreover, the expression $H = \frac{1}{\epsilon_q}H'$ given in section \ref{sec:luqsl2} and the fact that $K^{2\ell}=\id$ in $\luszt$ with $q^{2\ell}=1$ show that the involution $\iota$ leaves $H$ unchanged. The above identity hence also holds for the chain $\XXZ{N;z}{-}$, that is $H|x_1\dots x_N\rangle_z^- =2S^z|x_1\dots x_N\rangle_z^-$. The next results follow directly from the above remarks and propositions \ref{prop:AutWeylProj} and \ref{prop:SECWeylProj}.
\begin{theorem}\label{thm:SECX} Fix $0< d\leq N$ such that $d\equiv N \text{ mod } 2$ and let $x \in \XXZ{N;d,z}{\pm}$. Let also $a_1 \in \mathbb{Z}_{\geq 0}$ and $1\leq a_2 < \ell$ with $d\geq a_1\ell+a_2$. Then, $F^{(a_1\ell+a_2)}x = 0$ if and only if $x = F^{(\ell-a_2)}y$ for some $y \in \XXZ{N;d+2(\ell-a_2),z}{\pm}$. In other words, the kernel of $F^{(a_1\ell +a_2)}$ on $\XXZ{N;d,z}{\pm}$ is the image of $\XXZ{N;d+2(\ell-a_2),z}{\pm}$ by $F^{(\ell-a_2)}$.
\end{theorem}
\begin{rem} The hypothesis $d \geq a_1\ell+a_2$ of theorem \ref{thm:SECX} (and of proposition \ref{prop:SECWeylProj}) is equivalent to saying that the divided power $F^{(a_1\ell+a_2)}$ sends $\XXZ{N;d,z}{\pm}$ to an eigenspace of at least the same dimension. 
\end{rem}
\begin{theorem}\label{thm:AutX} Let $n\in \mathbb{Z}_{\geq 0}$ and $d \in \mathbb{Z}$ with $|d|\leq N$ and $d\equiv N\text{ mod }2$. Then, the action of $e^nf^n$ (or $f^ne^n$) on $\XXZ{N;z}{\pm}$ for $e = E^{(\ell)}$ and $f = F^{(\ell)}$ induces a $\mathbb{C}$-linear automorphism of $\XXZ{N;d,z}{\pm}$ if $|d-2n\ell|\leq |d|$ (or if $|d+2n\ell|\leq |d|$, respectively). 
\end{theorem}
We now want to investigate the relation between the action of $\lusztpm$ and $\atl{N}$ on the chains $\XXZ{N;z}{\pm}$. We need a lemma. 
\begin{lemma}\label{lem:lesFetG}Let $E^{(n)}_{\pm}$ and $F^{(n)}_{\pm}$ stand for the $\lusztpm$-action of $E^{(n)}$ and $F^{(n)}$ on $\XXZ{N;z}\pm$ through $\Delta_N$. Then, on $\BXXZ{N;z}$,
\begin{align*}
E^{(n)}_{\pm}|x_1\dots x_N\rangle_z^{\pm} &= \sum_{\substack{1\leq j_1 < \dots  < j_n \leq N\\ x_{j_1} = \dots  = x_{j_n} = -}}\Big(\prod_{k=0}^{n-1}\prod_{j=j_k+1}^{j_{k+1}-1}q^{\pm(k-n)x_j}\Big)|x_1\dots x_{j_1-1}(+)x_{j_1+1}\dots x_{j_2-1}(+)x_{j_2+1}\dots x_N\rangle_z^{\pm}\\
\intertext{and}
F^{(n)}_{\pm}|x_1\dots x_N\rangle_z^{\pm} &= \sum_{\substack{1\leq j_1 < \dots  < j_n \leq N\\ x_{j_1} = \dots  = x_{j_n} = +}} \Big(\prod_{k=1}^n\prod_{j=j_{k}+1}^{j_{k+1}-1} q^{\pm kx_j}\Big)|x_1\dots x_{j_1-1}(-)x_{j_1+1}\dots x_{j_2-1}(-)x_{j_2+1}\dots x_N\rangle_z^{\pm}
\end{align*}
where $j_0 = 0$ and $j_{n+1} = N+1$. Moreover, on the eigenspace $H=d$, the two actions are related by the spin flip $s$:
\begin{equation}\label{eq:lienFetG}
F_{\pm}^{(a)}s = q^{\mp a(d+a)}sE_{\mp}^{(a)}.
\end{equation}
\end{lemma}
\begin{proof}[Sketch of proof]The explicit form of $E^{(n)}_{\pm}$ and $F^{(n)}_{\pm}$ can be proved by induction on $N$. For relation \eqref{eq:lienFetG}, fix any basis vector of the eigenspace $H=d$ and, for this vector, any positions $1\leq j_1<j_2<\dots<j_a\leq N$ such that the spins $x_{j_i}$ are all $-$. The result of the action of either $F_{\pm}^{(a)}s$ or $q^{\mp a(d+a)}sE_{\mp}^{(a)}$ on this vector contains a single vector whose spins at the $j_i$'s are still $-$. A direct computation, using the explicit action of the $E^{(n)}_{\pm}$ and $F^{(n)}_{\pm}$ of the lemma, shows that the coefficients of this vector under either action coincide.This suffices to prove \eqref{eq:lienFetG}.
\end{proof}

It is well-known (see for example \cite{PS, Jimbo1, Jimbo2, MartinSchur}) that, on the open XXZ chain, the action of the Temperley-Lieb algebra $\tl{N}$ commutes with that of the quantum group $\lusztpm$. This
\textit{quantum Schur-Weyl duality} is compatible with the inclusion of $\tl{N}$ in $\atl{N}$ as the subalgebra generated by $e_1, \dots , e_{N-1}$. We thus have the following important result.
\begin{proposition}\label{prop:qSW} The action of $\lusztpm$ on $\XXZ{N;z}{\pm}$ commutes with that of the subalgebra $\tl{N}\subseteq \atl{N}$.
\end{proposition}
\noindent Unfortunately, this commutativity does not extend to the action of the whole affine Temperley-Lieb algebra $\atl{N}$ as, for instance, on the spin chain $\XXZ{2;z}{+}$, we generically have
\begin{equation*}
F_+\Omega_2|{+}{-}\rangle^+_z = z^{-1}F_+|{-}{+}\rangle_z^+ = z^{-1}|{-}{-}\rangle_z^+ \neq q^{-1}z |{-}{-}\rangle_z^+= q^{-1}\Omega_2|{-}{-}\rangle_z^+= \Omega_2 F_+|{+}{-}\rangle_z^+.
\end{equation*}
Nevertheless, as we explain in the next theorem, the divided powers of $\lusztpm$ can still be used to define $\atl{N}$-morphisms relating carefully chosen spin-eigenspaces of different periodic XXZ spin chains. This highly non-trivial fact, which was already known in some specific cases (cf.~\cite{PS, Deguchi, MDSAloop}), is one of the principal results of this paper. Recall that $\unlhd$ denotes the natural extension to $\Delta=\mathbb{Z}\times \mathbb{C}^{\times}$ of the partial order $\preceq$ defined on $\lambda$. More precisely, for $(d,z),(t,x)\in \Delta$, recall that $(t,x)$ succeeds to $(d,z)$ if there exists a non-negative integer $m$ satisfying $t = d+2m$ with
\begin{equation}\label{eq:conditionsAB2}
\text{(A)}\quad  z^2=q^t\text{ and }x=zq^{-m}\qquad\text{ or }\qquad
\text{(B)}\quad z^2=q^{-t}\text{ and }x=zq^{m}
\end{equation} 
and let $\unlhd$ be the weakest partial order on $\Delta$ containing $(d,z)\unlhd (t,x)$ for all such successions.
\begin{definition} Fix $(d,z)$ and $(t,x)\in\Delta_N$ with $a=\frac12(t-d)\geq 0$. The $\mathbb C$-linear maps
\begin{equation}\label{eq:morphTheo}
\ii_{(d,z);(t,x)}^\pm:\XXZ{N;t,x}\pm\to\XXZ{N;d,z}\pm\qquad\text{and}\qquad
\jj_{(t,x);(d,z)}^\pm:\XXZ{N;d,z}\pm\to\XXZ{N;t,x}\pm
\end{equation}
are defined through the action of divided powers on basis elements:
$$
|x_1\dots x_N\rangle_x^\pm\mapsto F_\pm^{(a)}|x_1\dots x_N\rangle_{z}^\pm
\qquad\text{and}\qquad
|x_1\dots x_N\rangle_{z}^\pm\mapsto E_\pm^{(a)}|x_1\dots x_N\rangle_x^\pm .
$$
In these expressions, all upper signs must be chosen or all lower ones. 
\end{definition}
We extend this definition to any $(d,z)$ and $(t,x) \in \Delta$ by setting $\ii_{(d,z);(t,x)}^{\pm} =
\jj_{(t,x);(d,z)}^{\pm} = 0$ if $(d,z)$ or $(t,x)$ is not in $\Delta_N$.
\begin{theorem}\label{thm:f}
Let $(d,z),(t,x)\in\Delta$. Then, the following maps are $\atl N$-linear under the stated condition:
\begin{align*}
\ii_{(d,z);(t,x)}^\pm &:\XXZ{N;t,x}\pm\to\XXZ{N;d,z}\pm\quad &\text{if $(d,z)\unlhd (t,x)$ through condition B for the upper signs (A for the lower),}\\
\jj_{(t,x);(d,z)}^\pm &:\XXZ{N;d,z}\pm\to\XXZ{N;t,x}\pm\quad &\text{if $(d,z)\unlhd (t,x)$ through condition A for the upper signs (B for the lower).}
\end{align*}
Suppose also $t \geq \max(1,|d|)$ with $(d,z)\unlhd(t,x)$ through condition B and let $(h,v)$ and $(D,Z)$ be the direct successors of $(t,x)$ through conditions A and B respectively. Then, if $\frac12(t-d)\not\equiv 0\text{ mod }\ell$, the sequence of $\atl{N}$-linear maps
\begin{equation}\label{eq:SEC1}
\XXZ{N;D,Z}+\xrightarrow{\ii_{(t,x);(D,Z)}^+}\XXZ{N;t,x}+\xrightarrow{\ii_{(d,z);(t,x)}^+}\XXZ{N;d,z}+,
\end{equation}
is exact, that is $\im \ii_{(t,x);(D,Z)}^+=\ker \ii_{(d,z);(t,x)}^+$. Similarly, $\im \ii_{(t,x);(h,v)}^-=\ker \ii_{(-d,z^{-1});(t,x)}^-$ when $\frac12(t+d)\not\equiv 0\text{ mod }\ell$ and the sequence of $\atl{N}$-morphisms
\begin{equation}\label{eq:SEC2}
\XXZ{N;h,v}-\xrightarrow{\ii_{(t,x);(h,v)}^-}\XXZ{N;t,x}-\xrightarrow{\ii_{(-d,z^{-1});(t,x)}^-}\XXZ{N;-d,z^{-1}}-
\end{equation}
is exact. In particular, the sequences \eqref{eq:SEC1} and \eqref{eq:SEC2} are automatically exact if $q^d$ is distinct from both $z^2$ and $z^{-2}$.
\end{theorem}
\begin{rem} The hypothesis $t \geq \max(1,|d|)$ of theorem \ref{thm:f} comes from the hypothesis $d\geq n\ell+a$ of theorem \ref{thm:SECX}. 
\end{rem}
We limited the presentation to $\atl N$-linear maps sending $\XXZ{N;x}\pm$ to $\XXZ{N;z}\pm$. However, the spin flip $s$ can clearly be used to construct morphisms between $\XXZ{N;x}\pm$ and $\XXZ{N;z}\mp$. Also, other exact sequences can be deduced using the family of $\jj^\pm$'s and the spin flip $s$. The theorem showcases the particular sequences \eqref{eq:SEC1} and \eqref{eq:SEC2} since these are the only ones used in section \ref{sec:struc} for the proof of theorem \ref{thm:main}.
\begin{rem} The morphisms of theorem \ref{thm:f} are richer than those explored in \cite{PS,Deguchi,MDSAloop} as they allow $\atl{N}$-linear maps relating the $S_z$-eigenspaces of XXZ chains having different twist parameters $z$ and $x$. We also recover all the morphisms presented in these articles as special cases of the morphisms given above.
\end{rem}
\begin{proof}[Proof of theorem \ref{thm:f}] We only consider $\ii = \ii_{(d,z);(t,x)}^+$ with $(d,z)\unlhd (t,x)$ via condition B as the proof for $\ii_{(d,z);(t,x)}^-$ is similar. The compatibility of $\ii$ with $e_1$ follows from proposition \ref{prop:qSW} (as the action of $e_1$ is independent of $z$, see \eqref{eq:e1}). Because of the relation $\Omega_Ne_i=e_{i-1}\Omega_N$, it then suffices to show that $\ii$ commutes with $\Omega_N$ on a basis element $|y_1\dots y_N\rangle_x^+$ of $\XXZ{N;t,x}{+}$. Let $a = \frac{1}{2}(t-d)$ and $j_{a+1}=N+1$. Note that $\Omega_N \ii |y_1\dots y_N\rangle_x^+ = \Omega_N F^{(a)}_+ |y_1\dots y_N\rangle_z^+  = \Gamma_{j_1=1}^{\Omega_N\ii}\delta_{y_1,+}+\Gamma_{j_1\neq 1}^{\Omega_N\ii}$ where
\begin{align*} \Gamma_{j_1=1}^{\Omega_N\ii} &= \sum_{\substack{1= j_1<\dots <j_a\leq N\\ y_{j_2}=\dots  = y_{j_a} = +}} \Big(\prod_{k=1}^{a}\prod_{j=j_{k}+1}^{j_{k+1}-1}q^{ky_j}\Big)z|y_2\dots y_{j_2-1}(-)y_{j_2+1}\dots y_{j_3-1}(-)y_{j_3+1}\dots y_N(-)\rangle_z^+,\\\Gamma_{j_1\neq 1}^{\Omega_N\ii} &= \sum_{\substack{2\leq j_1<\dots <j_a\leq N\\ y_{j_1}=\dots  = y_{j_a} = +}} \Big(\prod_{k=1}^{a}\prod_{j=j_{k}+1}^{j_{k+1}-1}q^{ky_j}\Big)z^{-y_1}|y_2 \dots y_{j_1-1}(-)y_{j_1+1}\dots y_{j_2-1}(-)y_{j_2+1}\dots y_Ny_1\rangle_z^+.
\end{align*}
In the same way, $\ii\Omega_N|y_1\dots y_N\rangle_x^{+} = x^{-y_1}F^{(a)}_+|y_2\dots y_Ny_1\rangle_z^+ = \Gamma_{j_a = N}^{\ii\Omega_N}\delta_{y_1,+}+\Gamma_{j_a\neq N}^{\ii\Omega_N}$ where
\begin{align*} 
\Gamma_{j_a=N}^{\ii\Omega_N} &= \sum_{\substack{1\leq j_1<\dots <j_a= N\\ y_{j_1+1}=\dots  = y_{j_{a-1}+1} = +}} \Big(\prod_{k=1}^{a-1}\prod_{j=j_{k}+1}^{j_{k+1}-1}q^{ky_{j+1}}\Big)x^{-1}|y_2\dots y_{j_1}(-)y_{j_1+2}\dots y_{j_2}(-)y_{j_2+2}\dots y_N(-)\rangle_z^+,\\
\Gamma_{j_a\neq N}^{\ii\Omega_N} &= \sum_{\substack{1\leq j_1<\dots <j_a\leq N-1\\ y_{j_1+1}=\dots  = y_{j_a+1} = +}} \Big(\prod_{k=1}^{a}\prod_{j=j_{k}+1}^{j_{k+1}-1-\delta_{k,a}}q^{ky_{j+1}}\Big)q^{ay_1}x^{-y_1}|y_2 \dots y_{j_1}(-)y_{j_1+2}\dots y_{j_2}(-)y_{j_2+2}\dots y_Ny_1\rangle_z^+.
\end{align*}
The translation in $\ii\Omega_N|y_1\dots y_N\rangle_x^{+}$ moved the positions of the spins acted upon by $F_+^{(a)}$. To compare $\Gamma_{j_1=1}^{\Omega_N \ii}$ and $\Gamma_{j_a=N}^{\ii\Omega_N}$, it is natural to substitute $j_i\mapsto j_{i-1}+1$ for $1\leq i\leq a+1$ in $\Gamma_{j_1=1}^{\Omega_N \ii}$. The power of $q$ given to the spins between $j_i+1$ and $j_{i+1}$ should then be adjusted accordingly: $q^{ky_j}\mapsto q^{(k+1)y_{j+1}}$. These steps give 
\begin{align*}
\Gamma_{j_1=1}^{\Omega_N \ii} &=  \sum_{\substack{0=j_0< j_1<\dots <j_{a}= N\\ y_{j_1+1}=\dots  = y_{j_{a-1}+1} = +}} \Big(\prod_{k=0}^{a-1}\prod_{j=j_{k}+1}^{j_{k+1}-1}q^{(k+1) y_{j+1}}\Big)z|y_2\dots y_{j_1}(-)y_{j_1+2}\dots y_{j_2}(-)y_{j_2+2}\dots y_N(-)\rangle_z^+\\
&= \sum_{\substack{0=j_0< j_1<\dots <j_{a}= N\\ y_{j_1+1}=\dots  = y_{j_{a-1}+1} = +}} \Big(\prod_{k=0}^{a-1}\prod_{j=j_{k}+1}^{j_{k+1}-1}q^{k y_{j+1}}\Big)q^{\sum_{j=1}^N y_j-\sum_{k=1}^{a-1} y_{j_k+1}-y_{1}}z|y_2\dots y_{j_1}(-)y_{j_1+2}\dots y_{j_2}(-)y_{j_2+2}\dots y_N(-)\rangle_z^+
\\
&= \sum_{\substack{1\leq j_1<\dots <j_{a}= N\\ y_{j_1+1}=\dots  = y_{j_{a-1}+1} = +}} \Big(\prod_{k=1}^{a-1}\prod_{j=j_{k}+1}^{j_{k+1}-1}q^{k y_{j+1}}\Big)q^{t-a}z|y_2\dots y_{j_1}(-)y_{j_1+2}\dots y_{j_2}(-)y_{j_2+2}\dots y_N(-)\rangle_z^+ = \Gamma_{j_a=N}^{\ii\Omega_N}
\end{align*}
where we used, in the last equality, the relations $z^2 = q^{-t}$ and $zq^a = x$ defining condition B of \eqref{eq:conditionsAB2} to deduce $q^{t-a}z=x^{-1}$.
A similar computation gives $\Gamma_{j_1\neq 1}^{\Omega_N \ii} = \Gamma_{j_a\neq N}^{\ii\Omega_N}$. Therefore, $\Omega_N\ii|y_1\dots y_N\rangle_x^+ = \ii\Omega_N|y_1\dots y_N\rangle_x^+$ and $\ii$ is indeed $\atl{N}$-linear.

We now want to show that $\jj^{+}_{(t,x);(d,z)}$ 
is $\atl{N}$-linear with the hypothesis that $(d,z)\unlhd (t,x)$ through condition A. For this purpose, note that the latter hypothesis is equivalent to $(-t,x^{-1})\unlhd (-d,z^{-1})$ via condition A. The above hence tells us that the map $F^-_{(-t,x^{-1});(-d,z^{-1})}:\XXZ{N;-d,z^{-1}}{-}\rightarrow \XXZ{N;-t,x^{-1}}{-}$ is $\atl{N}$-linear. The claim then follows from the fact that, by \eqref{eq:lienFetG}, the map $\jj^{+}_{(t,x);(d,z)}$ that we wish to study is equal, up to a non-zero factor, to the conjugation of $F^-_{(-t,x^{-1});(-d,z^{-1})}$ by the spin flip isomorphism $s$. The proof for $\jj^{-}_{(t,x);(d,z)}$ is analogous and is omitted.

We now consider sequence \eqref{eq:SEC1} with the hypotheses mentioned in the statement of the theorem and with $a=\frac{1}{2}(t-d)$. It is clear that all maps appearing in this sequence are $\atl{N}$-linear as $(d,z)\unlhd (t,x)$ and $(t,x)\preceq (D,Z)$ via condition B. Let $a_1,a_2 \in \mathbb{Z}_{\geq 0}$ be such that $a=a_1\ell+a_2$ with $a_2<\ell$. The definition of condition B gives $q^D = x^{-2} = (zq^{a})^{-2} = z^{-2}q^{d-t} = q^d$ so that $D = d+2n\ell$ for some $n\in \mathbb{N}$. We then have $0\leq D-t = 2(n\ell-a) = 2((n-a_1)\ell-a_2) \leq 2\ell$ as $(t,x)\preceq (D,Z)$ is a direct succession and we can thus conclude that $n = a_1+1$. Therefore, as the action of the divided powers on the XXZ chains does not depend on the twist parameter $z$ (see lemma \ref{lem:lesFetG}), the exactness of sequence \eqref{eq:SEC1} is equivalent to that of the following sequence of $\mathbb{C}$-linear morphisms
$$\XXZ{N;D,z}{+} \xrightarrow{\ F^{(\ell-a_2)}_+\ } \XXZ{N;t,z}{+} \xrightarrow{\ F^{(a_1\ell+a_2)}_+\ } \XXZ{N;d,z}{+}$$
The exactness of the latter sequence follows directly from theorem \ref{thm:SECX} if $a\not\equiv 0\modY\ell$ as the assumptions of this theorem are in this case equivalent to $a_2\not\equiv 0\modY\ell$ and $t \geq a$ (which is true since we suppose $t\geq|d|\geq -d$). 

For sequence \eqref{eq:SEC2}, note that, as $(d,z)\unlhd(t,x)$ through condition B with $t\geq |d|\geq -d$, we have $(-d,z^{-1})\unlhd (t,x)$ through condition A and both maps $\ii_{(-d,z^{-1});(t,x)}^-$ and $\ii_{(t,x);(h,v)}^-$ are $\atl N$-linear. Let $a = \frac{1}{2}(t+d)$ and choose $a_1,a_2\in \mathbb{Z}_{\geq 0}$ as above such that $a = a_1\ell+a_2$ with $a_2<\ell$. Since $(t,x)\preceq (h,v)$ also by condition A, $q^{h} = x^2 = (z^{-1}q^{-a})^{2}= z^{-2}q^{-t-d} = q^{-d}$ so $h=-d+2n\ell$ for some $n\in \mathbb{N}$.
However, this implies $0\leq h-t = 2(n\ell-a)=2((n-a_1)\ell-a_2)\leq 2\ell$ as the succession $(t,x)\preceq (h,v)$ is direct and we thus have once more $n=a_1+1$. The exactness of sequence \eqref{eq:SEC2} is therefore equivalent to that of the sequence of $\mathbb{C}$-linear maps 
$$\XXZ{N;h,z}-\xrightarrow{F^{(\ell-a_2)}_-}\XXZ{N;t,z}-\xrightarrow{F^{(a_1\ell+a_2)}_-}\XXZ{N;-d,z}-$$
where we changed again all twist parameters to $z$. The exactness of this last sequence is guaranteed by theorem \ref{thm:SECX} when $a\not\equiv 0\modY\ell$ as the hypotheses of this theorem amount here to $a_2 \not\equiv 0\modY\ell$ and $t\geq a$ (which follows from $t\geq |d| \geq d)$.

For the last statement, note that $\frac{1}{2}(t\pm d) \equiv 0\text{ mod }\ell$ implies $t\pm d=2n\ell$ for some $n\in \mathbb{Z}$. In this case, $q^t=q^{\mp d}$ and condition B forces $q^d=z^{\pm 2}$. Thus, if $(t,x)$ succeeds $(d,z)$ via condition B, $q^d \neq z^{\pm 2}$ implies that $\frac{1}{2}(t\pm d) \not\equiv 0 \modY \ell$.
\end{proof}
The linear maps $\ii_{(d,z);(t,x)}^\pm$ of the last theorem are related by the $\circ$-duality.
\begin{lemma}\label{lemma:AlmCommSquare} Let $(d,z),(s,y)\in \Delta$ be such that $(d,z)\unlhd (s,y)$ through condition A. Then, 
\begin{equation}\label{eq:AlmCommSquare}
\varrho_{d,z}(\ii_{(d,z^{-1});(s,y^{-1})}^+)^{\circ}=
q^{k(s-k)}\ii^-_{(d,z);(s,y)}\varrho_{s,y}
\end{equation}
with $k= \frac{1}{2}(s-d)$ and where $\varrho_{d,z}:(\XXZ{N;d,z^{-1}}{+})^{\circ}\to \XXZ{N;d,z}{-}$ and $\varrho_{s,y}:(\XXZ{N;s,y^{-1}}{+})^{\circ}\to \XXZ{N;s,y}-$ are the isomorphisms described in the proof of proposition \ref{prop:foncteurcirc}.
\end{lemma}
Clearly, by using the functor $\circ$ and the spin flip $s$, other relations of the form \eqref{eq:AlmCommSquare} may be obtained for successors via condition B or for the maps $\jj^{\pm}_{(d,z);(s,y)}$ of section \ref{sec:xxzLUq}. The above lemma showcases a particular relation, the only one needed for the proof of subcase (iii) of theorem \ref{thm:main}.
\begin{proof} The basis $\BXXZ{N;s}$ of $\XXZ{N;s,y^{-1}}{+}$ induces a basis for $(\XXZ{N;s,y^{-1}}{+})^{\circ}$ and it suffices to check relation \eqref{eq:AlmCommSquare} on a given element $|x_1...x_N\rangle^+_{y^{-1}}$ of this induced basis. Recall that $\circ$ acts as the identity on $\atl{N}$-morphisms so that, with our identifications,   
\begin{align*}
\rho_{d,z}(\ii^+_{(d,z^{-1});(s,y^{-1})})^{\circ}|x_1\dots x_N\rangle_{y^{-1}}^+ &= \rho_{d,z}(F^{(k)}_+ |x_1\dots x_N\rangle_{z^{-1}}^+)\\
&= \sum_{\substack{1\leq j_1 < ... < j_k\leq N\\ x_{j_1}=...=x_{j_k}=+}} \Big(\prod_{n=1}^k\prod_{j=j_{n}+1}^{j_{n+1}-1} q^{nx_j}\Big)
\varrho_{d,z}|x_1\dots x_{j_1-1}(-)x_{j_1+1}\dots x_{j_k-1}(-)x_{j_k+1}\dots x_N\rangle_{z^{-1}}^{+}
\intertext{where $j_0=0$ and $j_{k+1} = N+1$. The definition of $\varrho_{d,z}$ (see the proof of proposition \ref{prop:foncteurcirc}) then gives}
&= \sum_{\substack{1\leq j_1 < ... < j_k\leq N\\ x_{j_1}=...=x_{j_k}=+}} \Big(\prod_{n=1}^k\prod_{j=j_{n}+1}^{j_{n+1}-1} q^{nx_j}\Big)
|x_N\dots x_{j_k+1}(-)x_{j_k-1}\dots x_{j_{1}+1}(-)x_{j_{1}-1}\dots x_1\rangle_{z}^{-}.
\end{align*}
On the other hand, $q^{k(s-k)}\ii^-_{(d,z);(s,y)}\varrho_{s,y}|x_1...x_N\rangle_{y^{-1}}^+ = q^{k(s-k)}F^{(k)}_-|x_N...x_1\rangle_{z}^-$ which is equal to
\begin{equation*}
q^{k(s-k)}\sum_{\substack{1\leq j_1 < ... < j_k\leq N\\ x_{N-j_1+1}=...=x_{N-j_k+1}=+}} \Big(\prod_{n=1}^k\prod_{j=j_{n}+1}^{j_{n+1}-1} q^{-nx_{N-j+1}}\Big)
|x_N\dots x_{N-j_1+2}(-)x_{N-j_1}\dots x_{N-j_k+2}(-)x_{N-j_k}\dots x_1\rangle_z^-.
\end{equation*}
Using the change of variables $j_i\mapsto N-j_{k-i+1}+1$ for $0\leq i \leq k+1$ allows us to obtain the equivalent expression
\begin{equation*}
q^{k(s-k)}\sum_{\substack{1\leq j_1 < ... < j_k\leq N\\ x_{j_1}=...=x_{j_k}=+}} \Big(\prod_{n=1}^k\prod_{j=N-j_{k-n+1}+2}^{N-j_{k-n}} q^{-nx_{N-j+1}}\Big)
|x_N\dots x_{j_k+1}(-)x_{j_k-1}\dots x_{j_{1}+1}(-)x_{j_{1}-1}\dots x_1\rangle_z^-
\end{equation*}
which is equal to 
\begin{equation*}
q^{k(s-k)}\sum_{\substack{1\leq j_1 < ... < j_k\leq N\\ x_{j_1}=...=x_{j_k}=+}} 
\Big(\prod_{n=0}^{k-1}\prod_{j=j_{n}+1}^{j_{n+1}-1} q^{(n-k)x_{j}}\Big)|x_N\dots x_{j_k+1}(-)x_{j_k-1}\dots x_{j_{1}+1}(-)x_{j_{1}-1}\dots x_1\rangle_z^-
\end{equation*}
after the substitutions $n\mapsto k-n$ and $j\mapsto N-j+1$ in the products. The coefficient corresponding to a given choice of increasing sequence $1\leq j_1<...<j_n\leq N$ with $x_{j_1}=...=x_{j_n}=+$ in the preceding sum is
\begin{equation*}
q^{k(s-k)}\prod_{n=0}^{k-1}\prod_{j=j_{n}+1}^{j_{n+1}-1} q^{(n-k)x_{j}}
=q^{k(s-k)}\prod_{n=0}^{k}\prod_{j=j_{n}+1}^{j_{n+1}-1} q^{(n-k)x_{j}}
= q^{k(s-k)} q^{-k(\sum_{j=1}^N x_j-\sum_{n=1}^k x_{j_n})}
\Big(\prod_{n=1}^k\prod_{j=j_{n}+1}^{j_{n+1}-1} q^{nx_{j}}\Big) 
= \prod_{n=1}^k\prod_{j=j_{n}+1}^{j_{n+1}-1} q^{nx_{j}}
\end{equation*}
where we used the relation $\sum_{j=1}^N x_j = s$ for the last equality. This corresponds precisely to the coefficient found above for the same increasing sequence $1\leq j_1<...<j_n\leq N$ in $\rho_{d,z}(\ii^+_{(d,z^{-1});(s,y^{-1})})^{\circ}|x_1...x_N\rangle_{y^{-1}}^{+}$. The result thus follows.
\end{proof}
Theorem \ref{thm:f} has failed to assure the exactness of sequence \eqref{eq:SEC1} (resp.~\eqref{eq:SEC2}) when $q^d$ is $z^2$ (resp.~$z^{-2}$). The proof of theorem \ref{thm:main} for these situations will require other morphisms. Fix $n\in \mathbb{Z}_{\geq 0}$ and consider the following cases:
\begin{itemize}
\setlength{\itemsep}{2pt}
\item[(a)] Let $\kk^{(n)} = \ii_{(d,z);(d+2n\ell,zq^{n\ell})}^+: \XXZ{N;d+2n\ell,zq^{n\ell}}{+}\rightarrow \XXZ{N;d,z}{+}$. If $q^d = z^{-2}$, then $(d,z)\unlhd (d+2n\ell,zq^{n\ell})$ through condition B and, by theorem \ref{thm:f}, $\kk^{(n)}$ is $\atl{N}$-linear.
\item[(b)] Consider $\mm^{(n)}=s\ii^{-}_{(-d-2n\ell,z^{-1}q^{n\ell});(-d,z^{-1})}s : \XXZ{N;d,z}{+}\rightarrow \XXZ{N;d+2n\ell,zq^{n\ell}}{+}$ where $s$ denotes the spin flip. If $q^d = z^{2}$, then $(-d-2n\ell,z^{-1}q^{n\ell})\unlhd(-d,z^{-1})$ through condition A and, by theorem \ref{thm:f}, $\mm^{(n)}$ is $\atl{N}$-linear.
\end{itemize}
When $q^d\neq z^{\pm 2}$, the pair $(d+2n\ell,zq^{n\ell})$ does not succeed $(d,z)$ via condition B and the pair $(-d,z^{-1})$ does not succeed $(-d-2n\ell,z^{-1}q^{n\ell})$ through condition A (for the ordering $\unlhd$). In these circumstances, the maps $\kk^{(n)}$ and $\mm^{(n)}$ exist but are not $\atl N$-linear. However, with or without this $\atl{N}$-linearity, these maps $\kk^{(n)}$ and $\mm^{(n)}$ are always injective or surjective. 
\begin{theorem}\label{thm:Aut} With the notation given above, the map $\mm^{(n)}\kk^{(n)}$ (resp.~$\kk^{(n)}\mm^{(n)}$) is one-to-one if $|d|\leq |d+2n\ell|$ (resp.~if $|d| \geq |d+2n\ell|$). In particular, $\kk^{(n)}$ is injective (resp.~surjective) if $|d|\leq |d+2n\ell|$ (resp.~if $|d|\geq |d+2n\ell|$) and the map $\mm^{(n)}$ is injective (resp.~surjective) if $|d|\geq |d+2n\ell|$ (resp.~if $|d|\leq |d+2n\ell|$).
\end{theorem}
\begin{proof}
This follows from theorem \ref{thm:AutX} and equation \eqref{eq:lienFetG} as the $\lusztpm$-action on $\XXZ{N;d,z}\pm$ is independent of the twist $z$. \end{proof}
\end{subsection}
\end{section}

%%%%%%%%%%%%%
%%%%%%%%%%%%%
%%%%%   %%%%%
%%%%% 5 %%%%%
%%%%%   %%%%%
%%%%%%%%%%%%%
%%%%%%%%%%%%%

%%%%%%%%%%%
%
% the structure of the XXZ chain as a aTL_n-module
%
%%%%%%%%%%%

\begin{section}{The structure of the XXZ spin chains as $\atl N$-modules}\label{sec:struc} 

This section proves theorem \ref{thm:main} that exhibits the structure of the eigenspace $\XXZ{N;d,z}+$ as a module over $\atl{N}$. Throughout the parameter $d$ of the $\atl N$-module $\XXZ{N;d,z}+$ is assumed to be non-negative as the spin flip of proposition \ref{prop:spinflip} can be used for the cases with $d<0$. We also assume that $(d,z)$ is not one of the problematic pairs (see section \ref{sec:main}) and leave the analysis of these pairs to appendix \ref{app:problematic}. Each of the three subsections tackles one the three subcases (i), (ii) and (iii) announced in section \ref{sec:main}. They all start by tying the conditions on $d$ and $q$ to the successors of $(d,z)\in\lambda$. 
\begin{subsection}{Subcase (i)}\label{sub:i}
This subsection considers the case of a pair $(d,z)\in\lambda_N$ such that $q^d=z^2$ and $z^4=1$ (so that $z^2=\pm 1$). In this case, a direct successor $(t,x)$ of $(d,z)$ via condition A or B must verify $q^t=z^2=q^d$ and thus $\frac12(t-d)\equiv 0\text{ mod }\ell$ as $t$ must share its parity with $d$. Also, as $q^2$ is a primitive $\ell$-th root of unity, the direct successors via A and B are both equal to $(d+2\ell, zq^\ell)$. The successors of $(d,z)$ are those given by subcase (i) of section \ref{sec:main} and the Loewy diagram of $\Cell{N;d,z}$ is
\begin{equation}\label{eq:succCasi}
(d,z)\longrightarrow (d+2\ell,zq^\ell)\longrightarrow(d+4\ell, z)\longrightarrow(d+6\ell, zq^\ell)\longrightarrow\dots
\end{equation}
by theorem \ref{thm:GL}. Theorem \ref{thm:main} states here that $\XXZ{N;d,z}+$ is semisimple and isomorphic to the direct sum of (the irreducibles associated to) all the successors $(d+2m\ell,zq^{m\ell})$ with $m\in \mathbb{Z}_{\geq 0}$ such that $d+2m\ell \leq N$. This is what needs to be proved. 
\begin{proof}[Proof (theorem \ref{thm:main}, subcase (i)).] Let $m_d=\max\{m\in\mathbb{Z}_{\geq0}\,|\,d+2m\ell\leq N\}$ and observe that $\Cell{N;d,z}$ has $m_d+1$ composition factors. We proceed by induction on $m_d$. If $m_d=0$, then $\Cell{N;d,z}\simeq\Irre{N;d,z}$ and $\XXZ{N;d,z}+\simeq \Irre{N;d,z}$ by proposition \ref{prop:troisnoeuds}.

Let now $m_d>0$. The first step is to use the induction hypothesis to obtain the structure of $\XXZ{N;d+2\ell,zq^\ell}+$. The successors of $(d+2\ell,zq^\ell)$ are those of \eqref{eq:succCasi} with the first $(d,z)$ omitted. Thus $m_{d+2\ell}=m_d-1$. Moreover, the pair $(d+2\ell,zq^\ell)$ also falls in subcase (i) as $q^{d+2\ell}=q^d=z^2=(zq^{\ell})^2$ and $(zq^\ell)^4=z^4=1$. The induction hypothesis then gives
$$\XXZ{N;d+2\ell,zq^\ell}+\simeq\bigoplus_{1\leq m\leq m_d}\Irre{N;d+2m\ell,zq^{m\ell}}.$$

The next step is to prove that the eigenspace $\XXZ{N;d,z}+$ has a submodule isomorphic to $\XXZ{N;d+2\ell,zq^\ell}+$. The $\atl N$-morphisms $\kk^{(1)}:\XXZ{N;d+2\ell,zq^\ell}+\to \XXZ{N;d,z}+$ and $\mm^{(1)}:\XXZ{N;d,z}+\to \XXZ{N;d+2\ell,zq^\ell}+$ of section \ref{sub:homoXXZ} (or appendix \ref{app:problematic} if $\ell=1$) will do this work. Indeed, theorem \ref{thm:Aut} (or theorem \ref{thm:Auti} if $\ell =1$) reveals that $\kk^{(1)}$ is injective and $\mm^{(1)}$ is surjective. Thus, as $\atl{N}$-modules,
$$\XXZ{N;d,z}+\simeq \XXZ{N;d+2\ell,zq^\ell}+\oplus \ker \mm^{(1)}.$$
A part of $\ker\mm^{(1)}$ is easily identified. Indeed, corollary \ref{thm:gpEstUnQuotient} indicates that the generic part $\pg{N;d,z}$ of $\Cell{N;d,z}$ is isomorphic to a quotient of the image of $\ay{N;d,z}:\Cell{N;d,z}\to\XXZ{N;d,z}+$. Since $\ay{N;d,z}$ is non-zero (lemma \ref{thm:iNonZero}), $\XXZ{N;d,z}+$ must at least contain the head of $\Cell{N;d,z}$, that is the module $\Irre{N;d,z}$. As this simple is not a composition factor of $\XXZ{N;d+2\ell,zq^\ell}+$, it must be in $\ker\mm^{(1)}$.

The last step is a dimensional argument. The above results give
$$\dim\XXZ{N;d,z}+ = \sum_{1\leq m\leq m_d}\dim \Irre{N;d+2m\ell,zq^{m\ell}}+\dim\ker\mm^{(1)} \geq \sum_{0\leq m\leq m_d}\dim \Irre{N;d+2m\ell,zq^{m\ell}}=\dim\Cell{N;d,z}.$$
However $\dim\XXZ{N;d,z}+= \dim\Cell{N;d,z}$ and $\ker\mm^{(1)}$ must then be isomorphic to the last factor identified: $\Irre{N;d,z}$. Therefore, 
$$\XXZ{N;d,z}+\simeq\bigoplus_{0\leq m\leq m_d}\Irre{N;d+2m\ell,zq^{m\ell}}$$as claimed.
\end{proof}
\end{subsection}
\begin{subsection}{Subcase (ii)}\label{sub:ii}
The subcase (ii) relaxes the condition $z^4=1$. Thus, throughout this subsection, $q^d=z^2$ or $q^d=z^{-2}$ but not both. Consider first the case $q^d=z^2$. Then the direct successor $(s,y)$ of $(d,z)$ via condition A satisfies $q^s=z^2=q^d$ and $s$ must be the smallest integer larger than $d$ such that $\frac12(s-d)\equiv 0\text{ mod }\ell$. Hence $s=d+2\ell$ and $y=zq^{-\ell}$ so that, with the notation of subcase (iii) of section \ref{sec:main}, the family $(s_a,y_a)$ is contained in $(d_a,z_a)$. Moreover a successor $(t,x)$ of $(d,z)$ through B is easily computed to be $(t,x)=(2i\ell-d,z^{-1}q^{i\ell})$ with $i\in \mathbb{Z}_{\geq 0}$ such that $d<t=2i\ell-d\leq d+2\ell$. Note that $t$ and $d+2\ell$ cannot be equal. If they were, we would have $z^4=q^{2d}=q^dq^t=z^2z^{-2}=1$ which contradicts the hypothesis $z^4\neq 1$. Thus $(t,x)$ is a strict successor of $(d,z)$ with $d<t<d+2\ell$. A similar computation shows that the family $(h_a,v_a)$ is also included in $(d_a,z_a)$ and the Loewy diagram of $\Cell{N;d,z}$ is indeed given by subcase (ii) of theorem \ref{thm:GL}
\begin{equation}\label{eq:succCasii}
(d,z)\longrightarrow (t,x)\longrightarrow (d+2\ell,zq^\ell)\longrightarrow (t+2\ell,xq^\ell)\longrightarrow(d+4\ell, z)\longrightarrow(t+4\ell, x)\longrightarrow\dots
\end{equation}
The statement of the main theorem for this subcase (ii) is that the Loewy diagram for $\XXZ{N;d,z}+$ is 
\begin{equation}\label{eq:succCasiiA}
\begin{tikzpicture}[baseline={(current bounding box.center)},scale=0.45]
\node (d0) at (0,2) [] {$(d,z)$};
\node (t0) at (4,0) [] {$(t,x)$};
\node (d1) at (8,2) [] {$(d+2\ell,zq^\ell)$};
\node (t1) at (12,0) [] {$(t+2\ell,xq^\ell)$};
\node (d2) at (16,2) [] {$(d+4\ell,z)$};
\node (t2) at (20,0) [] {$(t+4\ell,x)$};
\node (d3) at (24,2) [] {$(d+6\ell,zq^\ell)$};
\node (t3) at (28,0) [] {$\dots$};
\draw[->] (d0) -- (t0);
\draw[->] (d1) -- (t0);\draw[->] (d1) -- (t1);
\draw[->] (d2) -- (t1);\draw[->] (d2) -- (t2);
\draw[->] (d3) -- (t2);\draw[->, dashed] (d3) -- (t3);
\end{tikzpicture}
\end{equation}
where all composition factors with $d+2i\ell$ or $t+2i\ell$ larger than $N$ must be deleted.

The case $q^d=z^{-2}$ is similar. A strict successor $(s,y)$ of $(d,z)$ through condition A with $d<s<d+2\ell$ exists and the Loewy diagram of $\Cell{N;d,z}$ remains that of \eqref{eq:succCasii} (with the obvious changes $t\mapsto s$ and $x\mapsto y$). The corresponding diagram for $\XXZ{N;d,z}+$, as claimed by theorem \ref{thm:main}, is then
\begin{equation}\label{eq:succCasiiB}
\begin{tikzpicture}[baseline={(current bounding box.center)},scale=0.45]
\node (d0) at (0,0) [] {$(d,z)$};
\node (t0) at (4,2) [] {$(s,y)$};
\node (d1) at (8,0) [] {$(d+2\ell,zq^\ell)$};
\node (t1) at (12,2) [] {$(s+2\ell,yq^\ell)$};
\node (d2) at (16,0) [] {$(d+4\ell,z)$};
\node (t2) at (20,2) [] {$(s+4\ell,y)$};
\node (d3) at (24,0) [] {$(d+6\ell,zq^\ell)$};
\node (t3) at (28,2) [] {$\dots$};
\draw[<-] (d0) -- (t0);
\draw[<-] (d1) -- (t0);\draw[<-] (d1) -- (t1);
\draw[<-] (d2) -- (t1);\draw[<-] (d2) -- (t2);
\draw[<-] (d3) -- (t2);\draw[<-, dashed] (d3) -- (t3);
\end{tikzpicture}
\end{equation}
where again all composition factors with $d+2i\ell$ or $s+2i\ell$ larger than $N$ must be deleted. The proof of structures \eqref{eq:succCasiiA} or \eqref{eq:succCasiiB} for the module $\XXZ{N;d,z}+$ is the goal of this subsection.
\begin{proof}[Proof (theorem \ref{thm:main}, subcase (ii)).] The proof is analogous to the one given in the last subsection for subcase (i) and also proceeds by induction. The $n$-th inductive hypothesis applies to {\em all} $(D,Z)\in\lambda_N$ as long as $q^D \in \{Z^2,Z^{-2}\}$ with $Z^4 \neq 1$ and as long as the number of composition factors in $\Cell{N;D,Z}$ is $\leq n$. It states that 
\begin{itemize}
\item[(a)] If $q^D=Z^2$, the structure of $\XXZ{N;D,Z}+$ is given by \eqref{eq:succCasiiA} after making the changes $(d,z)\mapsto (D,Z)$ and $(t,x)\mapsto (T,X)$ for $(T,X)$ the direct successor of $(D,Z)$ through condition B.
\item[(b)] If $q^D=Z^{-2}$, the structure of $\XXZ{N;D,Z}+$ is given by \eqref{eq:succCasiiB} after making the changes $(d,z)\mapsto (D,Z)$ and $(s,y)\mapsto (S,Y)$ for $(S,Y)$ the direct successor of $(D,Z)$ through condition A.
\end{itemize} 
This hypothesis holds true if $n\leq 2$ by proposition \ref{prop:troisnoeuds} so it suffices to prove it for the pair $(D,Z)=(d,z)$ assuming that $\Cell{N,d;z}$ has $n \geq 3$ composition factors. We first settle the case $q^d = z^2$. The proof is divided in four steps.

The first step identifies the generic part $\pg{N;d,z}$ of $\Cell{N;d,z}$. The pair $(d+2\ell, zq^\ell)$ directly succeeds $(d,z)$ via condition A and its successors are those of \eqref{eq:succCasii} with $(d,z)$ and $(t,x)$ omitted. Thus, $\pg{N;d,z} = \Cell{N;d,z}/\im \gl{}\simeq \Cell{N;d,z}/\Cell{N;d+2\ell, zq^\ell}$ with $\gl{}=\gl{(d,z);(d+2\ell, zq^\ell)}:\Cell{N;d+2\ell, zq^\ell}\to\Cell{N;d,z}$ the injective $\atl{N}$-morphism of proposition \ref{prop:grandThm}. This generic part must then have the Loewy diagram $(d,z)\rightarrow (t,x)$ by \eqref{eq:succCasii}. The latter diagram must be contained in the one of $\XXZ{N;d,z}{+}$ as $\pg{N;d,z}$ is a subquotient of $\XXZ{N;d,z}{+}$ by corollary \ref{thm:gpEstUnQuotient}.

The second step aims at finding a ``large'' submodule $\mathsf M$ of $\XXZ{N;d,z}+$ that would contain the rest of its composition factors. Recall that $(t,x)=(2i\ell-d,z^{-1}q^{i\ell})$ where $i\in\mathbb{Z}_{\geq 0}$ is such that $d<t<d+2\ell$. The pair $(-t,x^{-1})$ thus satisfies $q^{-t}=(x^{-1})^2$ and the map $\mm^{(i)}:\XXZ{N;-t,x^{-1}}+=\XXZ{N;d-2i\ell,zq^{i\ell}}+\to\XXZ{N;d,z}+$ of section \ref{sec:xxzLUq} is then an injective $\atl N$-morphism by theorem \ref{thm:Aut} since $|-t|=2i\ell-d>d=|d|$. A good candidate for the submodule $\mathsf M$ may hence be found in $\im\mm^{(i)}\simeq \XXZ{N;-t,x^{-1}}+$. To investigate the structure of this image, note that $\XXZ{N;-t,x^{-1}}+\simeq \XXZ{N;t,x}-\simeq(\XXZ{N;t,x^{-1}}+)^{\circ}$ by propositions \ref{prop:spinflip} and \ref{prop:foncteurcirc}. Also, as noted before, $q^{t} = (x^{-1})^{-2}$ with $(x^{-1})^4 = z^4 \neq 1$. The structure of $\Cell{N;t,x^{-1}}$ is easily deduced using \eqref{eq:succCasii}, theorem \ref{thm:GL} and lemma \ref{thm:partialOrders}. It is given by the Loewy diagram
\begin{equation*}
(t,x^{-1})\longrightarrow (d+2\ell,z^{-1}q^{\ell})\longrightarrow (t+2\ell,x^{-1}q^\ell)\longrightarrow (d+4\ell,z^{-1})\longrightarrow(t+4\ell, x^{-1})\longrightarrow\dots
\end{equation*}
where all the arrows correspond to a direct succession through condition A. In particular, $\Cell{N;t,x^{-1}}$ has $n-1$ composition factors and the inductive hypothesis holds for $(D,Z)= (t,x^{-1})$. It says that $\XXZ{N;t,x^{-1}}+$ has the Loewy diagram (case b) 
$$ \begin{tikzpicture}[baseline={(current bounding box.center)},scale=0.45]
\node (d0) at (0,0) [] {$(t,x^{-1})$};
\node (t0) at (4,2) [] {$(d+2\ell,z^{-1}q^\ell)$};
\node (d1) at (8,0) [] {$(t+2\ell,x^{-1}q^\ell)$};
\node (t1) at (12,2) [] {$(d+4\ell,z^{-1})$};
\node (d2) at (16,0) [] {$(t+4\ell,x^{-1})$};
\node (t2) at (20,2) [] {$(d+6\ell,z^{-1}q^\ell)$};
\node (d3) at (24,0) [] {$\dots$};
\draw[->] (t0) -- (d0);
\draw[->] (t0) -- (d1);\draw[->] (t1) -- (d1);
\draw[->] (t1) -- (d2);\draw[->] (t2) -- (d2);
\draw[->] (t2) -- (d3);
\end{tikzpicture}$$
The Loewy diagram of the submodule $\mathsf{M} = \im\mm^{(i)}\simeq (\XXZ{N;t,x^{-1}}{+})^{\circ}$ of $\XXZ{N;d,z}{+}$ may then be obtained from this last diagram by taking the $\circ$-dual of each node (see subsection \ref{sec:dualities}). These $\circ$-duals are given by proposition \ref{prop:foncteurcirc} and the diagram of $\mathsf{M}$ is easily seen to be precisely \eqref{eq:succCasiiA} with the node $(d,z)$ omitted.

The third step introduces a well-chosen quotient of $\XXZ{N;d,z}{+}$, that is the image of the surjective $\atl{N}$-homomorphism $\mm^{(1)}:\XXZ{N;d,z}{+}\rightarrow \XXZ{N;d+2\ell,zq^{\ell}}{+}$ described by theorem \ref{thm:Aut} (which applies $|d|=d\leq d+2\ell = |d+2\ell|$). The goal here is not to obtain any additional composition factors for the eigenspace $\XXZ{N;d,z}+$, but to reveal the structure of the quotient. It will be useful later to show that all the arrows of the Loewy diagram of $\XXZ{N;d,z}+$ are already contained in \eqref{eq:succCasiiA}. To obtain this structure, note that the Loewy diagram of $\Cell{N;d+2\ell,zq^{\ell}}$ is exactly \eqref{eq:succCasii} without the nodes $(t,x)$ and $(d,z)$. Since $\Cell{N;d+2\ell,zq^{\ell}}$ has $<n$ composition factors and $q^{d+2\ell}=z^2$ with $(zq^{\ell})^4 =z^4\neq 1$, the inductive hypothesis for $(D,Z) = (d+2\ell,zq^{\ell})$ gives the Loewy diagram of the quotient $\im \mm^{(1)}\simeq \XXZ{N;d+2\ell,zq^{\ell}}+$. This Loewy diagram is precisely \eqref{eq:succCasiiA} with the nodes $(t,x)$ and $(d,z)$ omitted. 

Finally the proof of the inductive hypothesis for $(D,Z)=(d,z)$ when $q^d=z^2$ concludes with the following observations. First, the composition factors found so far for $\XXZ{N;d,z}+$ are precisely those of $\Cell{N;d,z}$ and, since these two modules have the same dimension, there is no other composition factor in $\XXZ{N;d,z}+$. This allows us to conclude that the two factors $(d,z)$ and $(t,x)$ are those of the kernel of $\mm^{(1)}$ so that the generic part $\pg{N;d,z}$ is isomorphic to a submodule of $\XXZ{N;d,z}{+}$. Second, all arrows in \eqref{eq:succCasiiA} must be present in the Loewy diagram of $\XXZ{N;d,z}+$ as they all already appear in the diagram of either $\pg{N;d,z}$ or of $\mathsf{M} \simeq (\XXZ{N;t,x^{-1}}{+})^{\circ}$ which are two submodules of $\XXZ{N;d,z}{+}$. The only remaining question is whether the Loewy diagram of $\XXZ{N;d,z}+$ has more arrows tying composition factors. If such an arrow exists, it cannot have both source and target in $\pg{N;d,z}$ or both in $\mathsf M$ since this would contradict the structure found for these submodules. It cannot either have a source in $\mathsf M$ heading toward $(d,z)$ since $\mathsf M$ is a submodule and does not have $\Irre{N;d,z}$ as a composition factor. Finally, the arrow cannot start from $(d,z)$ and end at some node other than $(d,z)$ or $(t,x)$ as this would contradict the fact that $(d,z)\to(t,x)$ is a submodule of $\XXZ{N;d,z}{+}$. The structure \eqref{eq:succCasiiA} is hence precisely the Loewy diagram of $\XXZ{N;d,z}+$ if $q^d= z^2$.

The case $q^d=z^{-2}$ remains. To study it, remark that the pair $(D,Z)=(d,z^{-1})$ solves $q^D=Z^2$. The Loewy diagram of $\XXZ{N;d,z^{-1}}+$ may thus be deduced from the preceding analysis and is given by
$$\begin{tikzpicture}[baseline={(current bounding box.center)},scale=0.45]
\node (d0) at (0,2) [] {$(d,z^{-1})$};
\node (t0) at (4,0) [] {$(s,y^{-1})$};
\node (d1) at (8,2) [] {$(d+2\ell,z^{-1}q^\ell)$};
\node (t1) at (12,0) [] {$(s+2\ell,y^{-1}q^\ell)$};
\node (d2) at (16,2) [] {$(d+4\ell,z^{-1})$};
\node (t2) at (20,0) [] {$(s+4\ell,y^{-1})$};
\node (d3) at (24,2) [] {$(d+6\ell,z^{-1}q^\ell)$};
\node (t3) at (28,0) [] {$\dots$};
\draw[->] (d0) -- (t0);
\draw[->] (d1) -- (t0);\draw[->] (d1) -- (t1);
\draw[->] (d2) -- (t1);\draw[->] (d2) -- (t2);
\draw[->] (d3) -- (t2);\draw[->, dashed] (d3) -- (t3);
\end{tikzpicture}$$
where the nodes have been expressed, using lemma \ref{thm:partialOrders}, in terms of the direct successor $(s,y)$ of $(d,z)$ through condition A. Moreover, by theorem \ref{thm:isoByDual}, $(\XXZ{N;d,z^{-1}}+)^\star\simeq \XXZ{N;d,z}+$ and the Loewy diagram of $\XXZ{N;d,z}+$ is hence obtained by flipping all the arrows in the above diagram while changing all the nodes by their $\star$-duals. These $\star$-duals are given in proposition \ref{thm:isoByDual} and the diagram obtained for $\XXZ{N;d,z}{+}$ is then precisely \eqref{eq:succCasiiB}.
\end{proof}
\end{subsection}
%
% (iii)
%
\begin{subsection}{Subcase (iii)}\label{sub:iii}
The introductory paragraphs of subsections \ref{sub:i} and \ref{sub:ii} have %shown
demonstrated that the equalities $q^d=z^2$ and $q^d=z^{-2}$ imply coincidences among the four 
families $(d_a,z_a), (s_a,y_a),(t_a,x_a)$ and $(h_a,v_a)$ appearing in subcase (iii) of section \ref{sec:main}. In fact, these coincidences occur only when one of these two equalities holds.
\begin{lemma}\label{lemma:IntFam}
There is a non-empty intersection between the families $(d_a,z_a)$, $(s_a,y_a)$, $(t_a,x_a)$ and $(h_a,v_a)$ if and only if one of the two equalities $q^d= z^2$ or $q^d= z^{-2}$ holds.
\end{lemma} 
\begin{proof} The intersections occurring when $q^d=z^2$ or $q^d=z^{-2}$ have already been identified in the previous subsections and the sufficiency is thus already proven. For the necessity, the condition $d_0<\min(s_0,t_0)\leq\max(s_0,t_0)<h_0$ show that one of the following must be true whenever a coincidence between families occurs:
\begin{align*}
\text{(i)}\quad & (s_0,y_0)=(t_0,x_0),\qquad &\text{(ii)}\quad &(h_0,v_0)=(d_1,z_1),\qquad &\text{(iii)}\quad &(s_0,y_0)=(d_1,z_1),\\
\text{(iv)}\quad &(h_0,v_0)=(t_1,x_1),\qquad &\text{(v)}\quad &(t_0,x_0)=(d_1,z_1), \qquad &\text{(vi)}\quad &(h_0,v_0)=(s_1,y_1).
\end{align*}
In case (iii), the equality $s_0=d_1$ implies $s=d+2\ell$ and, as $s$ solves $z^2=q^s$, it follows that $q^d=q^{s-2\ell}=z^2$. For case (v), $t_0=d_1$ means $-s+\delta_t=d+2\ell$ and it follows again that $q^d=q^{-s+\delta_t-2\ell}=q^{-s}=z^{-2}$ since $\delta_t$ is a multiple of $2\ell$. Cases (iv) and (vi) are similar. For case (i), $s_0=t_0$ forces $s=-s+\delta_t$ so that $s=\frac{1}{2}\delta_t$. The equality $y_0=x_0$ then gives $zq^{-k}=z^{-1}q^{k+\delta_t/2}=zq^k$, that is $q^{2k}=1$, and it follows that $z^{2}=q^{s}=q^{s-2k}=q^d$. Case (ii) is analogous to case (i).
\end{proof}
For the remaining of this section, $q^d$ is neither $z^2$ nor $z^{-2}$ and the (distinct) successors of $(d,z)$ are those of section \ref{sec:main}:
\begin{equation}\label{eq:SuccCasiii}
\begin{tikzpicture}[baseline={(current bounding box.center)},scale=0.45]
\node (k0) at (0,2) [] {$(d,z)$};
\node (j0) at (5,2) [] {$(s,y)$};
\node (k1) at (10,2) [] {$(d+2\ell,zq^\ell)$};
\node (j1) at (15,2) [] {$(s+2\ell,yq^\ell)$};
\node (i0) at (5,0) [] {$(t,x)$};
\node (h0) at (10,0) [] {$(h,v)$};
\node (i1) at (15,0) [] {$(t+2\ell,xq^\ell)$};
\node (k2) at (20,2) [] {$\dots $};
\node (h1) at (20,0) [] {$\dots $};
\draw[<-, dashed] (k0) -- (j0);\draw[<-, dashed] (k0) -- (i0);
\draw[<-, dashed] (j0) -- (k1);\draw[<-, dashed] (j0) -- (h0);
\draw[<-, dashed] (i0) -- (k1);\draw[<-, dashed] (i0) -- (h0);
\draw[<-, dashed] (k1) -- (j1);\draw[<-, dashed] (k1) -- (i1);
\draw[<-, dashed] (h0) -- (j1);\draw[<-, dashed] (h0) -- (i1);
\draw[dashed,<-] (j1) -- (k2);\draw[dashed,<-] (j1) -- (h1);
\draw[dashed,<-] (i1) -- (k2);\draw[dashed,<-] (i1) -- (h1);
\end{tikzpicture}
\end{equation}
where a dashed arrow $(e,u)\begin{tikzpicture}[baseline={(current bounding box.center)},scale=0.4]\node at (0,-0.2) {$\phantom{.}$};\draw[<-, dashed] (0,0) -- (1.6,0);\end{tikzpicture}\,(f,w)$ means $(e,u)\preceq(f,w)$ and the indices of the first node of each family have been dropped: $(d_0,z_0)\mapsto (d,z)$, $(s_0,y_0)\mapsto (s,y)$, and so on. Recall that diagonal arrows tie successors through condition B and that horizontal ones tie successors through condition A. In the proof below, we shall assume that $(d,z)$ is not one of the problematic pairs, the proof for the latter pairs being postponed to appendix \ref{app:problematic}. Furthermore, in the diagram above and throughout this subsection, any nodes with $d_a,s_a,t_a$ or $h_a$ larger than $N$ are to be removed. The main theorem states that, when $q^d$ is neither $z^2$ nor $z^{-2}$, the Loewy diagram of the module $\XXZ{N;d,z}+$ is 
\begin{equation}\label{eq:subCaseiii}
\begin{tikzpicture}[baseline={(current bounding box.center)},scale=0.45]
\node (d0) at (0,2) [] {$(d,z)$};
\node (s0) at (4,4) [] {$(s,y)$};
\node (t0) at (4,0) [] {$(t,x)$};
\node (h0) at (8,2) [] {$(h,v)$};
\node (d1) at (12,2) [] {$(d+2\ell,zq^\ell)$};
\node (s1) at (16,4) [] {$(s+2\ell,yq^\ell)$};
\node (t1) at (16,0) [] {$(t+2\ell,xq^\ell)$};
\node (h1) at (20,2) [] {$(h+2\ell,vq^\ell)$};
\node (d2) at (24,2) [] {$(d+4\ell,z)$};
\node (s2) at (28,4) [] {$\dots$};
\node (t2) at (28,0) [] {$\dots$};
\draw[->] (3,3.6) -- (0.8,2.5);% (s0) -- (d0);
\draw[->] (0.8,1.5) -- (3,0.5);% (d0) -- (t0);
\draw[->] (5,3.6) -- (7.8,2.5);% (s0) -- (h0);
\draw[->] (5,3.8) -- (11.8,2.5);% (s0) -- (d1);
\draw[->] (7.8,1.5) -- (5,0.5); %(h0) -- (t0);
\draw[->] (11.8,1.5) -- (5,0.3); %(d1) -- (t0);
\draw[->] (13.75,3.8) -- (8.2,2.5); % (s1) -- (h0);
\draw[->] (13.75,3.6) -- (12.2,2.5); % (s1) -- (d1);
\draw[->] (8.2,1.5) -- (14,0.3); %(h0) -- (t1);
\draw[->] (12.2,1.5) -- (14,0.5); %(d1) -- (t1);
\draw[->] (18,3.6) -- (19.8,2.5); %(s1) -- (h1);
\draw[->] (18,3.8) -- (23.8,2.5); %(s1) -- (d2);
\draw[->] (19.8,1.5) -- (18,0.5); %(h1) -- (t1);
\draw[->] (23.8,1.5) -- (18,0.3); %(d2) -- (t1);
\draw[->] (27,3.8) -- (20.2,2.5); %(s2) -- (h1);
\draw[->] (27,3.6) -- (24.2,2.5); %(s2) -- (d2);
\draw[->] (20.2,1.5) -- (27,0.3); %(h1) -- (t2);
\draw[->] (24.2,1.5) -- (27,0.5); %(d2) -- (t2);
\end{tikzpicture}
\end{equation}
where now the head of the module has been drawn on the top line and the socle on the bottom. We will use the following lemma which is based on a more general result to appear in \cite{JonathanEtCie}\footnote{We thank J.~Bellet\^ete for communicating to us not only the result but also the outline of its proof.}.
%
% Jonathan
%
\begin{lemma}\label{lem:jonathan} Fix $(d_1,z_1),(d_2,z_2)\in\Lambda_N$ with $q^{d_1}$ neither $z_1^2$ nor $z_1^{-2}$. If $|d_1-d_2|>2\ell$, then $\text{\rm Ext}^1_{\atl n}(\Irre{N;d_1,d_1},\Irre{N:d_2,z_2})$ is zero and the Loewy diagram of an indecomposable module cannot contain the arrow $(d_1,z_1)\rightarrow(d_2,z_2)$.
\end{lemma}
\begin{proof}[Proof (theorem \ref{thm:main}, subcase (iii) for $(d,z)$ not a problematic pair).] The proof again proceeds by induction. However, in this case, the inductive hypothesis --- yet to be stated --- does not merely claim the structure \eqref{eq:subCaseiii} but also characterizes the image of some of the maps given in theorem \ref{thm:f}. To understand why this is useful, note that the exact sequence \eqref{eq:SEC1} gives (here the direct successor $(D,Z)$ of $(t,x)$ through condition B is $(d+2\ell,zq^{\ell})$ as $(d,z)\preceq (t,x)$ is a direct succession)
\begin{equation*}
\XXZ{N;d+2\ell,zq^\ell}+\xrightarrow{\ii_{(t,x);(d+2\ell,zq^\ell)}^+}\XXZ{N;t,x}+\xrightarrow{\ii_{(d,z);(t,x)}^+}\XXZ{N;d,z}+,
\end{equation*}
implies that the submodule $\im \ii_{(d,z);(t,x)}^+$ of $\XXZ{N;d,z}+$ is isomorphic to the quotient $\XXZ{N;t,x}+/\im \ii_{(t,x);(d+2\ell,zq^\ell)}^+$. The structure of this submodule could thus be deduced in an inductive proof from the Loewy diagrams of $\XXZ{N;t,x}{+}$ and $\im \ii_{(t,x);(d+2\ell,zq^\ell)}^+$. Analogously, using the spin flip of proposition \ref{prop:spinflip} on sequence \eqref{eq:SEC2} produces an exact sequence
$$\XXZ{N;h,v}-\xrightarrow{\ii_{(t,x);(h,v)}^-} \XXZ{N;t,x}-\xrightarrow{s\ii_{(-d,z^{-1});(t,x)}^{-}}\XXZ{N;d,z}+$$
which implies that the submodule $\im s\ii_{(-d,z^{-1});(t,x)}^{-}$ of $\XXZ{N;d,z}+$ is isomorphic to the quotient $\XXZ{N;t,x}-/\im \ii_{(t,x);(h,v)}^-$. We claim that this quotient is also isomorphic to $(\XXZ{N;t,x^{-1}}{+}/\im \ii^{+}_{(t,x^{-1});(h,v^{-1})})^{\circ}$. Indeed, as $(t,x)\preceq (h,v)$ directly via condition A, lemma \ref{lemma:AlmCommSquare} gives
\begin{equation*}
\varrho_{t,x}(\ii_{(t,x^{-1});(h,v^{-1})}^+)^{\circ}=q^{k(h-k)}\ii^-_{(t,x);(h,v)}\varrho_{h,v}
\end{equation*}
with $k=\frac{1}{2}(h-t)$ and where $\varrho_{t,x}:(\XXZ{N;t,x^{-1}}{+})^{\circ}\rightarrow\XXZ{N;t,x}{-}$ and $\varrho_{h,v} : (\XXZ{N;h,v^{-1}}{+})^{\circ}\rightarrow\XXZ{N;h,v}{-}$ are the isomorphisms described in the proof of proposition \ref{prop:foncteurcirc}. It follows that $\rho_{t,x}^{-1}$ induces an isomorphism 
$$\XXZ{N;t,x}{-}/\im \ii^-_{(t,x);(h,v)} \simeq (\XXZ{N;t,x^{-1}}{+})^{\circ}/\im (\ii_{(t,x^{-1});(h,v^{-1})}^+)^{\circ}= (\XXZ{N;t,x^{-1}}{+})^{\circ}/(\im \ii_{(t,x^{-1});(h,v^{-1})}^+)^{\circ}$$ 
where the equality $\im (\ii_{(t,x^{-1});(h,v^{-1})}^+)^{\circ}=(\im \ii_{(t,x^{-1});(h,v^{-1})}^+)^{\circ}$ was deduced from the exactitude and covariance of $\circ$. The claim above then follows from the application of $\circ$ on the canonical exact sequence
$$ 0 \to  \im \ii_{(t,x^{-1});(h,v^{-1})}^+ \to \XXZ{N;t,x^{-1}}{+} \to \XXZ{N;t,x^{-1}}{+}/\im \ii_{(t,x^{-1});(h,v^{-1})}^+\to 0$$
as this gives 
\begin{equation}\label{eq:isoQuot}
(\XXZ{N;t,x^{-1}}{+}/\im \ii_{(t,x^{-1});(h,v^{-1})}^+)^{\circ}\simeq (\XXZ{N;t,x^{-1}}{+})^{\circ}/(\im \ii_{(t,x^{-1});(h,v^{-1})}^+)^{\circ}\simeq \XXZ{N;t,x}{-}/\im \ii^-_{(t,x);(h,v)}.
\end{equation}
The Loewy diagram of the submodule $\im s\ii_{(-d,z^{-1});(t,x)}^-\subseteq \XXZ{N;d,z}{+}$ can thus again be deduced in an inductive proof from the structure of $\XXZ{N;t,x^{-1}}{+}$ and $\im \ii^{+}_{(t,x^{-1});(h,v^{-1})}$. To rephrase and summarize, adding a characterization of the image of the maps $\ii_{(d,z);(t,x)}^{+}$ in the inductive hypothesis will allow us to infer efficiently the structure of important submodules at the following induction steps.

Finally here is the $n$-th inductive hypothesis. It states that, for any $(D,Z)\in \lambda_N$ not a problematic pair such that $\Cell{N;D,Z}$ has at most $n$ composition factors and with $D$ distinct from both $Z^{2}$ and $Z^{-2}$, the following two statements hold:
\begin{itemize}
\item[(a)] The Loewy diagram of $\XXZ{N;D,Z}+$ is the one given below where $(S,Y)$ and $(T,X)$ are the direct successors of $(D,Z)$ through conditions A and B respectively, and with $(H,V)$ the immediate successor of $(T,X)$ through condition A (which is also the direct successor of $(S,Y)$ through condition B). 
\item[(b)] The image $\im \ii_{(D,Z);(T,X)}^+\subseteq \XXZ{N;D,Z}+$ is the submodule constituted of the nodes of the families $(T_a,X_a)$ and $(H_a,V_a)$ (for $a\geq 0$) joined by the red arrows in the diagram below.
\end{itemize}
\begin{equation}\label{eq:subCaseiiiHypoInd}
\begin{tikzpicture}[baseline={(current bounding box.center)},scale=0.45]
\node (d0) at (0,2) [] {$(D,Z)$};
\node (s0) at (4,4) [] {$(S,Y)$};
\node (t0) at (4,0) [] {$(T,X)$};
\node (h0) at (8,2) [] {$(H,V)$};
\node (d1) at (12,2) [] {$(D+2\ell,ZQ^\ell)$};
\node (s1) at (16,4) [] {$(S+2\ell,YQ^\ell)$};
\node (t1) at (16,0) [] {$(T+2\ell,XQ^\ell)$};
\node (h1) at (20,2) [] {$(H+2\ell,VQ^\ell)$};
\node (d2) at (24,2) [] {$(D+4\ell,Z)$};
\node (s2) at (28,4) [] {$\dots$};
\node (t2) at (28,0) [] {$\dots$};
\draw[->] (3,3.6) -- (0.8,2.5);% (s0) -- (d0);
\draw[->] (0.8,1.5) -- (3,0.5);% (d0) -- (t0);
\draw[->] (5,3.6) -- (7.8,2.5);% (s0) -- (h0);
\draw[->] (5,3.8) -- (11.8,2.5);% (s0) -- (d1);
\draw[thick,red, ->] (7.8,1.5) -- (5,0.5); %(h0) -- (t0);
\draw[->] (11.8,1.5) -- (5,0.3); %(d1) -- (t0);
\draw[->] (13.75,3.8) -- (8.2,2.5); % (s1) -- (h0);
\draw[->] (13.75,3.6) -- (12.2,2.5); % (s1) -- (d1);
\draw[thick,red,->] (8.2,1.5) -- (13.8,0.3); %(h0) -- (t1);
\draw[->] (12.2,1.5) -- (13.9,0.5); %(d1) -- (t1);
\draw[->] (18,3.6) -- (19.8,2.5); %(s1) -- (h1);
\draw[->] (18,3.8) -- (23.8,2.5); %(s1) -- (d2);
\draw[thick,red,->] (19.8,1.5) -- (18,0.5); %(h1) -- (t1);
\draw[->] (23.8,1.5) -- (18.1,0.3); %(d2) -- (t1);
\draw[->] (27,3.8) -- (20.2,2.5); %(s2) -- (h1);
\draw[->] (27,3.6) -- (24.2,2.5); %(s2) -- (d2);
\draw[thick,red,->] (20.2,1.5) -- (27,0.3); %(h1) -- (t2);
\draw[->] (24.2,1.5) -- (27,0.5); %(d2) -- (t2);
\end{tikzpicture}
\end{equation}

The base cases of the induction are treated first. Take a pair $(D,Z)\in \lambda_N$ verifying the hypotheses (that is $q^D \not\in\{Z^2,Z^{-2}\}$ with $(D,Z)$ not a problematic pair) and suppose that $\Cell{N;D,Z}$ has at most $3$ composition factors. Then $D+2\ell > N$ as the contrary would imply\footnote{We use here the fact that $(D,Z)$ is not a problematic pair as we intrinsically use theorem \ref{thm:GL} for the structure of $\Cell{N;D,Z}$.} the existence of the four (distinct, by lemma \ref{lemma:IntFam}) nodes $(D,Z)$, $(S,Y)$, $(T,X)$ and $(D+2\ell,Zq^{\ell})$ in the Loewy diagram of the module $\Cell{N;D,Z}$. Statement (a) of the inductive hypothesis is thus here a consequence of proposition \ref{prop:troisnoeuds} whereas statement (b) is trivial if $T > N$. If $T \leq N$, sequence \eqref{eq:SEC1} becomes
\begin{equation*}
0 \rightarrow \Irre{N;T,X} \xrightarrow{\ii_{(D,Z);(T,X)}^+} \XXZ{N;D,Z}+
\end{equation*}
as $D+2\ell > N$ and as $\XXZ{N;T,X}+ \simeq \Irre{N;T,X}$ by proposition \ref{prop:troisnoeuds}. This ends the proof of the $n$-th inductive hypothesis for $n\leq 3$. 

We now turn to the proof of the inductive hypothesis for a pair $(D,Z)=(d,z)$ such that $\Cell{N;d,z}$ has $n> 3$ composition factors. The proof is once again divided in several steps and considers simultaneously the pairs $(d,z)$ and $(d,z^{-1})$. The latter pair also satisfies the hypotheses and its successors in $\lambda_N$ are linked to that of $(d,z)$ by lemma \ref{thm:partialOrders}. In particular, $\Cell{N;d,z^{-1}}$ also has $n$ composition factors.

The first step proves statement (b) of the inductive hypothesis. By the analysis done above, $\im \ii^+_{(d,z);(t,x)}\subseteq \XXZ{N;d,z}+$ is isomorphic to the quotient $\XXZ{N;t,x}+/\im \ii^+_{(t,x);(d+2\ell,zq^{\ell})}$. Recall also that $q^t = z^{-2}$ and that $x^2 = z^2q^{t-d}=q^{-d}$ as $(d,z)\preceq (t,x)$ via condition B. Thus, $q^t = x^2$ or $q^t = x^{-2}$ give respectively the contradictions $q^d = x^{-2} = q^{-t} = z^{2}$ or $q^d = x^{-2} = q^t=z^{-2}$ and we must conclude that $q^t\not\in\{x^2,x^{-2}\}$. Moreover, the successors of $(t,x)$ in $\lambda_N$ are precisely those appearing in \eqref{eq:SuccCasiii} with the omission of $(d,z)$ and $(s,y)$. The module $\Cell{N;t,x}$ hence has $n-2$ composition factors and the inductive hypothesis holds for $(D,Z)= (t,x)$. The Loewy diagram obtained for $\XXZ{N;t,x}+$ is
\begin{equation*}
\begin{tikzpicture}[baseline={(current bounding box.center)},scale=0.45]
\node (d0) at (0,2) [] {$(t,x)$}; %
\node (s0) at (4,4) [] {$(h,v)$}; %
\node (t0) at (4,0) [] {$(d+2\ell,zq^\ell)$}; %
\node (h0) at (8,2) [] {$(s+2\ell,yq^\ell)$}; %
\node (d1) at (12,2) [] {$(t+2\ell,xq^\ell)$}; %
\node (s1) at (16,4) [] {$(h+2\ell,vq^\ell)$}; %
\node (t1) at (16,0) [] {$(d+4\ell,z)$}; %
\node (h1) at (20,2) [] {$(s+4\ell,y)$}; %
\node (d2) at (24,2) [] {$(t+4\ell,x)$}; %
\node (s2) at (28,4) [] {$\dots$};
\node (t2) at (28,0) [] {$\dots$};
\draw[->] (3,3.6) -- (0.8,2.5);% (s0) -- (d0);
\draw[->] (0.8,1.5) -- (3,0.5);% (d0) -- (t0);
\draw[->] (5,3.6) -- (7.8,2.5);% (s0) -- (h0);
\draw[->] (5,3.8) -- (11.8,2.5);% (s0) -- (d1);
\draw[red,thick,->] (7.8,1.5) -- (6,0.5); %(h0) -- (t0);
\draw[->] (11.8,1.5) -- (6,0.3); %(d1) -- (t0);
\draw[->] (13.75,3.8) -- (8.2,2.5); % (s1) -- (h0);
\draw[->] (13.75,3.6) -- (12.2,2.5); % (s1) -- (d1);
\draw[red,thick,->] (8.2,1.5) -- (14,0.3); %(h0) -- (t1);
\draw[->] (12.2,1.5) -- (14,0.5); %(d1) -- (t1);
\draw[->] (18,3.6) -- (19.8,2.5); %(s1) -- (h1);
\draw[->] (18,3.8) -- (23.8,2.5); %(s1) -- (d2);
\draw[red,thick,->] (19.8,1.5) -- (18,0.5); %(h1) -- (t1);
\draw[->] (23.8,1.5) -- (18,0.3); %(d2) -- (t1);
\draw[->] (27,3.8) -- (20.2,2.5); %(s2) -- (h1);
\draw[->] (27,3.6) -- (24.2,2.5); %(s2) -- (d2);
\draw[red,thick,->] (20.2,1.5) -- (27,0.3); %(h1) -- (t2);
\draw[->] (24.2,1.5) -- (27,0.5); %(d2) -- (t2);
\end{tikzpicture}
\end{equation*}
where the image $\im \ii_{(t,x);(d+2\ell,zq^{\ell})}^+$ (also obtained with the help of the inductive hypothesis) has been identified with red arrows. The corresponding diagram for the submodule $\im \ii^+_{(d,z);(t,x)}\simeq \XXZ{N;t,x}+/\im \ii^+_{(t,x);(d+2\ell,zq^{\ell})}$ of $\XXZ{N;d,z}+$ is therefore
\begin{equation}\label{eq:S1Casiii}
\begin{tikzpicture}[baseline={(current bounding box.center)},scale=0.45]
\node (d0) at (0,0) [] {$(t,x)$};
\node (t0) at (4,2) [] {$(h,v)$};
\node (d1) at (8,0) [] {$(t+2\ell,xq^\ell)$};
\node (t1) at (12,2) [] {$(h+2\ell,vq^\ell)$};
\node (d2) at (16,0) [] {$(t+4\ell,x)$};
\node (t2) at (20,2) [] {$\dots$};
\draw[<-] (d0) -- (t0);
\draw[<-] (d1) -- (t0);\draw[<-] (d1) -- (t1);
\draw[<-] (d2) -- (t1);\draw[<-] (d2) -- (t2);
\end{tikzpicture}
\end{equation}
which is exactly the diagram predicted by statement (b) of the induction hypothesis.

The second step repeats step 1 for the pair $(d,z^{-1})$. As said before, the successors of this pair are linked to those of $(d,z)$ by lemma \ref{thm:partialOrders}. The direct successor of $(d,z^{-1})$ via condition B is thus $(s,y^{-1})$ and $(s,y^{-1})\preceq (h,v^{-1})$ directly via condition A. Step 1 hence gives the following Loewy diagram for $\im \ii^+_{(d,z^{-1});(s,y^{-1})}\subseteq \XXZ{N;d,z^{-1}}{+}$:
\begin{equation*}
\begin{tikzpicture}[baseline={(current bounding box.center)},scale=0.45]
\node (d0) at (0,0) [] {$(s,y^{-1})$};
\node (t0) at (4,2) [] {$(h,v^{-1})$};
\node (d1) at (8,0) [] {$(s+2\ell,y^{-1}q^\ell)$};
\node (t1) at (12,2) [] {$(h+2\ell,v^{-1}q^\ell)$};
\node (d2) at (16,0) [] {$(s+4\ell,y^{-1})$};
\node (t2) at (20,2) [] {$\dots$};
\draw[<-] (d0) -- (t0);
\draw[<-] (d1) -- (t0);\draw[<-] (d1) -- (t1);
\draw[<-] (d2) -- (t1);\draw[<-] (d2) -- (t2);
\end{tikzpicture}
\end{equation*}
Since $\XXZ{N;d,z}{+}\simeq (\XXZ{N;d,z^{-1}}{+})^\star$, the $\star$-dual of the above submodule is isomorphic to a quotient of $\XXZ{N;d,z}{+}$. The Loewy diagram of this quotient is obtained by flipping the above diagram upside-down while changing its nodes by their $\star$-duals (given by proposition \ref{thm:isoByDual}). The corresponding diagram is thus
\begin{equation}\label{eq:Q1Casiii}
\begin{tikzpicture}[baseline={(current bounding box.center)},scale=0.45]
\node (d0) at (0,2) [] {$(s,y)$};
\node (t0) at (4,0) [] {$(h,v)$};
\node (d1) at (8,2) [] {$(s+2\ell,yq^\ell)$};
\node (t1) at (12,0) [] {$(h+2\ell,vq^\ell)$};
\node (d2) at (16,2) [] {$(s+4\ell,y)$};
\node (t2) at (20,0) [] {$\dots$};
\draw[<-] (t0) -- (d0);
\draw[<-] (t0) -- (d1);\draw[<-] (t1) -- (d1);
\draw[<-] (t1) -- (d2);\draw[<-] (t2) -- (d2);
\end{tikzpicture}
\end{equation}

The third step studies another remarkable submodule of $\XXZ{N;d,z}{+}$, namely the image $\im s\ii^+_{(-d,z^{-1});(t,x^{-1})}$ introduced earlier in the proof. To obtain its structure, recall that $\im s\ii^+_{(-d,z^{-1});(t,x^{-1})}\simeq (\XXZ{N;t,x^{-1}}{+}/\im \ii^+_{(t,x^{-1});(h,v^{-1})})^{\circ}$. Observe also that, by lemma \ref{thm:partialOrders}, the pairs $(t,x)$ and $(t,x^{-1})$ have the same number of successors in $\lambda_N$ so that $\Cell{N;t,x^{-1}}$ has $n-2$ composition factors, just like $\Cell{N;t,x}$. Thus, as $q^t \not\in \{x^2,x^{-2}\}$, the inductive hypothesis holds for $(D,Z)=(t,x^{-1})$ and gives the following Loewy diagram for $\XXZ{N;t,x^{-1}}{+}$:
\begin{equation*}
\begin{tikzpicture}[baseline={(current bounding box.center)},scale=0.45]
\node (d0) at (0,2) [] {$(t,x^{-1})$}; %
\node (s0) at (4,4) [] {$(d+2\ell,z^{-1}q^{\ell})$}; %
\node (t0) at (4,0) [] {$(h,v^{-1})$}; %
\node (h0) at (7.75,2) [] {$(s+2\ell,y^{-1}q^\ell)$}; %
\node (d1) at (12.5,2) [] {$(t+2\ell,x^{-1}q^\ell)$}; %
\node (s1) at (16,4) [] {$(d+4\ell,z^{-1})$}; %
\node (t1) at (16,0) [] {$(h+2\ell,v^{-1}q^{\ell})$}; %
\node (h1) at (20,2) [] {$(s+4\ell,y^{-1})$}; %
\node (d2) at (24,2) [] {$(t+4\ell,x^{-1})$}; %
\node (s2) at (28,4) [] {$\dots$};
\node (t2) at (28,0) [] {$\dots$};
\draw[->] (3,3.6) -- (0.8,2.5);% (s0) -- (d0);
\draw[->] (0.8,1.5) -- (3,0.5);% (d0) -- (t0);
\draw[->] (6.5,3.6) -- (7.8,2.5);% (s0) -- (h0);
\draw[->] (6.5,3.8) -- (11.8,2.5);% (s0) -- (d1);
\draw[red,thick,->] (7.8,1.5) -- (5.25,0.55); %(h0) -- (t0);
\draw[->] (11.8,1.5) -- (5.25,0.25); %(d1) -- (t0);
\draw[->] (13.75,3.8) -- (8.2,2.5); % (s1) -- (h0);
\draw[->] (13.75,3.6) -- (12.2,2.5); % (s1) -- (d1);
\draw[red,thick,->] (8.2,1.5) -- (13.7,0.25); %(h0) -- (t1);
\draw[->] (12.2,1.5) -- (13.75,0.55); %(d1) -- (t1);
\draw[->] (18,3.6) -- (19.8,2.5); %(s1) -- (h1);
\draw[->] (18,3.8) -- (23.8,2.5); %(s1) -- (d2);
\draw[red,thick,->] (19.8,1.5) -- (18.2,0.55); %(h1) -- (t1);
\draw[->] (23.8,1.5) -- (18.28,0.25); %(d2) -- (t1);
\draw[->] (27,3.8) -- (20.2,2.5); %(s2) -- (h1);
\draw[->] (27,3.6) -- (24.2,2.5); %(s2) -- (d2);
\draw[red,thick,->] (20.2,1.5) -- (27,0.3); %(h1) -- (t2);
\draw[->] (24.2,1.5) -- (27,0.5); %(d2) -- (t2);
\end{tikzpicture}
\end{equation*}
In this diagram, we identified the image $\im \ii^+_{(t,x^{-1});(h,v^{-1})}$ with red arrows and we used lemma \ref{thm:partialOrders} to relate the successors of $(t,x^{-1})$ to those of $(t,x)$. The corresponding Loewy diagram for the quotient $\XXZ{N;t,x^{-1}}{+}/\im \ii^+_{(t,x^{-1});(h,v^{-1})}$ is thus
\begin{equation*}
\begin{tikzpicture}[baseline={(current bounding box.center)},scale=0.45]
\node (d0) at (0,0) [] {$(t,x^{-1})$};
\node (t0) at (4,2) [] {$(d+2\ell,z^{-1}q^\ell)$};
\node (d1) at (8,0) [] {$(t+2\ell,x^{-1}q^\ell)$};
\node (t1) at (12,2) [] {$(d+4\ell,z^{-1})$};
\node (d2) at (16,0) [] {$(t+4\ell,x^{-1})$};
\node (t2) at (20,2) [] {$\dots$};
\draw[<-] (d0) -- (t0);
\draw[<-] (d1) -- (t0);\draw[<-] (d1) -- (t1);
\draw[<-] (d2) -- (t1);\draw[<-] (d2) -- (t2);
\end{tikzpicture}
\end{equation*}
and the diagram for $(\XXZ{N;t,x^{-1}}{+}/\im \ii^+_{(t,x^{-1});(h,v^{-1})})^{\circ}\simeq \im s\ii^+_{(-d,z^{-1});(t,x^{-1})}\subseteq \XXZ{N;d,z}{+}$ may then be deduced by simply changing the nodes given above by their $\circ$-duals (given in proposition \ref{prop:foncteurcirc}), that is
\begin{equation}\label{eq:S2Casiii}
\begin{tikzpicture}[baseline={(current bounding box.center)},scale=0.45]
\node (d0) at (0,0) [] {$(t,x)$};
\node (t0) at (4,2) [] {$(d+2\ell,zq^\ell)$};
\node (d1) at (8,0) [] {$(t+2\ell,xq^\ell)$};
\node (t1) at (12,2) [] {$(d+4\ell,z)$};
\node (d2) at (16,0) [] {$(t+4\ell,x)$};
\node (t2) at (20,2) [] {$\dots$};
\draw[<-] (d0) -- (t0);
\draw[<-] (d1) -- (t0);\draw[<-] (d1) -- (t1);
\draw[<-] (d2) -- (t1);\draw[<-] (d2) -- (t2);
\end{tikzpicture}
\end{equation}
The fourth step repeats step 3 for $(d,z^{-1})$. This method shows that $\XXZ{N;d,z^{-1}}{+}$ has a submodule with the Loewy diagram 
\begin{equation*}
\begin{tikzpicture}[baseline={(current bounding box.center)},scale=0.45]
\node (d0) at (0,0) [] {$(s,y^{-1})$};
\node (t0) at (4,2) [] {$(d+2\ell,z^{-1}q^\ell)$};
\node (d1) at (8,0) [] {$(s+2\ell,y^{-1}q^\ell)$};
\node (t1) at (12,2) [] {$(d+4\ell,z^{-1})$};
\node (d2) at (16,0) [] {$(s+4\ell,y^{-1})$};
\node (t2) at (20,2) [] {$\dots$};
\draw[<-] (d0) -- (t0);
\draw[<-] (d1) -- (t0);\draw[<-] (d1) -- (t1);
\draw[<-] (d2) -- (t1);\draw[<-] (d2) -- (t2);
\end{tikzpicture}
\end{equation*}
and we can thus deduce as before that $\XXZ{N;d,z}{+}\simeq (\XXZ{N;d,z^{-1}}{+})^\star$ has a quotient with the following structure
\begin{equation}\label{eq:Q2Casiii}
\begin{tikzpicture}[baseline={(current bounding box.center)},scale=0.45]
\node (d0) at (0,2) [] {$(s,y)$};
\node (t0) at (4,0) [] {$(d+2\ell,zq^\ell)$};
\node (d1) at (8,2) [] {$(s+2\ell,yq^\ell)$};
\node (t1) at (12,0) [] {$(d+4\ell,z)$};
\node (d2) at (16,2) [] {$(s+4\ell,y)$};
\node (t2) at (20,0) [] {$\dots$};
\draw[<-] (t0) -- (d0);
\draw[<-] (t0) -- (d1);\draw[<-] (t1) -- (d1);
\draw[<-] (t1) -- (d2);\draw[<-] (t2) -- (d2);
\end{tikzpicture}
\end{equation}

The last step uses the map $\ay{N;d,z}:\Cell{N;d,z}\rightarrow \XXZ{N;d,z}{+}$ and dimensional analysis to deduce the structure of $\XXZ{N;d,z}+$. Observe that, by corollary \ref{thm:gpEstUnQuotient}, the Loewy diagram of the image $\im \ay{N;d,z}$ contains at least the diagram $(d,z)\to (t,x)$ of the generic part $\pg{N;d,z}$ of $\Cell{N;d,z}$. Repeating this last argument for $\XXZ{N;d,z^{-1}}+$ and taking the $\star$-duals adds an arrow $(s,y)\to (d,z)$ in the diagram of $\XXZ{N;d,z}{+}$. The above results then show that every composition factor of $\Cell{N;d,z}$ is contained in some subquotient of $\XXZ{N;d,z}{+}$, that is in \eqref{eq:S1Casiii}, \eqref{eq:Q1Casiii}, \eqref{eq:S2Casiii}, \eqref{eq:Q2Casiii} or in $\pg{N;d,z}$. The modules $\XXZ{N;d,z}{+}$ and $\Cell{N;d,z}$ have thus exactly the same composition factors as $\dim \Cell{N;d,z}=\dim \XXZ{N;d,z}+$ 
and the diagram of $\XXZ{N;d,z}{+}$ must be of the form shown below with perhaps some missing arrows. Red and blue arrows are used here to identify the submodules \eqref{eq:S1Casiii} and \eqref{eq:S2Casiii}.
\begin{equation}\label{eq:subCaseiiiHypoIndEnCouleur}
\begin{tikzpicture}[baseline={(current bounding box.center)},scale=0.45]
\node (d0) at (0,2) [] {$(d,z)$};
\node (s0) at (4,4) [] {$(s,y)$};
\node (t0) at (4,0) [] {$(t,x)$};
\node (h0) at (8,2) [] {$(h,v)$};
\node (d1) at (12,2) [] {$(d+2\ell,zq^\ell)$};
\node (s1) at (16,4) [] {$(s+2\ell,yq^\ell)$};
\node (t1) at (16,0) [] {$(t+2\ell,xq^\ell)$};
\node (h1) at (20,2) [] {$(h+2\ell,vq^\ell)$};
\node (d2) at (24,2) [] {$(d+4\ell,z)$};
\node (s2) at (28,4) [] {$\dots$};
\node (t2) at (28,0) [] {$\dots$};
\draw[->] (3,3.6) -- (0.8,2.5);% (s0) -- (d0);
\draw[->] (0.8,1.5) -- (3,0.5);% (d0) -- (t0);
\draw[->] (5,3.6) -- (7.8,2.5);% (s0) -- (h0);
\draw[->] (5,3.8) -- (11.8,2.5);% (s0) -- (d1);
\draw[thick, red, ->] (7.8,1.5) -- (5,0.5); %(h0) -- (t0);
\draw[thick,blue,->] (11.8,1.5) -- (5,0.3); %(d1) -- (t0);
\draw[->] (13.75,3.8) -- (8.2,2.5); % (s1) -- (h0);
\draw[->] (13.75,3.6) -- (12.2,2.5); % (s1) -- (d1);
\draw[thick,red,->] (8.2,1.5) -- (14,0.3); %(h0) -- (t1);
\draw[thick,blue,->] (12.2,1.5) -- (14,0.5); %(d1) -- (t1);
\draw[->] (18,3.6) -- (19.8,2.5); %(s1) -- (h1);
\draw[->] (18,3.8) -- (23.8,2.5); %(s1) -- (d2);
\draw[thick,red,->] (19.8,1.5) -- (18,0.5); %(h1) -- (t1);
\draw[thick,blue,->] (23.8,1.5) -- (18,0.3); %(d2) -- (t1);
\draw[->] (27,3.8) -- (20.2,2.5); %(s2) -- (h1);
\draw[->] (27,3.6) -- (24.2,2.5); %(s2) -- (d2);
\draw[thick,red,->] (20.2,1.5) -- (27,0.3); %(h1) -- (t2);
\draw[thick,blue,->] (24.2,1.5) -- (27,0.5); %(d2) -- (t2);
\end{tikzpicture}
\end{equation}
This is precisely \eqref{eq:subCaseiii} and it only remains to show that there are no additional arrows tying nodes in the diagram of $\XXZ{N;d,z}{+}$. Assume that such an extra arrow exists. Then, it cannot start from a node appearing in the submodules \eqref{eq:S1Casiii} or \eqref{eq:S2Casiii} nor end at a node belonging to the quotients \eqref{eq:Q1Casiii} or \eqref{eq:Q2Casiii}. There are then three cases left.

\noindent \textbf{\itshape Case 1.} Suppose that the arrow starts at $(d,z)$. Then, as $t>d$, it cannot end at $(t+2a\ell,xq^{a\ell})$ for $a\geq 1$ by lemma \ref{lem:jonathan}. Also, as said above, it cannot end at a node appearing in \eqref{eq:Q1Casiii} or \eqref{eq:Q2Casiii}. The only possible missing arrow starting at $(d,z)$ is thus an additional arrow of the form $(d,z)\to(t,x)$ but this would contradict the structure of $\pg{N;d,z}$.

\noindent \textbf{\itshape Case 2.} Assume that the arrow ends at $(d,z)$. Then, a reasoning dual to that of case 1 gives a contradiction with either lemma \ref{lem:jonathan} or with the structure of the generic part $\pg{N;d,z^{-1}}$ of $\Cell{N;d,z^{-1}}$.

\noindent \textbf{\itshape Case 3.} Assume that the arrow starts at $(s+2a\ell,yq^{a\ell})$ and ends at $(t+2b\ell,xq^{b\ell})$ for some $a,b \in \mathbb{Z}_{\geq 0}$. Then, $|a-b|\geq 2$ since an arrow between $(s+2a\ell,yq^{a\ell})$ and $(t+2(a\pm 1)\ell,xq^{(a\pm 1)\ell})$ would be redundant in \eqref{eq:subCaseiiiHypoIndEnCouleur}. However, as $s$ and $t$ are both strictly comprised between $d$ and $d+2\ell$, the inequality $|a-b|\geq 2$ implies 
\begin{equation*}\begin{cases}
s-t+2(a-b)\ell \geq s-t+4\ell > 2\ell, & \text{if } a>b \text{ and}\\ 
s-t+2(a-b)\ell\leq s-t-4\ell < -2\ell & \text{if } a < b
\end{cases}
\end{equation*}
so that $|s-t+2(a-b)\ell|>2\ell$ in both situations. The arrow $(s+2a\ell,yq^{a\ell})\to (t+2b\ell,xq^{b\ell})$ thus contradicts lemma \ref{lem:jonathan}.

This ends the proof of subcase (iii) for pairs that are not problematic.
\end{proof}
\noindent The subcase (iii) for problematic pairs is the only part of theorem \ref{thm:main} remaining to prove; this is done in appendix \ref{app:problematic}.

Remark that the above demonstration (and the corresponding one given in appendix \ref{app:problematic} for problematic pairs) shows more than subcase (iii) of theorem \ref{thm:main} as it also characterizes the image of the morphisms $\ii^+_{(d,z);(t,x)}$ and $s\ii^-_{(-d,z^{-1});(t,x)}$ of section \ref{sec:xxzLUq}. We cast these additional results in a corollary. 
\begin{corollary} 
For any $(d,z)\in \lambda_N$ such that $d$ is distinct from both $z^{2}$ and $z^{-2}$, the module $\XXZ{N;d,z}{+}$ has the Loewy diagram \eqref{eq:subCaseiiiHypoIndEnCouleur} where the red and blue arrows respectively give the structure of the submodules $\im \ii^+_{(d,z);(t,x)}$ and $\im s\ii^-_{(-d,z^{-1});(t,x)}$.
\end{corollary} 
The results of the last subsections also allow to fully characterize the image of the morphism $\ay{N;d,z}$ of section \ref{sec:atlOnXXZ} for $(d,z)$ not a problematic pair. Indeed, the absence of an arrow going from $(d,z)$ to its direct successor through condition A in the Loewy diagrams of theorem \ref{thm:main} shows that $\im \ay{N;d,z}$ is precisely the generic part $\pg{N;d,z}$ of the cellular module $\Cell{N;d,z}$. This is still true when the pair $(d,z)$ is problematic (see the end of appendix \ref{app:problematic}) and we therefore have the following refinement of corollary \ref{thm:gpEstUnQuotient}.
\begin{corollary}\label{cor:Imind} Let $q,z\in \mathbb{C}^{\times}$ and $0\leq d\leq N$ with $d\equiv_2 N$. Then, $\pg{N;d,z}$ is isomorphic to $\im \ay{N;d,z}$.
\end{corollary}
\end{subsection}
\end{section}

%%%%%%%%%%%%%
%%%%%%%%%%%%%
%%%%%   %%%%%
%%%%% 6 %%%%%
%%%%%   %%%%%
%%%%%%%%%%%%%
%%%%%%%%%%%%%

%%%%%%%%%%%
%
% conclusion
%
%%%%%%%%%%%

\begin{section}{Concluding remarks}\label{sec:conclusion}

Theorem \ref{thm:main} constitutes the main result of this paper. It describes the Loewy structure of the $S^z$-eigenspaces $\XXZ{N;d,z}{\pm}$ of the periodic XXZ spin-$\frac{1}{2}$ chains seen as $\atl N(\beta)$-modules. Many cases had to be studied depending on the genericity of the complex number $q\in\mathbb C^\times$, on the sign $\pm$, on the existence of successors to the pair $(d,z)$ and finally on the relationship between $q^d$ and $z^2$ that determines the subcases (i), (ii) and (iii) of section \ref{sec:main}. It is thus appropriate to recall the highlights of the proof. The two dualities $\star$ and $\circ$ were used to show that the Loewy diagram of $\XXZ{N;d,z}-$ is obtained from that of $\XXZ{N;d,z}+$ by flipping all arrows (corollary \ref{cor:doubleDual}). This result, with the isomorphism $\XXZ{N;d,z}+\simeq \XXZ{N;-d,z^{-1}}-$ induced from the spin flip $s$ of proposition \ref{prop:spinflip}, allowed us to limit ourselves to modules $\XXZ{N;d,z}+$ with a non-negative $d$. The study of the morphism $\ay{N;d,z}:\Cell{N;d,z}\to \XXZ{N;d,z}+$ culminated in proposition \ref{prop:troisnoeuds} that made explicit the structure of any chain $\XXZ{N;d,z}+$ such that the associated $\Cell{N;d,z}$ has a semisimple radical. These particular chains can have at most three composition factors. Despite its technical character, this proposition had two consequences: first, it gave the structure of $\XXZ{N;d,z}+$ for $q$ generic and, second, it provided the seed cases for the inductive proofs leading to the subcases (i), (ii), (iii) when $q$ is a root of unity and $(d,z)$ is not a problematic pair. Section \ref{sec:struc} was devoted to the latter inductive proofs and subcase (iii) was completed for the problematic pairs in appendix \ref{sub:problematic}. The proof presented for theorem \ref{thm:main} has the interesting consequence of identifying the image of the map $\ay{N;d,z}$ (corollary \ref{cor:Imind}).

New tools were developed to complete this proof. Two should be underlined. The most important are probably the $\atl{N}(\beta)$-morphisms between eigenspaces $\XXZ{N;d,z}\pm$ given in theorem \ref{thm:f}. These morphisms appear naturally when studying the action of Lusztig's quantum group $\luszt$ on the XXZ spin chains and help to partially recover in the periodic case the quantum Schur-Weyl duality that holds for the open spin-chain (section \ref{sub:homoXXZ}). Also of interest in its own right is the explicit realization 
obtained for the finite-dimensional indecomposable projective modules of $\luszt$ (theorem \ref{thm:ProjReal}). This realization turned out to be crucial in the (recursive, with \eqref{eq:SEC1} and \eqref{eq:SEC2}) analysis of the image of the $\atl{N}(\beta)$-morphisms of theorem \ref{thm:f}. With these new tools, a complete description of the  spaces $\Hom(\XXZ{N;d,z}+,\XXZ{N;t,x}\pm)$, for all pairs $(d,z), (t,z)\in\Delta_N$, might be within reach. We wish to come back to this question in a near future.

The results given here naturally lead to hindsights about the representation theory of other algebras. Obvious examples are the one-boundary Temperley-Lieb algebras, also known as the {\em blob algebras} \cite{martin1994blob} or as the Temperley-Lieb algebras of type B \cite{GLdiagram}. As for the affine ones, the one-boundary TL algebras $\tlb{N}(\beta_1,\beta_2)$ form a family of associative algebras parametrized by a positive integer $N$ and two complex parameters $\beta_1,\beta_2\in\mathbb C$. (A larger family depending of three complex parameters was also introduced recently \cite{morin2015boundary}.) The algebra $\tlb{N}=\tlb{N}(\beta_1,\beta_2)$ is however finite-dimensional and can generically be obtained as a quotient of the form $\tlb{N}\simeq \atl{N}(\beta)/\langle Y-y\cdot \id\rangle$ where the two-sided ideal $\langle Y-y\cdot \id\rangle$ is generated by the difference of a central element $Y \in \atl{N}(\beta)$ and a multiple $y\in\mathbb C$ of the identity of $\atl N(\beta)$ \cite{GLlienBlob,BGJS2018}. The constant $y$ depends on the three parameters $\beta$, $\beta_1$ and $\beta_2$.
This definition induces an equivalence between the category of finite-dimensional $\tlb{N}$-modules and the one of finite-dimensional $\atl{N}(\beta)$-modules $M$ with $(Y-y\cdot \id)M = 0$. The equivalence preserves (in both directions) simple modules and non-split short exact sequences. The central element $Y$ also has a unique eigenvalue on any given XXZ-eigenspace $\XXZ{N;d,z}{+}$ (this is clear whenever $\XXZ{N;d,z}{+}$ is indecomposable, that is, in subcases (ii) or (iii), but requires a bit more work in subcase (i) where $\XXZ{N;d,z}{+}$ is semisimple). The eigenspace $\XXZ{N;d,z}{+}$ is thus naturally a $\tlb{N}$-module (for some $\beta_1,\beta_2$) and the Loewy diagram associated to this $\tlb{N}$-action is still of the form described in theorem \ref{thm:main}. In other terms, $\XXZ{N;d,z}{+}$ has the same structure as a $\tlb{N}$-module than as a $\atl{N}(\beta)$-module.

This structure, given in theorem \ref{thm:main}, is also reminiscent of that of Feigin-Fuchs modules over the Virasoro algebra $\mathfrak{Vir}$, which are infinite-dimensional representations arising naturally in the Coulomb gas free field realization of this algebra and play a role in the study of logarithmic conformal field theories (cf. \cite{feigin1982invariant, morin2015boundary}). This paper, and especially theorem \ref{thm:main}, may thus be seen as a step toward a better understanding of the link between the representation theory of the Lie algebra $\mathfrak{Vir}$ and that of the associative $\atl{N}(\beta)$ (or $\tlb{N}(\beta_1,\beta_2)$ by the above comments). Such a link is a tantalizing possibility since the continuum limit of spin chains are often tied to conformal field theories where $\mathfrak{Vir}$ is the symmetry algebra. The inquiry into this question started with the pioneering work of Koo and Saleur \cite{koo1994representations} and is still pursued actively. It is however fair to say that a complete mathematical undertanding of this link is still to come.

\end{section} %

\appendix
%%%%%%%%%%%%%
%%%%%%%%%%%%%
%%%%%   %%%%%
%%% app a %%%
%%%%%   %%%%%
%%%%%%%%%%%%%
%%%%%%%%%%%%%

%%%%%%%%%%%
%
% q-gymnastics
%
%%%%%%%%%%%
\begin{section}{Some properties of $q$-numbers}\label{app:a}

This appendix recalls the definitions of $q$-numbers, factorials and binomials while stating some of their properties. 

The $q$-numbers $[n]_q$ and $q$-factorial $[n]_q!$ are, for $n\in\mathbb Z$,
$$[n]_q=\frac{q^n-q^{-n}}{q-q^{-1}}\quad \text{and}\quad
[n]_q!=\begin{cases}[n]_q\cdot[n-1]_q!&\text{if }n\geq 1,\\
                   1& \text{if }n=0,\\
                   0&\text{otherwise.}
\end{cases}$$
The $q$-binomial coefficient, for $m,n\in\mathbb Z_{\geq 0}$, is defined by
$$\left[\begin{matrix}m\\n\end{matrix}\right]=\frac{[m]_q!}{[n]_q![m-n]_q!}$$
if $m\geq n$ and is $0$ otherwise. Here are two propositions that generalise properties of the (usual) binomial coefficients.
\begin{proposition}[Pascal identity]\label{thm:qPascal} Let $q\in\mathbb C^\times$ and $m,n\in\mathbb N$ with $n\leq m-1$. Then,
$$\left[\begin{matrix}m\\ n\end{matrix}\right]_q=
q^{\mp n}\left[\begin{matrix}m-1\\ n\end{matrix}\right]_q +
q^{\pm(m-n) }\left[\begin{matrix}m-1\\ n-1\end{matrix}\right]_q.
$$
\end{proposition}
\begin{proposition}[$q$-binomial theorem]\label{thm:qBinome} Let $q\in\mathbb C^\times$ and $m\in\mathbb Z_{\geq 0}$. Then,
$$\sum_{n=0}^m (-1)^nq^{n(m+1)}\left[\begin{matrix}m\\ n\end{matrix}\right]_q=
\prod_{n=1}^m(1-q^{2n}).
$$
\end{proposition}
The remaining results describe the behavior of $q$-numbers when $q$ is a root of unity. Strictly speaking, the left-hand side of the next identity should be understood within a limit $\lim_{q\to q_c}$ where $q_c$ is the root of unity considered.
\begin{proposition}[Lucas $q$-theorem \cite{StumQuiros}]\label{thm:qLucas} Let $q\in\mathbb C^\times$ be such that $q^2$ be a primitive $\ell$-th root of unity. Write $m,n\in\mathbb Z_{\geq0}$ as $m=m_1\ell+m_2$ and $n=n_1\ell+n_2$ with $m_1,n_1\geq 0$ and $0\leq m_2,n_2<\ell$. Then,
$$\left[\begin{matrix}m\\ n\end{matrix}\right]_q=q^{\ell(n_1 \ell(n_1-m_1)-m_2n_1-m_1n_2)}
\left(\begin{matrix}m_1\\ n_1\end{matrix}\right)\left[\begin{matrix}m_2\\ n_2\end{matrix}\right]_q$$
with $\big(\begin{smallmatrix}m_1\\ n_1\end{smallmatrix}\big)=0$ if $m_1<n_1$. In particular $\big[\begin{smallmatrix}m\\ n\end{smallmatrix}\big]_q=0$ if and only if $m_1<n_1$ or $m_2< n_2$.
\end{proposition}
\noindent As with the previous result, the next identities are to be understood multiplied by other terms and within limit signs. 
\begin{proposition}\label{prop:qnblim} If $q\in\mathbb C^\times$ is such that $q^2$ is a primitive $\ell$-th root of unity and if $m,n\in\mathbb Z$, then $q^{\ell^2} = (-1)^{\ell+1}q^{\ell}$,
$$[m\ell\pm n]_q=\pm q^{m\ell}[n]_q\quad\text{and}\quad
[m\ell]_q=mq^{(1-m)\ell}[\ell]_q.
$$
\end{proposition}

\end{section}

%%%%%%%%%%%%%
%%%%%%%%%%%%%
%%%%%   %%%%%
%%% app b %%%
%%%%%   %%%%%
%%%%%%%%%%%%%
%%%%%%%%%%%%%

%%%%%%%%%%%
%
% simple, weyl and projective of LU_1(sl_2)
%
%%%%%%%%%%%

\begin{section}{Representation theory of Lusztig's extension $\luszt$}\label{app:b}
This appendix is devoted to the proof of theorems \ref{thm:ProjReal} and \ref{thm:fusLUq}. These theorems are purely representation-theoretic in nature and are of interest in their own right. Throughout $q \in \mathbb{C}^{\times}$ is such that $q^2$ is a $\ell$th-primitive root of unity with $\ell \geq 2$. We start by stating some basic facts. The first one can be easily deduced from the definition of the divided powers of $U_{\text{res}}$. 
\begin{lemma}\label{prop:RelationsUres} Let $k,n\in \mathbb{Z}_{\geq 0}$. Then, in $U_{\text{res}}$,
$$E^{(k)}E^{(n)} = \qbin{k+n}{n}{t}E^{(k+n)}\quad\text{with}\quad K E^{(n)} K^{-1} = t^{2n} E^{(n)}\quad\text{and}\quad K F^{(n)} K^{-1} = t^{-2n} F^{(n)}.$$
\end{lemma}
\begin{lemma}[\cite{ChariPressley}]\label{lemma:K0ellt} Let $n,c \in \mathbb{Z}$ with $n\geq 0$. Then, the element $\qbin{K;c}{n}{t}$ of the rational form $U_t\mathfrak{sl}_2$ defined by 
\begin{equation}
\qbin{K;c}{n}{t} = \prod_{m=1}^n\left(\frac{K t^{c+1-m}-K^{-1}t^{m-1-c}}{t^m-t^{-m}}\right)
\end{equation}
also belongs to $U_{\text{res}}$. Moreover the subalgebra generated by $\{K^{\pm 1}, \qbin{K;c}{n}{t}\,|\, n,c\in \mathbb{Z} \text{ with } n\geq 0\}\subseteq U_{\text{res}}$ is commutative.
\end{lemma}
\begin{proposition}[\cite{ChariPressley}] In $\luszt$, the following identities hold: $E^{\ell}=F^{\ell}=K^{2\ell}-{\normalfont \id} = 0$.
\end{proposition}
\begin{proof} The first two assertions follow from the specialization of the relations $[\ell]_t!E^{(\ell)}=E^{\ell}$ and $[\ell]_t!F^{(\ell)}=F^{\ell}$ of $U_{\text{res}}$. For the last one, note that $K^{2\ell}-\id$ can be factorised as $\prod_{m=0}^{\ell-1} (K^2-q^{2m}\id)$. The definition of $\qbin{K;0}{\ell}{t}$ also gives
$$\prod_{m=1}^{\ell}(Kt^{1-m}-K^{-1}t^{m-1})=\qbin{K;0}{\ell}{t}\cdot \prod_{m=1}^{\ell}(t^m-t^{-m})$$
whose right-hand side contains a factor $(t^\ell-t^{-\ell})$ that is zero once specialized at $q$. The left-hand side of the definition becomes $\prod_{m=1}^{\ell}(Kq^{1-m}-K^{-1}q^{m-1})=K^{-\ell}q^{\ell(1-\ell)/2}\prod_{m=0}^{\ell-1}(K^2-q^{2m}\id)$ upon specialization and, since $K$ is invertible in $\luszt$, the product $\prod_{m=0}^{\ell-1}(K^2-q^{2m}\id)$, and thus $K^{2\ell}-\id$, must vanish.
\end{proof} 
\begin{subsection}{Explicit realisation of indecomposable projective modules}
For this section, fix integers $i,r,s \in \mathbb{Z}_{\geq 0}$ such that $i = r\ell+s$ and $s < \ell-1$. Let also $j = i+2(\ell-s-1)$ as in equation \eqref{fig:LoewyPq}.
\begin{lemma}\label{lemma:ExtWeyl} As vector spaces, $\Ext^1_{\luszt}(\Delta_q(i),\Delta_q(j))\simeq \mathbb{C}$.
\end{lemma}
\begin{proof} The short exact sequence of proposition \ref{prop:ProjAT} yields the long exact sequence in cohomology\footnote{Every object of this sequence should carry a subscript $\luszt$ which has been dropped for brevity.}
\begin{equation*}
0 \rightarrow \Hom(\Delta_q(i),\Delta_q(j)) \rightarrow \Hom(P_q(i),\Delta_q(j)) \rightarrow \End \Delta_q(j) \rightarrow \Ext^1(\Delta_q(i),\Delta_q(j)) \rightarrow \Ext^1(P_q(i),\Delta_q(j)) \rightarrow \dots
\end{equation*}
where the last written term is zero by projectivity of $P_q(i)$ in the category of finite-dimensional (type I) $\luszt$-modules. Hence, the extension group $\Ext^1(\Delta_q(i),\Delta_q(j))$ is a quotient of $\End \Delta_q(j)$ and, as the exact sequence of proposition \ref{prop:ProjAT} is non-split, it is enough to show that $\End\Delta_q(j)$ has dimension at most one to conclude the proof of the lemma. For that goal, apply the left-exact covariant functor $\Hom(\Delta_q(j),-)$ on the sequence given in proposition \ref{prop:Weyl} to get the exact sequence
\begin{equation}\label{eq:secExt}
0 \rightarrow \Hom(\Delta_q(j),L_q(i))\rightarrow \End \Delta_q(j) \rightarrow \Hom(\Delta_q(j),L_q(j))
\end{equation}
and note that $\Hom(\Delta_q(j),L_q(i)) = 0$ as the simple module $L_q(j)=\head \Delta_q(j)$ is not a subquotient of $L_q(i)$. Also, Schur's lemma gives $\Hom(\Delta_q(j),L_q(j)) \simeq \End L_q(j) \simeq \mathbb{C}$ so \eqref{eq:secExt} forces $\dim\End\Delta_q(j) \leq 1$ as desired.
\end{proof}
To provide an explicit realization for the indecomposable $\luszt$-module $P_q(i)$, it thus suffices to construct a non-trivial extension of the module $\Delta_q(j)$ by $\Delta_q(i)$. This extension will be obtained as the specialization (at $t=q$) of the module over the rational form $U_t\mathfrak{sl}_2$ given in the next proposition.
\begin{proposition}\label{prop:Tt} The $\mathbb{Q}(t)$-vector space $T_t(i)$ with basis $\{m_0,\dots,m_j,n_0,\dots,n_i\}$ is a $U_t\mathfrak{sl}_2$-module for the action 
\begin{equation*}
\begin{gathered}
Km_k = t^{j-2k}m_k, \qquad Kn_p = t^{i-2p}n_p, \qquad  E^{(v)}m_k = \qbin{j-k+v}{v}{t}m_{k-v}, \qquad F^{(v)}m_k = \qbin{k+v}{v}{t}m_{k+v},\\
 E^{(v)}n_p = \qbin{i-p+v}{v}{t}n_{p-v}+\gamma_{p,v}(t)m_{\ell-s+p-v-1}\quad\text{and}\quad F^{(v)}n_p = \qbin{p+v}{v}{t}n_{p+v}+\omega_{p,v}(t)m_{\ell-s+p+v-1}
\end{gathered}
\end{equation*}
where the coefficients $\gamma_{p,v}(t)$ and $\omega_{p,v}(t)$ are defined by
\begin{align*}
\gamma_{p,v}(t)& = \frac{1}{[v]_t!}\sum_{u=0}^{v-1} \qbin{\ell-s+p-u-2}{p-u}{t}\cdot \prod_{a=1}^u[i-p+a]_t\cdot \prod_{b=u+1}^{v-1}[i+\ell-s-p+b]_t\quad\text{and}\\
\omega_{p,v}(t) &= \begin{cases} {\displaystyle \frac{1}{[\ell-s+i]_t}\qbin{\ell-s+p+v-1}{p+v}{t}\qbin{p+v}{v}{t}} & \text{if } v > i-p,\\
0 & \text{if } v \leq i-p,\end{cases}
\end{align*}
with $\gamma_{p,v}(t) = \omega_{p,v}(t) = 0$ if $p < 0$ or $p>i$.
\end{proposition}
\begin{proof} It is not difficult to verify that the given action is compatible with the defining relations of $U_t\mathfrak{sl}_2$. The only slightly difficult check, which is of the relation $(t-t^{-1})[E,F]n_p = (K-K^{-1})n_p$, is now detailed:
\begin{align*}
[E,F]n_p &= E([p+1]_tn_{p+1}+\omega_{p,1}(t)m_{\ell-s+p})-F([i-p+1]_tn_{p-1}+\gamma_{p,1}(t)m_{\ell-s+p-2})\\
&= ([i-p]_t[p+1]_t-[i-p+1]_t[p]_t)n_{p}+\varsigma m_{\ell-s+p-1} = [i-2p]_t n_p + \varsigma m_{\ell-s+p-1}
\end{align*}
where the coefficient of $m_{\ell-s+p-1}$ is
\begin{align*}
\varsigma &= [p+1]_t
\gamma_{p+1,1}(t)-\gamma_{p,1}(t)[\ell-s+p-1]_t+\omega_{p,1}(t)[j-\ell+s-p+1]_t-[i-p+1]_t\omega_{p-1,1}(t).
\intertext{(Note that $\gamma_{p+1,1}(t)$ and $\omega_{p-1,1}(t)$ respectively vanish if $p=i$ and $p=0$.) Furthermore, if $v=1$, $\omega_{p,v}(t)$ takes a simpler form, namely $\omega_{p,1}(t)=\delta_{p,i}\genfrac{[}{]}{0pt}{1}{i+\ell-s-1}{i}_t$, as $p$ must satisfy $1=v>i-p$ while remaining in the range $0\leq p\leq i$. In particular, the term $[i-p+1]_t\omega_{p-1,1}(t)$ vanishes. The expressions for the $\gamma$'s and $\omega$'s then give} &= [p+1]_t (1-\delta_{p,i})
\qbin{\ell-s+p-1}{p+1}{t}-\qbin{\ell-s+p-2}{p}{t}[\ell-s+p-1]_t + \delta_{p,i}\qbin{i+\ell-s-1}{i}{t}[\ell-s-1]_t = 0
\end{align*}
where the factor $(1-\delta_{p,i})$ was added to handle the vanishing of $\gamma_{p+1,1}(t)$ when $p=i$. This forces as claimed $$(t-t^{-1})[E,F]n_p = (t-t^{-1})[i-2p]_tn_p = (K-K^{-1})n_p.$$

The only remaining thing to check is that the action of the divided powers agrees with $E^{(v)} = \frac{1}{[v]_t!}E^v$ and $F^{(v)} = \frac{1}{[v]_t!}F^v$ for any $v\in \mathbb{N}$. This is done by induction on $v$. There is nothing to prove when $v = 1$. In preparation for $v\geq 2$, note that $\gamma_{p,1}(t)=\genfrac{[}{]}{0pt}{1}{l-s+p-2}{p}_t$ so that the following recursive formula for the $\gamma$'s is easily obtained:
\begin{align*}
\qbin{i-p+v-1}{v-1}t &\gamma_{p-v+1,1}(t)+[i+\ell-s-p+v-1]_t\gamma_{p,v-1}(t) \notag\\
&=\frac1{[v-1]_t!}\gamma_{p-v+1,1}(t)\prod_{a=1}^{v-1}[i-p+a]_t+\frac1{[v-1]_t!}\sum_{u=0}^{v-2}\gamma_{p-u,1}(t)\prod_{a=1}^u [i-p+a]_t\cdot\prod_{b=u+1}^{v-1}[i+l-s-p+b]\notag\\
&=\frac1{[v-1]_t!}\sum_{u=0}^{v-1}\gamma_{p-u,1}(t)\prod_{a=1}^u [i-p+a]_t\cdot\prod_{b=u+1}^{v-1}[i+l-s-p+b]\\ 
&=[v]_t\gamma_{p,v}(t).\notag
\end{align*}
The induction hypothesis thus gives as claimed (all $\gamma$'s are evaluated at $t$)
\begin{align*}
\frac1{[v-1]_t!} E^vn_p&=E\left(\qbin{i-p+v-1}{v-1}t n_{p-v+1}+\gamma_{p,v-1}m_{\ell-s+p-v}\right)\\ 
&=\qbin{i-p+v-1}{v-1}t\big([i-p+v]_tn_{p-v}+\gamma_{p-v+1,1}m_{\ell-s+p-v-1}\big)
  +\gamma_{p,v-1}[j-\ell+s-p+v+1]_t m_{\ell-s+p-v-1}\\
&=[v]_t\left(\qbin{i-p+v}vt n_{p-v}+ \gamma_{p,v}m_{\ell-s+p-v-1}\right)\\
& =[v]_tE^{(v)}n_p. 
\end{align*}
This hypothesis also expresses $\frac1{[v-1]_t!} F^vn_p$ as a linear combination of $n_{p+v}$ and $m_{\ell-s+p+v-1}$. A quick computation brings the coefficient of $n_{p+v}$ to be the expected one for the action of $[v]_tF^{(v)}$. The coefficient of $m_{\ell-s+p+v-1}$ is in turn
$$\qbin{p+v-1}{v-1}t \omega_{p+v-1,1}(t)+[\ell-s+p+v-1]_t\omega_{p,v-1}(t).$$
The vanishing of $\omega_{p+v-1,1}(t)$ and $\omega_{p,v-1}(t)$ depends on whether $i-p$ is strictly smaller than $v-1$ (then $\omega_{p+v-1,1}(t)=0$), equal to $v-1$ (then $\omega_{p,v-1}(t)=0$) or strictly larger than $v-1$ (and then both $\omega$'s are zero and so is $\omega_{p,v}(t)$). A direct computation in the first two cases ($<$ and $=v-1$) shows that the non-vanishing term equals $[v]_t\omega_{p,v}(t)$, as required.
\end{proof}
From now on, we view $T_t(i)$ as a $\mathbb{Z}[t,t^{-1}]$-module and we restrict the above $U_t\mathfrak{sl}_2$-action to a $U_{\text{res}}$-action\footnote{The $U_{\text{res}}$-module thus obtained is not finitely-generated but will still give rise to a finitely-generated $\luszt$-module by specialization.}. In order to use this restricted $U_{\text{res}}$-module to obtain a $\luszt$-module by specialization, we need to show that the functions $\gamma_{p,v}(t)$ and $\omega_{p,v}(t)$ of proposition \ref{prop:Tt} have a well-defined limit as $t$ tends to $q$. This is the purpose of the next two lemmas. As before, $q\in\mathbb C^\times$ is such that $q^2$ is a $\ell$-th primitive root of unity. Note that a $t$-number $[n]_t$ can be written as $t^{-2(n-1)}(t^{2n}-1)/(t^2-1)$ and that the polynomial $(t^{2n}-1)/(t^2-1)$ has order $2(n-1)$ with all its roots distinct. In other words, the poles of $1/[n]_t$ are all simple. Moreover, by Lucas $q$-theorem \ref{thm:qLucas}, all $q$-binomial coefficients have a limit in $\mathbb C$ as $t\to q$.
\begin{lemma}\label{lemma:gammaLUq} The limit $\gamma_{p,v} = \lim_{t\rightarrow q}\gamma_{p,v}(t)$ exists and belongs to $\mathbb C$ for all $p\in\{0,\dots,i\}$ and $v \in \mathbb{N}$.\end{lemma}
\begin{proof} By the comments made before the statement, the denominator $1/[v]_t!$ is the only possible source of singularities in $\gamma_{p,v}(t)$. Also, the result holds trivially if $v<\ell$ as, in this case, $[v]_t!$ does not vanish when $t\to q$. If $v=\ell$, then the zero in $[v]_t$ is simple and the limit $\lim_{t\to q}\gamma_{\ell,v}(t)$ needs to be studied. It turns out that this is the only case to consider. Indeed, for $v>\ell$, lemma \ref{prop:RelationsUres} ties the two divided powers $E^{(v)}$ and $E^{(\ell)}$ by $\kappa E^{(v)}=(E^{(\ell)})^{v_1}E^{(v_2)}$ where $\kappa = \qbin{v}{v_2}t \qbin{v_1\ell}{\ell}t \qbin{(v_1-1)\ell}{\ell}t \dots \qbin{\ell}{\ell}t$ and where $v$ has been written as $v=v_1\ell+v_2$ with $v_1,v_2\in\mathbb Z_{\geq 0}$ and $v_2<\ell$. Note that $\lim_{t\to q}\kappa$ exists and is non-zero by theorem \ref{thm:qLucas}. Hence, for $v>\ell$, the limit $\lim_{t\to q}\gamma_{p,v}(t)$ can be obtained as the limit of a sum of $\gamma_{p',\ell}(t)$ and $\gamma_{p'',v'}(t)$ with $v'<\ell$ and $p', p''\in\{0,\dots, i\}$. The lemma thus rests on the study of $\lim_{t\to q}\gamma_{p,\ell}(t)$. The %rest 
end of the proof is devoted to it.

From now on $v=\ell$ and the goal is to prove that the sum appearing in $\gamma_{p,\ell}(t)$ has a vanishing limit. The difference $i-p$ is written as $p_1\ell+p_2$ with $p_1,p_2\in\mathbb Z_{\geq 0}$ and $p_2<\ell$ and the summand in $\gamma_{p,\ell}(t)$ as
\begin{equation*}
\lambda_u(t) = \qbin{\ell-s+p-u-2}{p-u}{t}\ \prod_{a=1}^u[i-p+a]_t \cdot \prod_{b=u+1}^{\ell-1}[i+\ell-s-p+b]_t
\end{equation*}
so that $\gamma_{p,\ell}(t) = \frac{1}{[\ell]_t!}\sum_{u=0}^{\ell-1} \lambda_u(t)$. We prove first that $\lim_{t\to q} \lambda_u(t)=0$ when $u$ is neither $\ell-p_2-1$ nor $s-p_2$.

\noindent {\bfseries \itshape Case 1: $\ell-p_2 \leq u$.} Then, $[(p_1+1)\ell]_t$ appears in the first product of $\lambda_u(t)$ when $a = \ell-p_2$, so the result follows.

\noindent  {\bfseries \itshape Case 2: $s-p_2 < u < \ell-p_2-1$.} Then the $q$-binomial in $\lambda_u(t)$ can be written as 
\begin{equation*}
\qbin{\ell-s+p-u-2}{p-u}{t} = \qbin{(r-p_1)\ell+\ell-p_2-u-2}{(r-p_1-1)\ell+\ell-p_2-u+s}{t}
\end{equation*}
whose limit is zero by theorem \ref{thm:qLucas} as $0\leq \ell-p_2-u-2<\ell-p_2-u+s< \ell$.

\noindent  {\bfseries \itshape Case 3: $u\leq s-(p_2+1)$.} Then, $[(p_1+1)\ell]_t$ appears in the second product of $\lambda_u(t)$ when $b = s-p_2$.

Two values of $u$ have escaped the above cases, namely $u=s-p_2$ and $u=\ell-p_2-1$. If $s<p_2$, only $\lambda_{\ell-p_2-1}(t)$ of these two possible $\lambda$'s will contribute to the sum $\sum_{u=0}^{\ell-1} \lambda_u(t)$ as $u$ takes non-negative values. However, in this case, $[(p_1+2)\ell]_t$ appears in the second product defining $\lambda_{\ell-p_2-1}(t)$ when $b = \ell-p_2+s$ and the aforementioned sum tends to zero. Hence only the case $s\geq p_2$ remains. A direct computation gives
\begin{align*}
\lim_{t\to q}\lambda_{s-p_2}(t)&=\qbin{(r-p_1)\ell+\ell-s-2}{(r-p_1)\ell}{q}\prod_{a=1}^{s-p_2}[i-p+a]_q\cdot \prod_{b=s-p_2+1}^{\ell-1}[i+\ell-s-p+b]_q,\\
\lim_{t\to q}\lambda_{\ell-p_2-1}(t)&=\qbin{(r-p_1-1)\ell+\ell-1}{(r-p_1-1)\ell+s+1}{q}\prod_{a=1}^{\ell-p_2-1}[i-p+a]_q\cdot \prod_{b=\ell-p_2}^{\ell-1}[i+\ell-s-p+b]_q.
\end{align*}
The $q$-binomials in these two limits can be simplified using theorem \ref{thm:qLucas} and proposition \ref{prop:qnblim}:
\begin{align*}
\qbin{(r-p_1)\ell+\ell-s-2}{(r-p_1)\ell}{q}&=q^{\ell(r-p_1)(\ell+s)},\\
\qbin{(r-p_1-1)\ell+\ell-1}{(r-p_1-1)\ell+s+1}{q}&=\qbin{\ell-1}{s+1}q q^{-\ell(r-p_1-1)(\ell+s)}=(-1)^{\ell+s}q^{\ell(r-p_1)(\ell+s)}.
\end{align*}
Gathering the common factors in 
$$\varsigma=q^{\ell(r-p_1)(\ell+s)}\prod_{a=1}^{s-p_2}[i-p+a]_q\cdot \prod_{b=\ell-p_2}^{\ell-1}[i+\ell-s-p+b]_q,$$
the limit $\lim_{t\to q}\sum_{u=0}^{\ell-1}\lambda_u(t)$ is then equal to 
\begin{align*}
\lim_{t\to q}\left(\lambda_{s-p_2}(t)+\lambda_{\ell-p_2-1}(t)\right)
&= \varsigma\left(\prod_{b=s-p_2+1}^{\ell-p_2-1}[i+\ell-s-p+b]_q+(-1)^{l+s}
\prod_{a=s-p_2+1}^{\ell-p_2-1}[i-p+a]_q\right)\\
\intertext{that gives, after changing the index in the first product to $c=\ell+s-2p_2-b$,}
&=\varsigma\left(\prod_{c=s-p_2+1}^{\ell-p_2-1}[2(p_1+1)\ell-(i-p+c)]_q+(-1)^{l+s}
\prod_{a=s-p_2+1}^{\ell-p_2-1}[i-p+a]_q\right)
\end{align*}
which is zero by theorem \ref{prop:qnblim}. This closes the argument.
\end{proof}
\begin{lemma}\label{lemma:omegaLUq}  The limit $\omega_{p,v} = \lim_{t\rightarrow q}\omega_{p,v}(t)$ exists and belongs to $\mathbb C$ for all $p\in\{0,\dots,i\}$ and $v \in \mathbb{N}.$\end{lemma}

\begin{proof} By Lucas $q$-theorem \ref{thm:qLucas}, the only potential singularity of $\omega_{p,v}$ comes from the $t$-number $[\ell-s+i]_t=[(r+1)\ell]_t$ appearing in its denominator. As the zero of this $t$-number at $t=q$ is of order one, it suffices to show that the product $\qbin{\ell-s+p+v-1}{p+v}{q}\qbin{p+v}{v}{q}$ appearing in $\omega_{p,v}$ vanishes when $m_{\ell-s+p+v-1} \neq 0$ and $i-p < v$. Write $v$ and $p$ in the usual forms (that is $v=v_1\ell+v_2$ and $p=p_1\ell+p_2$ with $v_1,v_2,p_1,p_2\in\mathbb Z_{\geq 0}$ and $v_2,p_2<\ell$) and consider the three following cases.

\noindent {\bfseries\itshape Case 1: $\ell\leq p_2+v_2$.} Then, $\qbin{p+v}{v}{q} = \qbin{(p_1+v_1+1)\ell+p_2+v_2-\ell}{v_1\ell+v_2}{q} = 0$ by theorem \ref{thm:qLucas} as $0\leq p_2+v_2-\ell < v_2<\ell$. 

\noindent {\bfseries\itshape Case 2: $s < p_2+v_2<\ell$.} Then, $\qbin{\ell-s+p+v-1}{p+v}{q} = \qbin{(p_1+v_1+1)\ell+p_2+v_2-s-1}{(p_1+v_1)\ell + p_2+v_2}{q} = 0$ again by theorem \ref{thm:qLucas}.

\noindent {\bfseries\itshape Case 3: $p_2+v_2\leq s$.} Then, $i < p+v$ gives $r+1 \leq p_1+v_1$ so that $m_{\ell-s+p+v-1} = 0$ as 
$$\ell-s+p+v-1 = (p_1+v_1+1)\ell-s+p_2+v_2-1\geq(r+2)\ell-s+p_2+v_2-1> (r+2)\ell-s-2 = j.$$
This concludes the proof as every possible case has been studied.
\end{proof}
\begin{corollary}\label{cor:Tq} The $\mathbb{C}$-vector space $T_q(i)$ with basis $\{m_0,...,m_j,n_0,...,n_i\}$ is a $\luszt$-module for the action given by\medskip
\begin{center}
$\displaystyle Km_k = q^{j-2k}m_k, \qquad Kn_p = q^{i-2p}n_p, \qquad  E^{(v)}m_k = \qbin{j-k+v}{v}{q}m_{k-v}, \qquad F^{(v)}m_k = \qbin{k+v}{v}{q}m_{k+v},$ \medskip\\  $\displaystyle \qquad E^{(v)}n_p = \qbin{i-p+v}{v}{q}n_{p-v}+\gamma_{p,v}m_{\ell-s+p-v-1}\quad $ and $ \displaystyle\quad F^{(v)}n_p = \qbin{p+v}{v}{q}n_{p+v}+\omega_{p,v}m_{\ell-s+p+v-1}$.
\end{center}\medskip
Also, the unrolled generator $H$ acts on $T_q(i)$ as $Hm_k = (j-2k)m_k$ and $Hn_p = (i-2p)n_p$. 
\end{corollary}
\begin{proof}
The first statement follows from proposition \ref{prop:Tt} and lemmas \ref{lemma:gammaLUq} and \ref{lemma:omegaLUq}. The last one is easily proved using the definition of the unrolled generator $H$ and proposition \ref{prop:Tt}.
\end{proof}
As $T_q(i)$ is clearly an extension of $\Delta_q(j)$ by $\Delta_q(i)$, lemma \ref{lemma:ExtWeyl} shows that it is enough to show the following result in order to finish the proof theorem \ref{thm:ProjReal}. Recall that the construction of $T_q(i)$ and theorem \ref{thm:ProjReal} both assume that $s<\ell-1$. 
\begin{lemma}\label{lemma:TqInd} The $\luszt$-module $T_q(i)$ is an indecomposable extension of $\Delta_q(j)$ by $\Delta_q(i)$.
\end{lemma}
\begin{proof} Suppose {\itshape ad absurdum} that the extension is trivial and fix a $\luszt$-linear isomorphism $\varphi:\Delta_q(i)\oplus\Delta_q(j)\rightarrow T_q(i)$. Let $\{\overline m_k\}_{k=0}^j\subseteq \Delta_q(j)$ and $\{\overline n_p\}_{p=0}^i\subseteq \Delta_q(i)$ be the bases given in definition \ref{def:Weyl} and observe that the restriction of $\varphi$ to $\Delta_q(j)$ must satisfy $\varphi(\overline m_k) = \varsigma m_k$ for some $\varsigma\in \mathbb{C}^{\times}$ as $\End_{\luszt}(\Delta_q(j)) \simeq \mathbb{C}$ by the proof of lemma \ref{lemma:ExtWeyl}.

Write $\varphi(\overline n_0) = \sum_{k = 0}^j \alpha_k^{(m)} m_k+\sum_{p=0}^i \alpha_p^{(n)} n_p$ where $\alpha_0^{(m)},\dots,\alpha_j^{(m)},\alpha_0^{(n)},\dots,\alpha_i^{(n)}\in \mathbb{C}$. Then,
\begin{equation*} 0 =  
H\varphi(\overline n_0)-\varphi(H\overline n_0) =
2\sum_{k=0}^j \alpha_k^{(m)}(\ell-s-1-k)m_k
- 2\sum_{p=0}^i p\alpha_p^{(n)}n_p
\end{equation*}
gives $\varphi(\overline n_0)=\alpha_0^{(n)}n_0+\alpha_{\ell-s-1}^{(m)}m_{\ell-s-1}$. However, this implies
\begin{center}
$0 = E \varphi(\overline n_0) = (\gamma_{0,1}\alpha_{0}^{(n)}+\alpha_{\ell-s-1}^{(m)}[j-\ell+s+2]_q)m_{\ell-s-2} = (\alpha_{0}^{(n)}+\alpha_{\ell-s-1}^{(m)}[(r+1)\ell]_q)m_{\ell-s-2} = \alpha_{0}^{(n)}m_{\ell-s-2}$,
\end{center}
and $\alpha_0^{(n)}=0$ as $s<\ell-1$. Hence $\varsigma\varphi(\overline n_0) = \varsigma\alpha_{\ell-s-1}^{(m)}m_{\ell-s-1} = \alpha_{\ell-s-1}^{(m)}\varphi(\overline m_{\ell-s-1})$ and we contradict the injectivity of $\varphi$.
\end{proof}
\end{subsection}
\begin{subsection}{Fusion rules for simple and projective modules} Before proving the fusion rules underlying theorem \ref{thm:fusLUq}, we must recall the \textit{Poincaré-Birkhoff-Witt} factorization of $\luszt$ and the corresponding concept of highest weight vector.
\begin{definition}[Highest weight vector] Fix a $\luszt$-module $M$. Then, a non-zero $v\in M$ is of highest weight $\mu\in \mathbb{C}$ if $Hv = \mu v$ and $E^{(n)}v = 0$ for all $n\in \mathbb{N}$. The module $M$ is itself of highest weight $\mu$ if $M = \luszt v$ for $v$ of highest weight $\mu$.
\end{definition}
For any $i\in \mathbb{Z}_{\geq 0}$, the Weyl module $\Delta_q(i)$ is a highest weight modules of weight $i$. In addition, proposition \ref{prop:Weyl} (and the remark after it) shows that $L_q(i)$ is the unique finite-dimensional irreducible (type I) $\luszt$-module of highest weight $i$.
\begin{lemma}\label{lemma:topphp} Let $M$ be a finite-dimensional $\luszt$-module (of type I and) of highest weight $i\in \mathbb{Z}_{\geq 0}$. Then, $\head M$ is isomorphic to a direct sum of the form $L_q(i)^{\oplus n} = \bigoplus_{m=1}^n L_q(i)$ for some $n\in \mathbb{N}$. In particular, $L_q(i)$ is contained in $\head M$. 
\end{lemma}
\begin{proof} Let $L$ be an irreducible direct factor of the semisimple module $\head M$ and let $v$ be a highest weight vector of $M$. Let also $\pi : M \rightarrow L$ be the natural projection. Then, the surjectivity of $\pi$ gives $L = \luszt \pi(v)$ and its $\luszt$-linearity forces $H\pi(v) = i\pi(v)$ with $E^{(n)}\pi(v) = 0$ for any $n \in \mathbb{N}$. The simple module $L$ is thus of highest weight $i$ and we must have a $\luszt$-linear isomorphism $L \simeq L_q(i)$ by the comment made just before the lemma. 
\end{proof}
\begin{proposition}[\cite{Lentner}]\label{prop:PBW} There is a $\mathbb{C}$-linear isomorphism (induced by multiplication)
\begin{equation*}\luszt \simeq \mathbb{C}[F^{(\ell)}]\otimes (\mathbb{C}[F]/\langle F^{\ell}\rangle) \otimes \mathbb{C}[K]/\langle K^{2\ell}-{\normalfont \id}\rangle \otimes \mathbb{C}[\qbin{K;0}{\ell}{q}]\otimes \mathbb{C}[E^{(\ell)}]\otimes (\mathbb{C}[E]/\langle E^{\ell}\rangle)
\end{equation*}
with $\qbin{K;0}{\ell}{q}$ the specialization at $t=q$ of the element defined in lemma \ref{lemma:K0ellt} and $\mathbb{C}[F]/\langle F^{\ell} \rangle$ the quotient of the algebra $\mathbb{C}[F]$ by the ideal generated by $F^{\ell}$ (and similarly for the other quotients). In particular, by lemma \ref{prop:RelationsUres}, monomials of the form $F^{(n_1)}K^{n_2}\qbin{K;0}{\ell}{q}^{n_3}E^{(n_4)}$, where $n_1,n_2,n_3,n_4\in \mathbb{Z}_{\geq 0}$, generate $\luszt$ as a $\mathbb{C}$-vector space.
\end{proposition}
We are now ready to prove theorem \ref{thm:fusLUq}. We start by considering tensor products of the form $L_q(i)\otimes L_q(1)$ for $i\in \mathbb{Z}_{\geq 0}$. We then show a similar result for products of the form $P_q(i)\otimes L_q(1)$ by using the structure of the modules $P_q(i)$ and by proving the exactness of the endofunctor $-\otimes L_q(1)$ of the category of finite-dimensional (type I) $\luszt$-modules. We will also often use the following well-known fact (see for example \cite{Bcorings} for a proof of a more general result\footnote{Section 15.9 of \cite{Bcorings} gives a Hom-Tensor adjunction for right-modules over a Hopf algebra. Restricting the corresponding adjunction for left-modules over $\luszt$ to the full subcategory $\mathscr{C}$ of finite-dimensional (type I) modules, we get a natural isomorphism $\Hom_{\mathscr{C}}(P\otimes M,-)\simeq \Hom_{\mathscr{C}}(P,\Hom_{\mathbb{C}}^S(M,-))$ where $\Hom_{\mathbb{C}}^S(M,N)$ is a $\luszt$-module defined over the vector space $\Hom_{\mathbb{C}}(M,N)$ (cf. \cite{Bcorings}). This is enough to conclude the projectivity of $P\otimes M$ in $\mathscr{C}$ as the functors $\Hom_{\mathbb{C}}^S(M,-)$ and $\Hom_{\mathbb{C}}(M,-)$ act in the same way on $\luszt$-linear morphisms (so $\Hom_{\mathbb{C}}^S(M,-)$ is in particular an exact functor).}). 
\begin{lemma}\label{lemma:ProdProj} Let $P,M$ be finite-dimensional $\luszt$-modules and assume $P$ projective in the category of finite-dimensional (type I) $\luszt$-modules. Then, $P\otimes M$ is also projective in this category.
\end{lemma}
\begin{proposition}\label{prop:fusSimple} Let $i=r\ell+s\in\mathbb Z_{\geq 0}$ with $r,s\in \mathbb{Z}_{\geq 0}$ and $s<\ell$. Then,
\begin{equation*}
L_q(i)\otimes L_q(1) \simeq (1-\delta_{s,0})M_q(i-1)\oplus (1-\delta_{s,\ell-1})L_q(i+1)\end{equation*}
where
\begin{equation*}
M_q(i-1) = \begin{cases}
	L_q(i-1) & \text{if } s \neq \ell-1, \\
	P_q(i-1) & \text{if } s = \ell-1.
\end{cases}
\end{equation*}
\end{proposition}
\begin{proof} Let $j=i+2(\ell-s-1)$ and $\mathscr{B}_j= \{m_{a\ell+b}\,|\,0\leq a\leq r,\,\ell-s-1\leq b < \ell\}\subseteq \Delta_q(j)$ be the basis of $L_q(i)$ given\footnote{The case $s = \ell-1$ was left out of the remark following proposition \ref{prop:Weyl}. In this case, $L_q(i)=\Delta_q(i)$ and $\mathscr{B}_j= \{m_k\}_{k=0}^i$ is the basis of definition \ref{def:Weyl}.} just after proposition \ref{prop:Weyl}. Let also $\{m_0',m_1'\}\subseteq L_q(1)=\Delta_q(1)$ be the basis of definition \ref{def:Weyl} and write $M = L_q(i)\otimes L_q(1)$. Then an easy computation using the coproduct of section \ref{sec:luqsl2} shows that the element $y = m_{\ell-s-1}\otimes m_0' \in M$ is a highest weight vector of weight $i+1$. It thus follows from lemma \ref{lemma:topphp} that the submodule $\langle y\rangle \subseteq M$ generated by $y$ contains a copy of $L_q(i+1)$ in its head. In particular, if $s = 0$, we get $\dim M = 2(r+1) = \dim L_q(i+1)$ and we must have $M\simeq L_q(i+1)$. The result is hence proved in this case and we may suppose hereafter that $s\geq 1$. 

In this situation, another direct computation shows that $x = m_{\ell-s-1}\otimes m_1'+q^{\ell-s}m_{\ell-s}\otimes m_0' \in M$ is a highest weight vector of weight $i-1$. Lemma \ref{lemma:topphp} then implies that $L_q(i-1)$ appears in the head of the submodule $\langle x\rangle\subseteq M$ generated by $x$ so that $\dim \langle x\rangle\geq \dim L_q(i-1)=(r+1)s$. We now want to prove the inequality $\dim \langle x \rangle \leq (r+1)s$ in order to deduce that $\langle x\rangle\simeq L_q(i-1)$ is a submodule of $M$. To do so, note that $K$ with the specialization at $q$ of the element $\qbin{K;0}{\ell}{t}\in U_{\text{res}}$ of lemma \ref{lemma:K0ellt} both act diagonally\footnote{This follows easily from the definition of $\qbin{K;0}{\ell}{t}$ and from the fact that $\Delta$ is an algebra morphism with $\Delta(K)x = (K\otimes K)x = q^{i-1}x$.} on $x$ through the coproduct $\Delta$. Moreover, $E^{(n)}$ annihilates  the highest weight vector $x$. Thus, proposition \ref{prop:PBW} shows that $\langle x\rangle$ is the linear span of the elements $\Delta(F^{(n)})x$ with $n\in \mathbb{Z}_{\geq 0}$. Also, by definition of  $\Delta$,
\begin{align*}
\Delta(F^{(n)})x &= (F^{(n)}\otimes K^n)(m_{\ell-s-1}\otimes m_1') + q^{\ell-s}(F^{(n)}\otimes K^n +q^{n-1}F^{(n-1)}\otimes K^{n-1}F)m_{\ell-s}\otimes m_0'\\
&= \qbin{\ell+n-s}{n}{q}(m_{\ell+n-s-1}\otimes m_1'+q^{\ell+n-s}m_{\ell+n-s}\otimes m_0')
\end{align*}
where the last equality follows from theorem \ref{thm:qPascal}. Write as usual $n$ as $n_1\ell+n_2$ with $n_1,n_2\in \mathbb{Z}_{\geq 0}$ and $n_2 < \ell$. Then the vector $\Delta(F^{(n)})x$ is zero if $n_1 \geq r+1$ (as this gives $\ell+n-s-1 > j$) or if $n_2 \geq s$ (since $\qbin{\ell+n-s}{n}{q}=\qbin{(n_1+1)\ell+n_2-s}{n_1\ell+n_2}{q} = 0$ in this case by theorem \ref{thm:qLucas}). Therefore, $\langle x\rangle$ is contained in the linear span of $\{\Delta(F^{(n_1\ell+n_2)})x\,|\, 0\leq n_1 \leq r \text{ and } 0\leq n_2< s\}$ and we indeed have $\dim \langle x\rangle \leq (r+1)s$. The module $M$ thus contains a submodule isomorphic to $L_q(i-1)$.

Suppose $s = \ell-1$. Then, proposition \ref{prop:ProjAT} implies $L_q(i) = P_q(i)$ and $M = L_q(i)\otimes L_q(1)$ must thus be projective in the category $\mathscr{C}$ of finite-dimensional (type I) $\luszt$-modules by lemma \ref{lemma:ProdProj}. However, the submodule $\langle x\rangle \simeq L_q(i-1)$ of $M$ just established is not projective in $\mathscr{C}$ and must therefore lie in the socle of a direct factor of $M$ that is projective in this category. Proposition \ref{prop:ProjAT} (iii) shows that this direct factor is $P_q(i-1)$ with dimension $\dim P_q(i-1)=2(r+1)\ell=\dim M$ so that $M\simeq P_q(i-1)$ as claimed for $s=\ell-1$\footnote{Note that the first paragraph of the proof says that $M$ also contains a submodule $\langle y \rangle$ whose head is isomorphic to $L_q(i+1)$. When $s=\ell-1$, this submodule is isomorphic to $\Delta_q(i+1)$ and is indeed contained in $P_q(i-1)$ by proposition \ref{prop:ProjAT} (iii) (as the $J$ associated to $I=i-1$ is $J=i+1$).}.

Suppose finally $1\leq s < \ell-1$ and recall the submodule $\langle y \rangle \subseteq M$ introduced at the beginning of the proof. As for the submodule $\langle x\rangle$, this submodule $\langle y \rangle$ is the linear span of the $\Delta(F^{(n)})y$ with $n\in \mathbb{Z}_{\geq 0}$:
\begin{equation*}
\Delta(F^{(n)})y = q^n\qbin{\ell+n-s-1}{n}{q}m_{\ell+n-s-1}\otimes m_0'+
\qbin{\ell+n-s-2}{n-1}{q}m_{\ell+n-s-2}\otimes m_1'.
\end{equation*}
If again $n=n_1\ell+ n_2$, then $\Delta(F^{(n)})y$ is zero if $n_2 \geq s+2$ (as this forces $\qbin{\ell+n-s-1}{n}{q}=\qbin{\ell+n-s-2}{n-1}{q}=0$ by theorem \ref{thm:qLucas}). It is also zero if $n_1 \geq r+1$. Indeed, in this situation, $\ell+n-s-2 \geq j$ and both $m_{\ell+n-s-1}$ and $m_{\ell+n-s-2}$ are zero unless $\ell+n-s-2=j$. This last equality can however happen only when $n_2=0$ and implies $\qbin{\ell+n-s-2}{n-1}q=\qbin{n_1\ell+\ell-s-1}{(n_1-1)\ell+\ell-1}q=0$ by theorem \ref{thm:qLucas}. The submodule $\langle y\rangle\subseteq M$ is thus the linear span of $\{\Delta(F^{(n_1\ell+n_2)})y\ |\ 0\leq n_1 \leq r \text{ and } 0\leq n_2\leq s+1\}$ and $\dim \langle{y}\rangle \leq (r+1)(s+2)$. The simple module $L_q(i+1)$ must then be isomorphic to $\langle y\rangle \subseteq M$ as $L_q(i+1)\subseteq \head \langle y\rangle$ with $\dim L_q(i+1)=(r+1)(s+2)\geq\dim \langle y\rangle$. We thus have constructed submodules $\langle x \rangle$ and $\langle y\rangle$ of $M$ satisfying $\langle x \rangle \simeq L_q(i-1)$ and $\langle y \rangle \simeq L_q(i+1)$. It is then easy to show that $M \simeq L_q(i-1)\oplus L_q(i+1)$ as claimed for $1\leq s<\ell-1$ since, in this case, $\dim M = 2(r+1)(s+1) = (r+1)s+(r+1)(s+2) = \dim (L_q(i-1)\oplus L_q(i+1))$.
\end{proof}
\begin{corollary}\label{cor:fusProj}
Let $i=r\ell+s\in\mathbb Z_{\geq 0}$ with $r,s\in \mathbb{Z}_{\geq 0}$ and $s<\ell-1$. Then
\begin{equation*}
P_q(i)\otimes L_q(1) \simeq (1+\delta_{s,\ell-2})P_q(i+1)\oplus (1-\delta_{r,0}\delta_{s,0})P_q(i-1)\oplus \delta_{s,0}P_q(i+2\ell-1).
\end{equation*}
\end{corollary}
\begin{proof} The endofunctor $-\otimes L_q(1)$ of the category $\mathscr{C}$ of finite-dimensional (type I) $\luszt$-modules is clearly exact since it acts on morphisms precisely like the usual endofunctor $-\otimes L_q(1)$ of the category of finite-dimensional $\mathbb{C}$-vector spaces. The action of this functor on $\soc P_q(i)\simeq L_q(i)$ together with proposition \ref{prop:fusSimple} (with $s<\ell-1$) thus shows the existence of an inclusion of $(1-\delta_{s,0})L_q(i-1)\oplus L_q(i+1) \simeq \soc P_q(i) \otimes L_q(1)$ into $M=P_q(i)\otimes L_q(1)$. However, by lemma \ref{lemma:ProdProj}, $M$ is projective in the category $\mathscr{C}$ so this inclusion (and proposition \ref{prop:ProjAT}) implies that $M$ has a direct factor isomorphic to $(1-\delta_{s,0})P_q(i-1)\oplus P_q(i+1)$. Note that proposition \ref{prop:ProjAT} also states that $L_q(i-1)$ (or $L_q(i+1)$) is projective only when $s=0$ (resp.~$s=\ell-2$). The inclusion $(1-\delta_{s,0})L_q(i-1)\oplus L_q(i+1)\subsetneq (1-\delta_{s,0})P_q(i-1)\oplus P_q(i+1)$ is thus strict if $0\neq s \neq \ell-2$. Also, $M \simeq P_q(i-1)\oplus P_q(i+1)$ in this case as $\dim M = 4\ell(r+1) = 2\ell(r+1)+2\ell(r+1)= \dim P_q(i-1) + \dim P_q(i+1).$ The statement is thus proved for all $s$ except $0$ and $\ell-2$. 

If $s = \ell-2\neq 0$, then the submodule $L_q(i+1) = P_q(i+1)$ appearing in $\soc P_q(i)\otimes L_q(1)$ is simple and projective in $\mathscr{C}$ by proposition \ref{prop:ProjAT}. The action of $-\otimes L_q(1)$ on $\head P_q(i) = L_q(i)$ must then give another copy of $L_q(i+1)$ in the head of $M$ (by proposition \ref{prop:fusSimple}) so that $M$ must have $P_q(i-1)\oplus 2L_q(i+1)$ as a direct factor. This gives $M \simeq P_q(i-1)\oplus 2P_q(i+1)$ as claimed since $\dim M = 4\ell(r+1) = 2\ell(r+1)+\ell(r+1)+\ell(r+1) = \dim (P_q(i-1)\oplus P_q(i+1)\oplus P_q(i+1))$.

If $s = 0$, then $j = i+2(\ell-1)=(r+1)\ell+\ell-2$ and $i_2 = i-2=(r-1)\ell+\ell-2$ with $L_q(i_2) = 0$ if $r = 0$. Moreover proposition \ref{prop:fusSimple} shows that the action of $-\otimes L_q(1)$ on the subquotients $L_q(j)$ and $L_q(i_2)$ of $P_q(i)$ leads to two subquotients of $M$ isomorphic to $L_q(j+1)$ and $L_q(i_2+1)$.
Both are projective in $\mathscr{C}$ by proposition \ref{prop:ProjAT} (that is $L_q(j+1)=P_q(j+1)$ and $L_q(i_2+1)=P_q(i_2+1)$) and the module $M$ must thus have a direct factor isomorphic to $(1-\delta_{r,0})P_q(i-1)\oplus P_q(i+2\ell-1)$. Of course, the projective $P_q(i+1)$ found in the first paragraph of this proof is also a direct factor of $M$ (with multiplicity at least $2$ if $s = \ell-2$). Two cases are now considered.

\noindent {\bfseries\itshape Case 1: $s=0\neq \ell-2$.} Then, as $P_q(i-1)=L_q(i-1)$ and $P_q(i+2\ell-1)=L_q(i+2\ell-1)$, we have
\begin{equation*}
\dim P_q(i+1)\oplus (1-\delta_{r,0})\dim P_q(i-1)\oplus \dim P_q(i+2\ell-1)= 2\ell(r+1)+(1-\delta_{r,0})\ell r+\ell(r+2)=4\ell(r+1)=\dim M
\end{equation*}
so that $M \simeq P_q(i+1)\oplus (1-\delta_{r,0})P_q(i-1)\oplus P_q(i+2\ell-1)$ as claimed.

\noindent {\bfseries\itshape Case 2: $s=\ell-2=0$.} As in the case $s=\ell-2$, the action of $-\otimes L_q(1)$ on both $\soc P_q(i)$ and $\head P_q(i)$ gives rise to a copy of $P_q(i+1)$. Then, $M \simeq 2P_q(i+1)\oplus (1-\delta_{r,0})P_q(i-1)\oplus P_q(i+2\ell-1)$ as the direct sum on the right is a direct factor of $M$ and since\footnote{By the above results, all the modules appearing in the decomposition $2P_q(i+1)\oplus (1-\delta_{r,0})P_q(i-1)\oplus P_q(i+2\ell-1)$ are simple when $s=\ell-2=0$.} $\dim M = \ell(r+1)+\ell(r+1)+(1-\delta_{r,0})\ell r+\ell(r+2)= \dim(2P_q(i+1)\oplus (1-\delta_{r,0})P_q(i-1)\oplus P_q(i+2\ell-1))$.

This concludes the proof.
\end{proof}
\end{subsection}
\end{section}

%%%%%%%%%%%%%
%%%%%%%%%%%%%
%%%%%   %%%%%
%%% app c %%%
%%%%%   %%%%%
%%%%%%%%%%%%%
%%%%%%%%%%%%%

%%%%%%%%%%%%%%%%%%%%%%%%%%%%%%%%
%
\begin{section}{Remaining cases and problematic pairs}\label{app:problematic}
%
%%%%%%%%%%%%%%%%%%%%%%%%%%%%%%%%
In this appendix, we finish the analysis done in section \ref{sec:xxzLUq} by letting $q = \pm 1$ and by replacing Lusztig's quantum group $\luszt$ by the universal envelopping algebra $U\mathfrak{sl}_2$. We also show some results involving the problematic pairs $(0,\pm i) \in \lambda_N$ (with $N$ even, see section \ref{sec:main}) which were left unproven in sections \ref{sec:xxzatln} and \ref{sec:struc}.
%
% q = + 1 ou -1
%
\begin{subsection}{XXZ chain morphisms for $\ell=1$}\label{sec:qpm1}
Throughout this subsection, the parameter $q$ is set to $+1$ or $-1$ and the algebra $U_q\mathfrak{sl}_2$ is understood to be the universal enveloping algebra $U\mathfrak{sl}_2$. This algebra is generated by elements $e,f,h$ satisfying
$$[h,e]=2e,\qquad [h,f]=-2f\quad\text{and}\quad[e,f]=h.$$ 
For $n\in\mathbb{Z}_{\geq 0}$, we denote by $L(n)$ the unique simple $U\mathfrak{sl}_2$-module with $\dim L(n) = n+1$. This representation has a basis $\{v_0,\dots ,v_n\}$ with explicit action defined by $ev_i = i(n-i+1)v_{i-1}$, $fv_i = v_{i+1}$ and $hv_i = (n-2i)v_i$ (where $v_{-1} = v_{n+1} = 0$). Let $\Delta:U\mathfrak{sl}_2\rightarrow U\mathfrak{sl}_2\otimes U\mathfrak{sl}_2$ be the algebra morphism given by
\begin{equation*}
\Delta(e) = e\otimes 1+q(1\otimes e), \qquad \Delta(f) = f\otimes 1+q(1\otimes f) \quad \text{and}\quad \Delta(h) = h\otimes 1+1\otimes h.
\end{equation*}
This morphism for $q=1$ is the usual coproduct used to define the tensor product of representations of $\mathfrak{sl}_2$. A quick check shows that it is also an algebra morphism at $q=-1$. Define recursively a family of maps\footnote{Remark that $\Delta_3 \neq (\Delta\otimes \id)\circ \Delta$ when $q = -1$ as $\Delta$ is then not coassociative. Therefore the pullback of a threefold tensor product of $U\mathfrak{sl}_2$-modules by $(\Delta\otimes \id)\circ \Delta$ and its pullback by $\Delta_3$ can be non-isomorphic  and care is to be taken when defining a $U\mathfrak{sl}_2$-action on the XXZ chains with $q = -1$.} $\{\Delta_n:U\mathfrak{sl}_2\rightarrow U\mathfrak{sl}_2^{\otimes n}\}_{n\geq 2}$ with $\Delta_2= \Delta$ and $\Delta_{n+1} = (\id\otimes \Delta)\circ\Delta_n$ where $\id$ denotes the identity operator on $U\mathfrak{sl}_2^{\otimes (n-1)}$. Identify moreover $L(1)$ with $\mathbb{C}^2$ through the change of basis given by $v_0\mapsto |+\rangle$ with $v_1\mapsto |-\rangle$ and associate the spin chain $\XXZ{N;z}{+}$ ($=\XXZ{N;z}{-}$) to the pullback of $(L_q(1))^{\otimes N}$ by $\Delta_N$. Then, an easy computation gives the following explicit formulas for the $U\mathfrak{sl}_2$-action on $\XXZ{N;z}{+}$
\begin{equation*} 
e|x_1\dots x_N\rangle_z^+ = \sum_{\substack{1\leq j\leq N\\ x_j = -}} q^{j-1}|x_1\dots x_{j-1}(+)x_{j+1}\dots x_N\rangle_z^+, \qquad 
f|x_1\dots x_N\rangle_z^+ = \sum_{\substack{1\leq j\leq N\\ x_j=+}}q^{j-1}|x_1\dots x_{j-1}(-)x_{j+1}\dots x_N\rangle_z^+
\end{equation*}
and $h|x_1\dots x_N\rangle_z^+ = 2S^z |x_1\dots x_N\rangle_z^+ = d|x_1\dots x_N\rangle_z^+$ where $d = \sum_{j=1}^N x_j$. 
\begin{theorem}\label{thm:fl1} Let $d \in \mathbb{Z}$ be of the parity of $N$ and $z \in \mathbb{C}^{\times}$ such that $q^d=z^2$. Set $d_m = d+2m$ with $z_m = zq^m$ for $m\in\mathbb{Z}$. Then, $\ii_{d_m,z_m}^{(1)}:\XXZ{N;d_m,z_m}{+}\rightarrow \XXZ{N;d_{m-1},z_{m-1}}{+}$ given by $|x_1\dots x_N\rangle_{z_m}^+\mapsto f|x_1\dots x_N\rangle_{z_{m-1}}^+$ is $\atl{N}$-linear.
\end{theorem}
\begin{proof} Fix a basis element $|x_1\dots x_N\rangle_{z_m}^+$ of $\XXZ{N;d_m,z_m}{+}$. Then, a straightforward computation using \eqref{eq:e1} gives
\begin{equation*}
e_1\ii_{d_m,z_m}^{(1)}|x_1\dots x_N\rangle_{z_m}^+ = \varsigma+\varsigma_1\quad\text{and}\quad \ii_{d_m,z_m}^{(1)}e_1|x_1\dots x_N\rangle_{z_m}^+=\varsigma+\varsigma_2
\end{equation*}
where, as $q^2=1$,
\begin{align*}
\varsigma_1 &=
\delta_{x_1,+}\delta_{x_2,+}((1-q^{2})|(+)(-)x_3\dots x_N\rangle_{z_{m-1}}^++(q-q^{-1})|(-)(+)x_3\dots x_N\rangle_{z_{m-1}}^+) = 0, \\
\varsigma_2 &= \delta_{x_1+x_2,0}((\delta_{x_2,+}+q\delta_{x_1,+})|(-)(-)x_3\dots x_N\rangle_{z_{m-1}}^+ - q^{x_1}\delta_{x_1,+}|(-)x_2\dots x_N\rangle_{z_{m-1}}^+ - q^{x_1+1}\delta_{x_2,+}|x_1(-)x_3\dots x_N\rangle_{z_{m-1}}^+)=0 \text{ and}\\
\varsigma &= \sum_{\substack{3\leq j\leq N \\ x_j = +}}q^{j-1} 
\delta_{x_1+x_2,0}(|x_2x_1x_3\dots x_{j-1}(-)x_{j+1}\dots x_N\rangle_{z_{m-1}}^+ - q^{x_1}|x_1\dots x_{j-1}(-)x_{j+1}\dots x_N\rangle_{z_{m-1}}^+).
\end{align*}
Thus, $e_1\ii_{d_m,z_m}^{(1)}|x_1\dots x_N\rangle_{z_m}^+ =\ii_{d_m,z_m}^{(1)}e_1|x_1\dots x_N\rangle_{z_m}^+$ as desired. We also have
\begin{align*} 
\Omega_N \ii^{(1)}_{d_m,z_m}|x_1\dots x_N\rangle_{z_m}^+ &= z_{m-1}\delta_{x_1,+}|x_2\dots x_N(-)\rangle_{z_{m-1}}^++z_{m-1}^{-x_1}\sum_{\substack{2\leq j \leq N\\x_j =+}}q^{j-1}|x_2\dots x_{j-1}(-)x_{j+1}\dots x_Nx_1\rangle_{z_{m-1}}^+,\\
\ii^{(1)}_{d_m,z_m}\Omega_N |x_1\dots x_N\rangle_{z_m}^+&= z_m^{-1}q^{N-1}\delta_{x_1,+}|x_2\dots x_N(-)\rangle_{z_{m-1}}^++z_m^{-x_1}\sum_{\substack{1\leq j\leq N-1\\x_{j+1}=+}}q^{j-1}|x_2\dots x_j(-)x_{j+2}\dots x_Nx_1\rangle_{z_{m-1}}^+,
\end{align*}
so $\Omega_N \ii^{(1)}_{d_m,z_m}|x_1\dots x_N\rangle_{z_m}^+ = \ii^{(1)}_{d_m,z_m}\Omega_N|x_1\dots x_N\rangle_{z_m}^+$ as $z_{m-1}^{-x_1} = q^{-1}z_m^{-x_1}$ and $z_{m-1} = q^{d-1}z_m^{-1} = q^{N-1}z_m^{-1}$ by hypothesis.
\end{proof}
Suppose that $q^{d} = z^{2}$. Then, $q^{-d}=(z^{-1})^2$ and theorem \ref{thm:fl1} produces a morphism of $\atl{N}$-modules from $\XXZ{N;-d_{m-1},z_{m-1}^{-1}}{+}$ to $\XXZ{N;-d_m,z_m^{-1}}{+}$. Denote by $\jj^{(1)}_{d_m,z_m}:\XXZ{N;d_{m-1},z_{m-1}}{+}\rightarrow \XXZ{N;d_m,z_m}{+}$ the $\atl{N}$-linear morphism obtained by conjugating this last map with the spin flip of section \ref{sec:atlOnXXZ}. (Recall that $\XXZ{N;d,z}+$ and $\XXZ{N;d,z}-$ are equal as $q^2 =1$.) Note also that the assumption $q^d = z^2$ amounts here to saying that $(d_{m-1},z_{m-1})\unlhd (d_m,z_m)$ whenever $m\in \mathbb{Z}$ where $\unlhd$ is the partial order on $\Delta=\mathbb{Z}\times \mathbb{C}^{\times}$ given through \eqref{eq:conditionsAB2}. Fix $m\in \mathbb{N}$ and let $\kk^{(m)}:\XXZ{N;d_m,z_m}{+}\rightarrow\XXZ{N;d,z}{+}$ and $\mm^{(m)}:\XXZ{N;d,z}{+}\rightarrow\XXZ{N;d_m,z_m}{+}$ be the $\atl{N}$-morphisms defined by
\begin{equation*}
\kk^{(m)} = \ii^{(1)}_{d_1,z_1}\ii^{(1)}_{d_2,z_2}\dots \ii^{(1)}_{d_m,z_m}\qquad\text{and}\qquad \mm^{(m)}=\jj^{(1)}_{d_m,z_m}\jj^{(1)}_{d_{m-1},z_{m-1}}\dots \jj^{(1)}_{d_1,z_1}
\end{equation*}
so $\kk^{(m)}|x_1\dots x_N\rangle_{z_m}^+ = f^m|x_1\dots x_N\rangle_z^+$ and $\mm^{(m)}|x_1\dots x_N\rangle_z^+ = e^m|x_1\dots x_N\rangle_{z_m}^+$ for $x_1,\dots ,x_N \in \{+,-\}$ since the $U\mathfrak{sl}_2$-action on $\XXZ{N;z}{+}$ is independent\footnote{We also use the fact that $\mm_{d_m,z_m}^{(1)}|x_1\dots x_N\rangle_{z_{m-1}}^+ = 
sf|{-}x_1\dots {-}x_N\rangle_{z_m}^+ = e|x_1\dots x_N\rangle_{z_m}^+$ with $s$ the spin flip of section \ref{sec:atlOnXXZ}.} of $z$. The morphisms $\kk^{(m)}$ and $\mm^{(m)}$ satisfy properties similar to those described in theorem \ref{thm:Aut}. To prove this fact, observe that, since the category of finite-dimensional $U\mathfrak{sl}_2$-modules is semisimple, we can decompose the $U\mathfrak{sl}_2$-module $\XXZ{N;z}{+}$ as a sum of irreducibles $L(n)$ with $n\in \mathbb{Z}_{\geq 0}$. Within $L(n)$, the action of $e^jf^j$ on the basis vector $v_i$ gives a non-zero multiple of $v_i$ if $i+j\leq n$. This well-known fact and its analogue for $f^je^j$ lead to the following lemma. 
\begin{lemma} Let $n,m\in \mathbb{Z}_{\geq 0}$ and $d \in \mathbb{Z}$ with $|d|\leq n$ and $d \equiv n$ modulo 2. Then, the action of $e^mf^m$ (or $f^me^m$) on $L(n)$ induces a $\mathbb{C}$-linear automorphism of the eigenspace $\Eval{L(n)}{h=d}{}$ if $|d-2m|\leq |d|$ (or $|d+2m|\leq |d|$, respectively).
\end{lemma} 
\noindent The analogue of theorem \ref{thm:Aut} follows directly from this lemma as the action of $h\in U\mathfrak{sl}_2$ on the chain $\XXZ{N;z}+$ coincide with the action of the diagonal spin operator $2S^z$ defining the eigenspaces $\XXZ{N;d,z}+$.
\begin{theorem}\label{thm:Auti} With the above notation, the maps $\mm^{(m)}\kk^{(m)}$ and $\kk^{(m)}\mm^{(m)}$ are respectively one-to-one if $|d|\leq |d+2m|$ and $|d|\geq |d+2m|$. In particular, $\kk^{(m)}$ is injective (or surjective) if $|d|\leq |d+2m|$ (or if $|d|\geq |d+2m|$, respectively) and $\mm^{(m)}$ is injective (or surjective) if $|d|\geq |d+2m|$ (or if $|d|\leq |d+2m|$, respectively).
\end{theorem}
\end{subsection}
%
% les paires problematiques
%
\begin{subsection}{Problematic pairs}\label{sub:problematic}
This section studies the problematic pairs, that is the pairs $(0,q)$ and $(0,q^{-1})$ in the case where $q+q^{-1}=0$ and $N$ is even. Throughout the section $q$ and $N$ will be assumed to satisfy these conditions. Of course, this means that $q$ is either $i$ or $-i$. Note that $(0,q)\sim (0,q^{-1})$ in $\Lambda$ and remark that theorem 2.4 of \cite{GL-Aff} implies that the head of the module $\Cell{N;0,q}$ does not provide any new simple module that is not already labeled in $\Lambda_N$ (see \eqref{eq:lambda}). The main purpose of this subsection is thus to determine the content and structure of these cellular modules. Here it is.
\begin{proposition}\label{prop:StrucCellProb} Let $N\in2\mathbb N$ and $q$ be either $i$ or $-i$. The Loewy diagram of the cellular module $\Cell{N;0,q}$ is
\begin{equation*}
\begin{tikzpicture}[baseline={(current bounding box.center)},scale=1/3]
\node (k0) at (0,9) [] {$(2,1)$};
\node (k3) at (8,9) [] {$(2,-1)$};
\node (j0) at (0,6) [] {$(4,-q)$};
\node (k1) at (0,3) [] {$(6,-1)$};
\node (j1) at (0,0) [] {\ \ \ \ \ $\vdots$\ \ \ \ \ };
\node (i0) at (8,6) [] {$(4,q)$};
\node (h0) at (8,3) [] {$(6,1)$};
\node (i1) at (8,0) [] {\ \ \ \ \ $\vdots$\ \ \ \ \ };
\node (k2) at (0,-3) [] {$(N,q^{1-\frac N2})$};
\node (h1) at (8,-3) [] {$(N,-q^{1-\frac N2})$\ .};
\draw[->] (k0) -- (j0);\draw[->] (k0) -- (i0);
\draw[->] (k3) -- (i0);\draw[->] (k3) -- (j0);
\draw[->] (j0) -- (k1);\draw[->] (j0) -- (h0);
\draw[->] (i0) -- (k1);\draw[->] (i0) -- (h0);
\draw[->] (k1) -- (j1);\draw[->] (k1) -- (i1);
\draw[->] (h0) -- (j1);\draw[->] (h0) -- (i1);
\draw[->] (j1) -- (k2);\draw[->] (i1) -- (k2);
\draw[->] (j1) -- (h1);\draw[->] (i1) -- (h1);
\end{tikzpicture}
\end{equation*}
\end{proposition}
\begin{proof}
The first step here is to compute the direct successors in $\lambda_N$ of $(0,q)$. These are precisely the nodes of the proposed diagram, with arrows pointing toward the successors. These nodes are labelled by an integer $i\in \mathbb{N}$ which is bounded above by $n=N/2$ and are given explicitly by $(2i,y_i)$ for the left column and by $(2i,-y_i)$ for the right where $y_i=q^{1-i}$.  The direct successors of $(0,q)$ through condition A and B are respectively $(2,1)$ and $(2,-1)$. Proposition \ref{prop:grandThm} applies as the pairs $(0,q)$, $(2,1)$ and $(2,-1)$ all belong to $\lambda_N$. There are therefore two injective $\atl{N}$-morphisms $\Cell{N;2,1}\to\Cell{N;0,q}$ and $\Cell{N;2,-1}\to\Cell{N;0,q}$ and the Loewy diagram of $\Cell{N;0,q}$ must contain the diagrams of $\Cell{N;2,1}$ and $\Cell{N;2,-1}$. Fortunately, the latter diagrams may be deduced from theorem \ref{thm:GL}: they are precisely those obtained by removing, from the diagram drawn in the statement, the node $(2,-1)$ for $\Cell{N;2,1}$ and $(2,1)$ for the other. It is thus sufficient to show that the dimension of $\Cell{N;0,q}$ is equal to that of the sum of the composition factors appearing in the statement.

As said above, the modules $\Cell{N;2,1}$ and $\Cell{N;2,-1}$ share all their composition factors except for their respective simple heads $\Irre{N;2,1}$ and $\Irre{N;2,-1}$. Since $\dim\Cell{N;2,1}=\dim\Cell{N;2,-1}$ (the dimension \eqref{eq:dimW} of $\Cell{N;d,z}$ is independent of $z$), the two simples $\Irre{N;2,1}$ and $\Irre{N;2,-1}$ have the same dimension. This argument can be repeated for each pair of irreducibles $\Irre{N;2i,y_i}$ and $\Irre{N;2i,-y_i}$. We thus want to prove that $\dim\Cell{N;0,q}=\sum_{1\leq i\leq n}2d_i$ where $d_i=\dim \Irre{N;2i,y_i}$. The dimensions $d_i$ satisfy
\begin{align*}
d_i&=\dim \Irre{N;2i,y_i}=\dim\Cell{N;2i,y_i}-2\sum_{i+1\leq j\leq n}d_j\\
&=\dim \Cell{N;2i,y_i}-\dim \Cell{N;2(i+1),y_{i+1}}-d_{i+1}= \begin{pmatrix}2n\\ n+i\end{pmatrix}-\begin{pmatrix}2n\\ n+i+1\end{pmatrix}-d_{i+1}.
\end{align*}
The third equality was obtained by gathering up all $d_j$'s in the sum but one of the the two $d_{i+1}$'s and the last one follows from \eqref{eq:dimW}. This recurrence ends at $d_n=\dim\Irre{N;N,y_n}=\dim\Cell{N;N,y_n}=1$. Its solution is
$$d_i=\begin{pmatrix}2n\\ n+i\end{pmatrix}+2\sum_{1\leq j\leq n-i}(-1)^j\begin{pmatrix}2n\\ n+i+j\end{pmatrix}.$$
Hence, 
\begin{align*}
\sum_{1\leq i\leq n} 2d_i &=\dim\Cell{N;2,y_1}+d_1
   =\begin{pmatrix} 2n\\ n-1\end{pmatrix}+\begin{pmatrix} 2n\\ n+1\end{pmatrix}+2\sum_{1\leq j\leq n-1}(-1)^j\begin{pmatrix}2n\\ n+1+j\end{pmatrix}\\
   &=-2\sum_{1\leq i\leq n}(-1)^i\begin{pmatrix}2n\\ n+i\end{pmatrix}
    =2(-1)^{n+1}\sum_{0\leq i\leq n-1}(-1)^i\begin{pmatrix}2n\\ i\end{pmatrix}.\\
\intertext{Note that the alterning sum contains the $n$ first terms of the binomial expansion of $(1-1)^{2n}$. Thus,}
   &=\begin{pmatrix} 2n\\ n\end{pmatrix}+(-1)^{n+1}\sum_{0\leq i\leq 2n}(-1)^i\begin{pmatrix}2n\\ i\end{pmatrix}=\begin{pmatrix} 2n\\ n\end{pmatrix}+(-1)^{n+1}(1-1)^{2n}\\
   &=\dim\Cell{N;0,q}
\end{align*}
as desired. This ends the demonstration.
\end{proof}
We can now adapt the proofs of theorem \ref{thm:main} and of corollary \ref{thm:gpEstUnQuotient} to the case of problematic pairs. 
\begin{proof}[Proof (corollary \ref{thm:gpEstUnQuotient}, problematic case).]
The proof given for proposition \ref{prop:StrucCellProb} shows that $\pg{N;0,q} \simeq \Irre{N;2,-1} = \head \Cell{N;2,-1}$ which is a quotient of $\im \ay{N;0,q}$ by proposition \ref{thm:Bsurvit}. To repeat this argument for the other problematic pair $(0,q^{-1})$, note that $\Cell{N;0,q}\simeq \Cell{N;0,q^{-1}}$ by definition of the cellular modules as the cokernel of some $f_z$ (see section \ref{sec:atln}). Note also that $(0,q^{-1})\preceq (2,-1)$ and $(0,q^{-1})\preceq (2,1)$ directly through condition A and B (respectively) by lemma \ref{thm:partialOrders}. Proposition \ref{prop:StrucCellProb} thus gives\footnote{We stress that the generic part of $\Cell{N;0,q}$ and of $\Cell{N;0,q^{-1}}$ are not isomorphic even though the modules $\Cell{N;0,q}$ and $\Cell{N;0,q^{-1}}$ are.} $\pg{N;0,q^{-1}} \simeq \Irre{N;2,1} = \head \Cell{N;2,1}$ which is a quotient of $\im\ay{N;0,q^{-1}}$ by proposition \ref{thm:Bsurvit}.
\end{proof}
In the following proof which follows almost exactly the steps of subsection \ref{sub:iii}, we will often use implicitly the Loewy diagram of proposition \ref{prop:StrucCellProb} in order to identify the successors of the problematic pair $(0,q)$ for the order $\preceq$. 
\begin{proof}[Proof (theorem \ref{thm:main}, subcase (iii), problematic case).] Fix $N\in 2\mathbb{N}$ and $q\in \mathbb{C}$ with $q^2=-1$. Then, the reasoning of subsection \ref{sub:iii} shows that $\XXZ{N;0,q}{+}$ admits a submodule $\mathsf{M}$ such that $\mathsf{M}\simeq \XXZ{N;2,-1}{+}/\im \ii^+_{(2,-1);(4,-q)}$ (namely $\mathsf{M}=\im \ii^+_{(0,q);(2,-1)}$). However, $(D,Z)=(2,-1)$ is not a problematic pair and is such that $q^D \not\in \{Z^2,Z^{-2}\}$. The proof of subsection \ref{sub:iii} thus shows that $\XXZ{N;2,-1}{+}$ has the Loewy diagram
\begin{equation*}
\begin{tikzpicture}[baseline={(current bounding box.center)},scale=0.45]
\node (d0) at (0,2) [] {$(2,-1)$}; %
\node (s0) at (3,4) [] {$(4,q)$}; %
\node (t0) at (3,0) [] {$(4,-q)$}; %
\node (h0) at (5,2) [] {$(6,-1)$}; %
\node (d1) at (8,2) [] {$(6,1)$}; %
\node (s1) at (10,4) [] {$(8,-q)$}; %
\node (t1) at (10,0) [] {$(8,q)$}; %
\node (h1) at (12,2) [] {$(10,1)$}; %
\node (d2) at (15,2) [] {$(10,-1)$}; %
\node (s2) at (17,4) [] {$\dots$};
\node (t2) at (17,0) [] {$\dots$};
\draw[->] (2,3.6) -- (0.8,2.7);% (s0) -- (d0);
\draw[->] (0.8,1.3) -- (2,0.6);% (d0) -- (t0);
\draw[->] (3.8,3.6) -- (5,2.6);% (s0) -- (h0);
\draw[->] (4,3.8) -- (7.7,2.6);% (s0) -- (d1);
\draw[red,thick,->] (5,1.5) -- (3.5,0.6); %(h0) -- (t0);
\draw[->] (7.8,1.5) -- (4.1,0.4); %(d1) -- (t0);
\draw[->] (8.7,3.8) -- (5.2,2.6); % (s1) -- (h0);
\draw[->] (8.7,3.6) -- (7.9,2.6); % (s1) -- (d1);
\draw[red,thick,->] (5.2,1.5) -- (9.15,0.4); %(h0) -- (t1);
\draw[->] (8,1.5) -- (9.8,0.6); %(d1) -- (t1);
\draw[->] (11.2,3.6) -- (12,2.6); %(s1) -- (h1);
\draw[->] (11.2,3.8) -- (14.8,2.6); %(s1) -- (d2);
\draw[red,thick,->] (12.1,1.5) -- (10,0.6); %(h1) -- (t1);
\draw[->] (15,1.45) -- (10.8,0.4); %(d2) -- (t1);
\draw[->] (16.4,3.9) -- (12.2,2.6); %(s2) -- (h1);
\draw[->] (16.6,3.8) -- (14.9,2.6); %(s2) -- (d2);
\draw[red,thick,->] (12.2,1.5) -- (16.3,0.3); %(h1) -- (t2);
\draw[->] (14.9,1.5) -- (16.5,0.5); %(d2) -- (t2);
\end{tikzpicture}
\end{equation*}
where the image $\im \ii^+_{(2,-1);(4,-q)}$ has been identified with red arrows. The submodule $\mathsf{M}\subseteq \XXZ{N;0,q}{+}$ hence has the structure: 
\begin{equation*}
\begin{tikzpicture}[baseline={(current bounding box.center)},scale=0.45]
\node (d0) at (0,2) [] {$(2,-1)$}; %
\node (s0) at (3,4) [] {$(4,q)$}; %
\node (d1) at (6,2) [] {$(6,1)$}; %
\node (s1) at (9,4) [] {$(8,-q)$}; %
\node (d2) at (12,2) [] {$(10,-1)$}; %
\node (s2) at (15,4) [] {$\dots$};
\draw[->] (2,3.6) -- (0.8,2.7);% (s0) -- (d0);
\draw[->] (4,3.8) -- (5.2,2.6);% (s0) -- (d1);
\draw[->] (7.7,3.6) -- (6.9,2.6); % (s1) -- (d1);
\draw[->] (10.2,3.6) -- (11,2.6); %(s1) -- (d2);
\draw[->] (14.6,3.7) -- (13.3,2.6); %(s2) -- (d2);
\end{tikzpicture}
\end{equation*}
Using the above analysis for the problematic pair $(0,q^{-1})$ and applying the $\star$-duality (with the help of proposition \ref{thm:isoByDual}), we can moreover deduce that $\XXZ{N;0,q}{+}$ has a quotient with the Loewy diagram:
\begin{equation}\label{eq:QuotProb}
\begin{tikzpicture}[baseline={(current bounding box.center)},scale=0.45]
\node (d0) at (0,4) [] {$(2,1)$}; %
\node (s0) at (3,2) [] {$(4,q)$}; %
\node (d1) at (6,4) [] {$(6,-1)$}; %
\node (s1) at (9,2) [] {$(8,-q)$}; %
\node (d2) at (12,4) [] {$(10,1)$}; %
\node (s2) at (15,2) [] {$\dots$};
\draw[->] (0.9,3.6) -- (2.1,2.5);% (s0) -- (d0);
\draw[->] (4.7,3.6) -- (3.75,2.5);% (s0) -- (d1);
\draw[->] (7.2,3.6) -- (7.9,2.5); % (s1) -- (d1);
\draw[->] (10.9,3.6) -- (10,2.5); %(s1) -- (d2);
\draw[->] (13.2,3.7) -- (14.5,2.5); %(s2) -- (d2);
\end{tikzpicture}
\end{equation}
Also, the proof of subsection \ref{sub:iii} shows that $\XXZ{N;0,q}{+}$ has another submodule, namely $\im s \ii^-_{(0,-q);(2,-1)}$, which is isomorphic to $(\XXZ{N;2,-1}{+}/\im \ii^+_{(2,-1);(4,-q)})^{\circ} \simeq \mathsf{M}^{\circ}$. By the above analysis and proposition \ref{prop:foncteurcirc}, the structure of this submodule is
\begin{equation*}
\begin{tikzpicture}[baseline={(current bounding box.center)},scale=0.45]
\node (d0) at (0,2) [] {$(2,-1)$}; %
\node (s0) at (3,4) [] {$(4,-q)$}; %
\node (d1) at (6,2) [] {$(6,1)$}; %
\node (s1) at (9,4) [] {$(8,q)$}; %
\node (d2) at (12,2) [] {$(10,-1)$}; %
\node (s2) at (15,4) [] {$\dots$};
\draw[->] (1.8,3.6) -- (0.8,2.7);% (s0) -- (d0);
\draw[->] (4.3,3.8) -- (5.2,2.6);% (s0) -- (d1);
\draw[->] (8.1,3.6) -- (6.8,2.6); % (s1) -- (d1);
\draw[->] (9.9,3.6) -- (10.9,2.6); %(s1) -- (d2);
\draw[->] (14.6,3.7) -- (13.3,2.6); %(s2) -- (d2);
\end{tikzpicture}
\end{equation*}
Repeating this argument for the pair $(0,q^{-1})$ and applying once again the $\star$-duality (with the help of proposition \ref{thm:isoByDual}), we deduce that $\XXZ{N;0,q}{+}$ has in addition a quotient with the following Loewy diagram:
\begin{equation*}
\begin{tikzpicture}[baseline={(current bounding box.center)},scale=0.45]
\node (d0) at (0,4) [] {$(2,-1)$}; %
\node (s0) at (3,2) [] {$(4,-q)$}; %
\node (d1) at (6,4) [] {$(6,1)$}; %
\node (s1) at (9,2) [] {$(8,q)$}; %
\node (d2) at (12,4) [] {$(10,-1)$}; %
\node (s2) at (15,2) [] {$\dots$};
\draw[->] (1.2,3.6) -- (2,2.6);% (s0) -- (d0);
\draw[->] (5,3.8) -- (3.8,2.6);% (s0) -- (d1);
\draw[->] (7,3.8) -- (8.1,2.6); % (s1) -- (d1);
\draw[->] (10.7,3.6) -- (9.7,2.6); %(s1) -- (d2);
\draw[->] (13.4,3.7) -- (14.7,2.5); %(s2) -- (d2);
\end{tikzpicture}
\end{equation*}
We can now finish the proof as in subsection \ref{sub:iii}. Indeed, we have shown (see proposition \ref{prop:StrucCellProb}) that all composition factors of $\Cell{N;0,q}$ appear in some subquotient of $\XXZ{N;0,q}{+}$. Since $\dim \XXZ{N;0,q}{+} = \dim \Cell{N;0,q}$, we have found all composition factors of the eigenspace $\XXZ{N;0,q}{+}$ and its Loewy diagram must be the one announced in theorem \ref{thm:main}, that is
\begin{equation}\label{fig:StrucXProb}
\begin{tikzpicture}[baseline={(current bounding box.center)},scale=0.45]
\node (s0) at (3,4) [] {$(2,1)$}; %
\node (t0) at (3,0) [] {$(2,-1)$}; %
\node (h0) at (5,2) [] {$(4,q)$}; %
\node (d1) at (8,2) [] {$(4,-q)$}; %
\node (s1) at (10,4) [] {$(6,-1)$}; %
\node (t1) at (10,0) [] {$(6,1)$}; %
\node (h1) at (12,2) [] {$(8,-q)$}; %
\node (d2) at (15,2) [] {$(8,q)$}; %
\node (s2) at (17,4) [] {$\dots$};
\node (t2) at (17,0) [] {$\dots$};
\draw[->] (3.8,3.6) -- (5,2.6);% (s0) -- (h0);
\draw[->] (4,3.8) -- (7.7,2.6);% (s0) -- (d1);
\draw[red,thick,->] (5,1.5) -- (3.5,0.6); %(h0) -- (t0);
\draw[->] (7.8,1.5) -- (4.1,0.4); %(d1) -- (t0);
\draw[->] (8.7,3.8) -- (5.2,2.6); % (s1) -- (h0);
\draw[->] (8.7,3.6) -- (7.9,2.6); % (s1) -- (d1);
\draw[red,thick,->] (5.2,1.5) -- (9.15,0.4); %(h0) -- (t1);
\draw[->] (8,1.5) -- (9.8,0.6); %(d1) -- (t1);
\draw[->] (11.2,3.6) -- (12,2.6); %(s1) -- (h1);
\draw[->] (11.2,3.8) -- (14.8,2.6); %(s1) -- (d2);
\draw[red,thick,->] (12.1,1.5) -- (10,0.6); %(h1) -- (t1);
\draw[->] (15,1.45) -- (10.8,0.4); %(d2) -- (t1);
\draw[->] (16.4,3.9) -- (12.2,2.6); %(s2) -- (h1);
\draw[->] (16.6,3.8) -- (14.9,2.6); %(s2) -- (d2);
\draw[red,thick,->] (12.2,1.5) -- (16.3,0.3); %(h1) -- (t2);
\draw[->] (14.9,1.5) -- (16.5,0.5); %(d2) -- (t2);
\end{tikzpicture}
\end{equation}
with perhaps some missing arrows and where the submodule $M=\im \ii^+_{(0,q);(2,-1)}$ has been identified with red arrows. The fact that there is no missing arrows is shown exactly like in subsection \ref{sub:iii} if $N>2$. When $N=2$, the precedent diagram becomes the one of the module $\Cell{2;0,q}$ (see proposition \ref{prop:StrucCellProb}). It must therefore be missing an arrow since $\Cell{2;0,q}\not\simeq \XXZ{2;0,q}{+}$ (as an easy computation shows that the generator $e_1\in \atl{N}$ acts as zero on $\Cell{2;0,q}$ but not of $\XXZ{2;0,q}{+}$). However, the only arrow that can be added while still producing a diagram compatible with the submodules and quotients found above is $(2,1)\rightarrow (2,-1)$. The Loewy diagram of $\XXZ{2;0,q}{+}$ must then be given by this arrow, that is
\begin{equation*}
\begin{tikzpicture}[baseline={(current bounding box.center)},scale=0.45]
\node (s0) at (0,0) [] {$(2,1)$}; %
\node (t0) at (4,0) [] {$(2,-1)$}; %
\draw[->] (s0) -- (t0);%
\end{tikzpicture}
\end{equation*}
The structure of the other problematic eigenspace $\XXZ{N;0,q^{-1}}{+} \simeq (\XXZ{N;0,q}{+})^{\star}$ can also be deduced from \eqref{fig:StrucXProb} with the help of proposition \ref{thm:isoByDual}. The resulting diagram for $N=2$ is given by the arrow $(2,-1)\rightarrow (2,1)$. For $N>2$, we rather obtain
\begin{equation}\label{fig:StrucXProb2}
\begin{tikzpicture}[baseline={(current bounding box.center)},scale=0.45]
\node (s0) at (3,4) [] {$(2,-1)$}; %
\node (t0) at (3,0) [] {$(2,1)$}; %
\node (h0) at (5,2) [] {$(4,-q)$}; %
\node (d1) at (8,2) [] {$(4,q)$}; %
\node (s1) at (10,4) [] {$(6,1)$}; %
\node (t1) at (10,0) [] {$(6,-1)$}; %
\node (h1) at (12,2) [] {$(8,q)$}; %
\node (d2) at (15,2) [] {$(8,-q)$}; %
\node (s2) at (17,4) [] {$\dots$};
\node (t2) at (17,0) [] {$\dots$};
\draw[->] (4,3.5) -- (5,2.6);% (s0) -- (h0);
\draw[->] (4.2,3.8) -- (7.7,2.6);% (s0) -- (d1);
\draw[red,thick,->] (5,1.5) -- (3.5,0.6); %(h0) -- (t0);
\draw[->] (7.8,1.5) -- (3.9,0.4); %(d1) -- (t0);
\draw[->] (8.9,3.8) -- (5.2,2.6); % (s1) -- (h0);
\draw[->] (8.9,3.6) -- (7.9,2.6); % (s1) -- (d1);
\draw[red,thick,->] (5.2,1.5) -- (8.9,0.4); %(h0) -- (t1);
\draw[->] (8,1.5) -- (9.8,0.6); %(d1) -- (t1);
\draw[->] (11,3.6) -- (12,2.6); %(s1) -- (h1);
\draw[->] (11,3.8) -- (14.8,2.6); %(s1) -- (d2);
\draw[red,thick,->] (12.1,1.5) -- (10,0.6); %(h1) -- (t1);
\draw[->] (14.8,1.45) -- (11.1,0.4); %(d2) -- (t1);
\draw[->] (16.4,3.9) -- (12.2,2.6); %(s2) -- (h1);
\draw[->] (16.6,3.8) -- (14.9,2.6); %(s2) -- (d2);
\draw[red,thick,->] (12.2,1.5) -- (16.3,0.3); %(h1) -- (t2);
\draw[->] (14.9,1.5) -- (16.5,0.5); %(d2) -- (t2);
\end{tikzpicture}
\end{equation}
where the image $\im \ii_{(0,-q);(2,1)}$ is again identified with red arrows (this image is the $\star$-dual of the quotient \eqref{eq:QuotProb} of $\XXZ{N;0,q}{+}$). The diagrams obtained are compatible with theorem \ref{thm:main} and we can thus conclude the proof.%fin
\end{proof}
We end this appendix by noting that a straightforward comparison between the diagram \eqref{fig:StrucXProb} (or the corresponding diagram for $N=2$) and the one given in proposition \ref{prop:StrucCellProb} shows that, for the problematic pair $(0,q)$, the image $\im \ay{N;0,q}$ is exactly the generic part $\pg{N;0,q} \simeq \Irre{N;2,-1}$. This is also true for the other problematic pair $(0,-q)$ as the generic part $\pg{N;0,-q}$ is then isomorphic to the simple module $\Irre{N;2,1}$ (see the proof of corollary \ref{thm:gpEstUnQuotient} for the problematic case) which is a submodule of $\XXZ{N;0,-q}{+}$ by diagram \eqref{fig:StrucXProb2} (or by the corresponding diagram for $N=2$). This is a special case of the aforementioned corollary \ref{cor:Imind}.

\end{subsection}
\end{section}

%%%%%%%%%%%
%
% many thanks
%
%%%%%%%%%%%

\section*{Acknowledgements}

We thank Jonathan Bellet\^ete and Alexi Morin-Duchesne for their interest in the project and their useful comments. We are particularly indebted to Jonathan Bellet\^ete for telling us the result generalizing lemma \ref{lem:jonathan} with its proof. 
%Y ajout
Useful discussions with Ibrahim Assem and Lo\"\i c Poulain d'Andecy are gratefully acknowledged. %Y fin
TP held scholarships from the Natural Sciences and Engineering Research Council of Canada and the Fonds de recherche Nature et technologies (Qu\'ebec) while this work was done. TP also received funding from the European Union's Horizon 2020 research and innovation programme under the Marie Sk\l{}odowska-Curie grant agreement \includegraphics[scale=0.08]{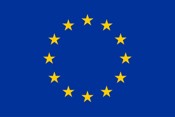} No 945322. YSA holds a grant from the Natural Sciences and Engineering Research Council of Canada. This support is gratefully acknowledged.

\raggedright
\singlespacing
\nocite{*} 
\bibliographystyle{unsrt}
\bibliography{xxz16}

\begin{thebibliography}{10}

\bibitem{PS}
V~Pasquier amd H~Saleur.
\newblock Common structures between finite systems and conformal field theories
  through quantum groups.
\newblock {\em Nucl.~Phys.~B}, 330:523--553, 1990.

\bibitem{GL-Aff}
J~Graham and G~Lehrer.
\newblock The representation theory of affine {Temperley}-{Lieb} algebras.
\newblock {\em Enseign. Math.}, 44:173--218, 1998.

\bibitem{MDSA}
A~Morin-Duchesne and Y~Saint-Aubin.
\newblock A homomorphism between link and {XXZ} modules over the periodic
  {Temperley}-{Lieb} algebra.
\newblock {\em J.~Phys.~A}, 46:285207, 2013.
\newblock \href{http://arxiv.org/abs/1203.4996}{\textsf{arXiv:1203.4996
  [math-ph]}}.

\bibitem{Deguchi}
T~Deguchi, K~Fabricius, and BM~McCoy.
\newblock The $sl_2$ loop algebra symmetry of the six-vertex model at roots of
  unity.
\newblock {\em J.~Stat.~Phys.}, 102:701--736, 2001.
\newblock
  \href{https://arxiv.org/abs/cond-mat/9912141}{\textsf{arXiv:cond-mat/9912141}}.

\bibitem{MDSAloop}
A~Morin-Duchesne and Y~Saint-Aubin.
\newblock Jordan cells of periodic loop models.
\newblock {\em J.~Phys.~A}, 46:494013, 2013.
\newblock \href{https://arxiv.org/abs/1302.5483}{\textsf{arXiv:1302.5483v3}}.

\bibitem{GRS1}
AM~Gainutdinov, N~Read, and H~Saleur.
\newblock Bimodule structure in the periodic $g\ell(1|1)$ spin chain.
\newblock {\em Nucl.~Phys.~B}, 871:289--329, 2013.
\newblock \href{http://arxiv.org/abs/1112.3407}{\textsf{arXiv:1112.3407v3}}.

\bibitem{GRS2}
AM~Gainutdinov, N~Read, and H~Saleur.
\newblock Associative algebraic approach to logarithmic {CFT} in the bulk: the
  continuum limit of the $g\ell(1|1)$ periodic spin chain, {H}owe duality and
  the interchiral algebra.
\newblock {\em Comm.~Math.~Phys.}, 341:35--103, 2016.
\newblock \href{https://arxiv.org/abs/1207.6334}{\textsf{arXiv:1207.6334v2}}.

\bibitem{koo1994representations}
WM~Koo and H~Saleur.
\newblock Representations of the {V}irasoro algebra from lattice models.
\newblock {\em Nuclear Physics B}, 426(3):459--504, 1994.

\bibitem{Jantzen}
JC~Jantzen.
\newblock {\em Lectures on quantum groups}, volume~6 of {\em Graduate Studies
  in Mathematics}.
\newblock American Mathematical Society, Providence, RI, 1995.

\bibitem{ChariPressley}
V~Chari and A~Pressley.
\newblock {\em A guide to quantum groups}.
\newblock Cambridge University Press, Cambridge, UK, 1995.

\bibitem{Klimyk}
A~Klimyk and K~Schm\"udgen.
\newblock {\em Quantum groups and their representations}.
\newblock Texts and monographs in physics. Springer, Berlin, Germany, 1997.

\bibitem{AndersenTubbenhauer}
HH~Andersen and D~Tubbenhauer.
\newblock Diagram categories for ${U}_q$-tilting modules at roots of unity.
\newblock {\em Transform.~Groups}, 1:29--89, 2017.
\newblock \href{https://arxiv.org/abs/1409.2799}{\textsf{arXiv:1409.2799v2}}.

\bibitem{BSAfusion}
J~Bellet\^{e}te and Y~Saint-Aubin.
\newblock On the computation of fusion over the affine {Temperley}-{Lieb}
  algebra.
\newblock {\em Nucl.~Phys.~B}, 937:333--370, 2018.
\newblock \href{http://arxiv.org/abs/1802.03575}{\textsf{arXiv:1802.03575}}.

\bibitem{MS}
P~Martin and H~Saleur.
\newblock On an algebraic approach to higher dimensional statistical mechanics.
\newblock {\em Comm.~Math.~Phys.}, 158:1555--190, 1993.
\newblock
  \href{http://arxiv.org/abs/hep-th/9208061}{\textsf{arXiv:hep-th/9208061}}.

\bibitem{ridout2014standard}
David Ridout and Yvan Saint-Aubin.
\newblock Standard modules, induction and the structure of the temperley-lieb
  algebra.
\newblock {\em Advances in Theoretical and Mathematical Physics},
  18(5):957--1041, 2014.

\bibitem{Green}
RM~Green.
\newblock On representations of affine {Temperley}-{Lieb} algebras.
\newblock In I~Reiten, SO~Smal\o{}, and \O~Solberg, editors, {\em Conference
  Proceedings, Canadian Mathematical Society}, volume~24. American Mathematical
  Society, 1998.

\bibitem{AGR}
FC~Alcaraz, U~Grimm, and V~Rittenberg.
\newblock The {XXZ} heisenberg chain, conformal invariance and the operator
  content of $c<1$ systems.
\newblock {\em Nucl.~Phys.~B}, 316:735--768, 1989.

\bibitem{Lentner}
S~Lentner.
\newblock The unrolled quantum group inside lusztig’s quantum group of
  divided powers.
\newblock {\em Letters in Mathematical Physics}, 109(7):1665--1682, 2019.

\bibitem{AndersenPolo}
HH~Andersen, P~Polo, and Kexin Wen.
\newblock Representations of quantum algebras.
\newblock {\em Inventiones mathematicae}, 120(1):409--410, 1995.

\bibitem{bushlanov2009lusztig}
PV~Bushlanov, BL~Feigin, AM~Gainutdinov, and IY~Tipunin.
\newblock Lusztig limit of quantum sl(2) at root of unity and fusion of (1, p)
  virasoro logarithmic minimal models.
\newblock {\em Nuclear Physics B}, 818(3):179--195, 2009.

\bibitem{bushlanov2012lusztig}
PV~Bushlanov, AM~Gainutdinov, and IY~Tipunin.
\newblock Kazhdan-{L}usztig equivalence and fusion of {K}ac modules in
  {V}irasoro logarithmic models.
\newblock {\em Nuclear Physics B}, 862(3):232--269, 2012.
\newblock \href{http://arxiv.org/abs/1102.0271}{\textsf{arXiv:1102.0271v2
  [hep-th]}}.

\bibitem{Jimbo1}
M~Jimbo.
\newblock A $q$-difference analogue of ${U}(\frak g)$ and the {Y}ang-{B}axter
  equation.
\newblock {\em Lett.~Math.~Phys.}, 10:63--69, 1985.

\bibitem{Jimbo2}
M~Jimbo.
\newblock A $q$-analogue of ${U}(\frak{gl}({N}+1))$, {H}ecke algebra, and the
  {Y}ang-{B}axter equation.
\newblock {\em Lett.~Math.~Phys.}, 11:247--252, 1986.

\bibitem{MartinSchur}
PP~Martin.
\newblock On {S}chur-{W}eyl duality, ${A}_n$ {H}ecke algebras and quantum
  $sl({N})$ on $\otimes^{n+1}\mathbb{C}^{N}$.
\newblock {\em Int.~Jour.~Modern Phys.}, 07:645--673, 1992.

\bibitem{JonathanEtCie}
J~Bellet\^{e}te, AM~Gainutdinov, JL~Jacobsen, H~Saleur, and TS~Tavares.
\newblock Topological defects in the affine {T}emperley-{L}ieb algebra: the
  critical cases.
\newblock {\em in preparation}, 2022.

\bibitem{martin1994blob}
P~Martin and H~Saleur.
\newblock The blob algebra and the periodic {T}emperley-{L}ieb algebra.
\newblock {\em Letters in mathematical physics}, 30(3):189--206, 1994.

\bibitem{GLdiagram}
J~Graham and G~Lehrer.
\newblock Diagram algebras, hecke algebras and decomposition numbers at toors
  of unity.
\newblock {\em Ann.~Scient.~\'Ec.~Norm.~Sup., 4e s\'erie}, 36:479--524, 2003.

\bibitem{morin2015boundary}
A~Morin-Duchesne, J~Rasmussen, and D~Ridout.
\newblock Boundary algebras and {Kac} modules for logarithmic minimal models.
\newblock {\em Nuclear Physics B}, 899:677--769, 2015.

\bibitem{GLlienBlob}
J~Graham and G~Lehrer.
\newblock Cellular algebras and diagram algebras in representation theory.
\newblock In {\em Representation theory of algebraic groups and quantum
  groups}, pages 141--173, 2004.

\bibitem{BGJS2018}
J~Bellet\^ete, AM~Gainutdinov, JL~Jacobsen, H~Saleur, and TS~Tavares.
\newblock Topological defects in lattice models and affine {T}emperley-{L}ieb
  algebras.
\newblock 2018.
\newblock \href{https://arxiv.org/abs/1811.02551}{\textsf{arXiv:1811.02551}}.

\bibitem{feigin1982invariant}
BL~Feigin and DB~Fuks.
\newblock Invariant skew-symmetric differential operators on the line and
  {V}erma modules over the virasoro algebra.
\newblock {\em Functional Analysis and Its Applications}, 16(2):114--126, 1982.

\bibitem{StumQuiros}
B~Le Stum and A~Quir\'os.
\newblock On quantum integers and rationals.
\newblock In F~Chamizo, J~Gu\`ardia, A~Rojas-Le\'on, and JM~Tornero, editors,
  {\em Trends in number theory}, volume 649 of {\em Contemporary Mathematics},
  pages 107--131. American Mathematical Society, Providence, RI, 2015.
\newblock \href{https://arxiv.org/abs/1310.8143}{\textsf{arXiv:1310.8143v1}}.

\bibitem{Bcorings}
T~Brzezinski and R~Wisbauer.
\newblock {\em Corings and comodules}.
\newblock Cambridge University Press, 2003.

\bibitem{saleur1989virasoro}
H~Saleur.
\newblock Virasoro and {T}emperley-{L}ieb algebras.
\newblock In {\em Knots, Topology and QFT, Proceedings of the Johns Hopkins
  Workshop, Florence}, pages 485--496. World Scientific, 1989.

\bibitem{GoodmanWenzl}
F~Goodman and H~Wenzl.
\newblock The {T}emperley-{L}ieb algebra at roots of unity.
\newblock {\em Pac.~J.~Math.}, 161:307--334, 1993.

\bibitem{MartinBook}
P~Martin.
\newblock {\em {Potts} Models and Related Problems in Statistical Mechanics},
  volume~5 of {\em Advances in Statistical Mechanics}.
\newblock World Scientific, Singapore, 1991.

\bibitem{gainutdinov2013physical}
AM~Gainutdinov, JL~Jacobsen, H~Saleur, and R~Vasseur.
\newblock A physical approach to the classification of indecomposable
  {V}irasoro representations from the blob algebra.
\newblock {\em Nuclear Physics B}, 873(3):614--681, 2013.

\end{thebibliography}
\end{document}